\documentclass[reqno, 11pt]{amsart}
\usepackage{amsmath}	
\usepackage{amssymb}	
\usepackage{amsthm}
\usepackage{amsfonts}
\usepackage{latexsym}
\usepackage{mathrsfs}
\usepackage[all,ps,cmtip]{xy}
\usepackage{tabularx}
\usepackage{pict2e}
\usepackage{nicematrix}
\usepackage{graphicx}
\usepackage{xypic}
\usepackage{tikz-cd}
\usepackage{tikz}
\usetikzlibrary{matrix,arrows}
\usepackage[colorlinks=true, citecolor=blue, urlcolor=blue, linkcolor=blue]{hyperref}
\usepackage[text={145mm, 228mm},centering]{geometry}
\input diagxy
\input xy

\theoremstyle{plain} 
\newtheorem{thm}{Theorem}[section]
\newtheorem{lem}[thm]{Lemma} 
\newtheorem{prop}[thm]{Proposition} 
\newtheorem{cor}[thm]{Corollary} 

\theoremstyle{definition} 
\newtheorem{defn}[thm]{Definition}
\newtheorem{rem}[thm]{Remark} 
\newtheorem{example}[thm]{Example}
\newtheorem{quest}[thm]{Question}
\newtheorem{conj}[thm]{Conjecture}
\newtheorem{alg}[thm]{Algorithm} 
\newtheorem{set}[thm]{Setup} 

\DeclareMathOperator{\Z}{\mathbb{Z}}

\DeclareMathOperator{\CO}{\mathcal{O}}
\DeclareMathOperator{\CC}{\mathcal{C}}
\DeclareMathOperator{\CN}{\mathbb{C}}

\DeclareMathOperator{\CD}{\mathcal{D}}
\DeclareMathOperator{\D}{\mathrm{D}}

\DeclareMathOperator{\sign}{\mathrm{sign}}

\DeclareMathOperator{\Hom}{\mathrm{Hom}}
\DeclareMathOperator{\disc}{\mathrm{disc}}

\DeclareMathOperator{\Id}{\mathrm{id}}

\allowdisplaybreaks

\numberwithin{equation}{section}

\begin{document}

\title{On lattice polarizable cubic fourfolds}

\author{Song Yang \and Xun Yu}
\address{Center for Applied Mathematics, Tianjin University, Weijin Road 92, Tianjin 300072, P.R. China}%
\email{syangmath@tju.edu.cn, xunyu@tju.edu.cn}%

\begin{abstract}
We extend non-emtpyness and irreducibility of Hassett divisors to the moduli spaces of $M$-polarizable cubic fourfolds for higher rank lattices $M$, which in turn provides a systematic approach for describing the irreducible components of intersection of Hassett divisors. 
We show that Fermat cubic fourfold is contained in every Hassett divisor, 
which yields a new proof of Hassett's existence theorem of special cubic fourfolds. 
We obtain an algorithm to determine the irreducible components of the intersection of any two Hassett divisors and we give new examples of rational cubic fourfolds. 
Moreover, we derive a numerical criterion for the algebraic cohomology of a cubic fourfold having an associated K3 surface and answer a question of Laza by realizing infinitely many rank $11$ lattices as the algebraic cohomologies of cubic fourfolds having no associated K3 surfaces.
\end{abstract}

\maketitle

\setcounter{tocdepth}{1}
\tableofcontents


\section{Introduction}

Cubic fourfolds have received considerable attention in the last few decades, 
both because of the importance of their rationality problem, 
and because they are closely related to K3 surfaces. 
Let $X$ be a cubic fourfold, that is, a smooth hypersurface of degree three in complex projective five-space. 
Its algebraic cohomology $A(X):=H^{4}(X, \Z)\cap H^{2,2}(X)$ is a positive definite lattice containing the square $h_X^2$ of the hyperplane class.
By Beauville--Donagi \cite{BD85} and Hassett \cite{Has00}, 
the orthogonal complement $\langle h_X^2 \rangle^{\perp} \subset A(X)$ is an even lattice. 
Moreover, thanks to Voisin \cite[\S4, Proposition 1]{Voi86} and Hassett \cite{Has00}, 
$A(X)$ has no roots (i.e., $\nexists\; v\in A(X)$ with the intersection number $(v.v)=2$); see also Proposition \ref{pre-version} and Lemma \ref{non-empty-criterion}.
In the seminal work \cite{Has00}, Hassett introduced the notion of a {\it special cubic fourfold (of discriminant $d$)}, i.e., a cubic fourfold $X$ whose $A(X)$ has a rank $2$ primitive sublattice containing $h_X^2$ (of discriminant $d$). 
Recall that the moduli space $\CC$ of smooth cubic fourfolds is a quasi-projective variety of dimension $20$. 
Building on Voisin's Torelli theorem for cubic fourfolds \cite{Voi86},
Hassett \cite[Theorem 1.0.1]{Has00} proved that the moduli space $\mathcal{C}_{d}$ of the special cubic fourfolds of discriminant $d$ is an irreducible divisor of $\CC$ and $\mathcal{C}_{d}$ is non-empty if and only if 
\begin{equation*}\tag{$*$}\label{onestar}
d>6 \;\;\; \textrm{and}\;\;\; d\equiv 0, 2\; (\mathrm{mod}\; 6).
\end{equation*}
{\it Hassett divisors} (i.e., non-empty $\CC_d$) have played fundamental roles in many studies of cubic fourfolds. 
Our first main result is  the following generalization of Hassett's theorem. 

\begin{thm}[see Theorem \ref{mainthm-one-equal}]\label{mainthm-one}
Let $M$ be a positive definite lattice of rank $r(M)\geq 2$. 
We denote by $\CC_{M}\subset \CC$ the cubic fourfolds $X$ such that $A(X)$ has a primitive sublattice containing $h_X^2$ which is isometric to $M$. If the following conditions hold
\begin{enumerate}
\item $\exists\; \mathfrak{o}\in M$ such that $(\mathfrak{o}.\mathfrak{o})=3$ and the sublattice $\langle \mathfrak{o}\rangle^\perp\subset M$ is even;
\item $M$ has no roots; and
\item $r(M)+\ell(M)\leq 20$,
\end{enumerate}
then $\CC_M$ is a non-empty irreducible closed subvariety of codimension $r(M)-1$ in $\CC$.
\end{thm}

Here $\ell(M)$ denotes the minimal number of generators of the finite abelian group $M^\vee/M$ (Section \ref{sec:lattice}). 
Note that the condition (3) is automatic if $r(M)\le 10$ (Theorem \ref{prop:main}). 
Generally speaking, 
the higher the rank of $M$ is, 
the richer geometric structure an {\it $M$-polarizable cubic fourfold} (i.e., a cubic fourfold in $\CC_M$) has.  
Many authors studied $M$-polarizable cubics and their moduli for concrete lattices $M$ with $r(M)\ge 3$ from various points of view (see \cite{Tre84,Has99,AT14,ABBVA14,LPZ18,Laz18,AHTVA19,BRS19,DM19,Aue20} etc.). Our Theorem \ref{mainthm-one} gives a unified explanation for non-emptyness and irreducibility of $\CC_M$ for those explicit $M$. 
Based on \cite{Voi86,Has00,Laz10,Loo09},
the non-emptyness of $\CC_M$ has been studied for $M$ with a prescribed primitive embedding into $H^4(X,\Z)$ in \cite[\S 2]{Has16} and \cite[Proposition 5]{YY20} (cf. Proposition \ref{pre-version} and Lemma \ref{non-empty-criterion}).
However, the irreducibility of $\CC_M$ was not investigated for general lattices $M$ with $r(M)\ge 3$. 
Let us briefly explain the proof of Theorem \ref{mainthm-one}. As in \cite{Has00}, Theorem \ref{mainthm-one} boils down to highly nontrivial problems in lattice theory (i.e., existence and uniqueness of certain primitive embeddings) via the Torelli theorem for cubic fourfolds. 
We solve those problems based on Nikulin \cite{Nik80} and technique of glue (see Propositions \ref{exist-primi-embed}, \ref{unique-primi-embed}). 
Along the way, 
we obtain a stronger result (Corollary \ref{non-empty-CM}) on non-emptyness of $\CC_M$ replacing (3) by $r(M)+\ell(M)\leq 22$ and $r(M)\le 21$.

The rationality of cubic fourfolds is a long standing open problem; 
we refer to the excellent surveys \cite{Has16,Huy19}. 
In 1972, Clemens–Griffiths \cite{CG72} proved that all smooth cubic threefolds are irrational via Hodge theory. 
It is expected that a very general cubic fourfold is irrational; nevertheless, no such example is known. 
Cubic fourfolds are linked to K3 surfaces via Hodge theory, due to Hassett's work \cite{Has00}, 
and via derived categories, due to Kuznetsov's work \cite{Kuz10}. 
By \cite{AT14,Huy17,BLMNPS}, these two viewpoints identify the same set of cubic fourfolds. 
Such cubics are called {\it having associated K3 surfaces} and are parametrized exactly by the countably infinite union of the Hassett divisors $\CC_d$ with 
\begin{equation}\tag{$**$}\label{twostar}
4 \nmid d,\; 9\nmid d,\; p \nmid d \text{ for any odd prime } p \equiv 2\;  (\mathrm{mod}\, 3).
\end{equation}
Kuznetsov \cite{Kuz10} conjectured that a cubic fourfold is rational if and only if it has an associated K3 surface. 
The first five values of $d$ with the properties \eqref{onestar} and \eqref{twostar} are $14$, $26$, $38$, $42$ and $62$. 
To the best of our knowledge, currently known examples of rational cubic fourfolds are parametrized by:  
(i) a countably infinite union of divisors in $\CC_{8}$ (\cite{Has99}); 
(ii) $\CC_{14}$ (\cite{Fan43,Tre84,BD85,Tre93,BRS19,Aue20}); 
(iii) a countably infinite union of divisors in $\CC_{18}$ (\cite{AHTVA19}); 
(iv) three divisors in $\CC_{20}$ (\cite{FL20});
(v) $\CC_{26}$, $\CC_{38}$, $\CC_{42}$ (\cite{RS19a,RS18,RS19b}).
Kuznetsov conjecture holds for these examples. 
In fact, $\CC_{8}$ (resp. $\CC_{18}$, resp. $\CC_{20}$) does not satisfy \eqref{twostar} and cubics in (i) (resp. (iii), resp. (iv)) are contained in intersection between $\CC_{8}$ (resp. $\CC_{18}$, resp. $\CC_{20}$) and some Hassett divisors $\CC_d$ with $d$ satisfying \eqref{twostar} (cf. Remarks \ref{rem:C8rationalcubics}, \ref{rem:C18rationalcubics}). 

From discussions above, it is significant to study intersection of Hassett divisors for, among other things, understanding the (conjectural) locus of rational cubics inside $\CC$ and finding new examples of rational cubics.  The first issue for such intersection is non-emptyness. As already mentioned above, for some relatively small values $d$, $\CC_d$ may intersect (infinitely) many other Hassett divisors. 
Addington--Thomas \cite[Theorem 4.1]{AT14} showed that for any Hassett divisor $\CC_d$,  there exists a $\CC_{d^\prime}$ with $d^\prime$ satisfying \eqref{onestar} and \eqref{twostar} such that the intersection $\CC_8\cap \CC_d\cap\CC_{d^\prime}$ is non-empty. 
In \cite[Theorem 7]{YY20} we proved that the intersection of any two Hassett divisors is non-empty. It is known that there exists a collection of infinitely many Hassett divisors whose intersection is non-empty; however, an explicit description of such a collection is unknown. Now we obtain the following result, which solves the non-emptyness issue completely.

\begin{thm}\label{mainthm-two}
The intersection $\mathcal{Z}$ of all Hassett divisors is non-empty and contains the Fermat cubic fourfold.
\end{thm}

The proof of Theorem \ref{mainthm-two} is based on explicit description of the algebraic cohomology of the Fermat cubic fourfold.
More precisely, we show that it is generated by the classes of $21$ planes (Proposition \ref{algcohom-Fermat}) and combining with two classical results in number theory, 
Lagrange's four-square theorem and an analogous result of Ramanujan \cite{Ram}, we get the result. 
As a byproduct, this yields a new proof of Hassett's existence theorem \cite[Theorem 4.3.1]{Has00} of special cubic fourfolds. 
Moreover, Theorem \ref{mainthm-two} provides an explicit example of cubic for each Hassett divisor.  
Note that cubics in $\mathcal{Z}$ are rational and the dimension of $\mathcal{Z}$ is at least $13$ (Theorem \ref{subvar-Inter-all-HasD}). 
Another crucial issue regarding intersection is the determination of irreducible components. 
The description of the irreducible components of $\CC_{d_1}\cap \CC_{d_2}$ with relatively small $d_i$  has been studied extensively in \cite{ABBVA14,BRS19,Awa19,FL20} etc. 
Our Theorem \ref{mainthm-one} provides a systematic approach for describing the irreducible components of intersection of arbitrarily many Hassett divisors (see Proposition \ref{prop:int}). 
In order to find {\bf all} irreducible components,
the overlattices must be taken into account (cf. Theorem \ref{C20intC12}).
In particular, we find an algorithm to compute the irreducible components of intersection of any two Hassett divisors (Algorithm \ref{alg:int}). 
As an illustration, we determine the irreducible components of $\CC_{20}\cap\CC_{38}$ and $\CC_{20}\cap\CC_{12}$ respectively (Theorems \ref{C20intC38}, \ref{C20intC12}). 
Building on our complete description of the irreducible components of $\CC_{20}\cap \CC_{38}$, 
we find new examples of rational cubics: 

\begin{cor}[see Corollary \ref{new-rat-cubics}]
There are two rank $3$ positive definite lattices $M_1$ and $M_2$ of discriminants $197$ and $213$ respectively such that $\CC_{M_i}$ ($i=1,2$) are non-empty irreducible divisors of $\CC_{20}$. Moreover, $\CC_{M_i}$ parametrize rational cubic fourfolds and are not contained in $\CC_{8}$, $\CC_{14}$, $\CC_{18}$, $\CC_{26}$, $\CC_{38}$, 
$\CC_{42}$.
\end{cor}

Note that $\CC_{M_1}\subset \mathcal{C}_{98}$ and $\CC_{M_2} \subset \mathcal{C}_{206}$, 
where $98$ and $206$ satisfy \eqref{twostar}.
In general, we introduce the notion of admissible lattices as follows.
Let $M$ be a positive definite lattice of rank $r(M)\ge 2$ containing a {\it distinguished element} $\mathfrak{o}$ (i.e., the condition $(1)$ in Theorem \ref{mainthm-one} holds for $\mathfrak{o}\in M$).
We say that $M$ is {\it admissible} if $M$ has a rank $2$ primitive sublattice $K$, $\mathfrak{o}\in K$ such that the discriminant of $K$ satisfies \eqref{twostar} (see Definition \ref{def:admissiblelattice}).
In particular, for a cubic fourfold $X$, the algebraic cohomology $A(X)$ is an admissible lattice if and only if $\CC_{A(X)}\subset \mathcal{C}_{d}$ for some $d$ satisfying \eqref{twostar}.
As already mentioned above, admissibility of the algebraic cohomology of a cubic fourfold is conjecturally equivalent to the rationality. 
It is of importance to give characterizations of admissibility via lattice theory. 
By \cite[Theorem 1.0.2]{Has00}, admissibility is characterized by the numerical condition \eqref{twostar} on discriminant if the rank $r(A(X))$ of $A(X)$ is $2$. 
In \cite{AT14} (see also \cite{Add16,Huy17}) it was shown that $A(X)$ is admissible if and only if a closely related lattice of signature $(r(A(X))-1,2)$ contains the rank $2$ even hyperbolic lattice $U$, 
or equivalently, the transcendental lattice $T(X)$ of $X$ admits a primitive embedding into the even unimodular lattice of signature $(19,3)$.
Laza \cite[Proposition 2.5, Remark 2.6]{Laz18} pointed out that admissibility can often be determined by comparing two numerical invariants $\ell(A(X))$ and $r(A(X))$ but the exceptional case $\ell(A(X))=r(A(X))-1$ is delicate. 
Using a number theoretical result (Proposition \ref{prop:representprime}),
we derive a new characterization for admissibility of general lattices in terms of some numerical conditions. 
More precisely, we have the following numerical criterion:

\begin{thm}[{see Theorems \ref{thm:admM1}, \ref{thm:admM2}}]\label{mainthm-intro-3}
Let $M$ be a positive definite lattice of rank $r(M)\ge 2$ containing a distinguished element $\mathfrak{o}$, and let $\mathbf{b}$ be a basis of $M$ containing $\mathfrak{o}$.
Suppose that $f$ is the associated form of $(M,\mathbf{b})$ (see Definition \ref{associated-form}). 
We write 
$$
f=\lambda g,
$$
where $\lambda$ is the greatest common divisor of all the coefficients of $f$ and $g$ is a primitive integral quadratic form. 
Then the following two statements hold:
\begin{enumerate}
\item If $M\neq \langle \mathfrak{o} \rangle\oplus \langle \mathfrak{o} \rangle_{M}^{\perp}$,
         then $M$ is admissible if and only if $\lambda$ satisfies \eqref{twostar};
\item If $M= \langle \mathfrak{o} \rangle\oplus \langle \mathfrak{o} \rangle_{M}^{\perp}$,
         then $M$ is admissible if and only if $\lambda$ satisfies \eqref{twostar} and $g$ represents an integer $c\equiv\, 1\;  (\mathrm{mod}\, 3)$.
\end{enumerate}
\end{thm}

As one of the key features, our criterion is quite effective for explicit lattices (for rank $3$ case, see Propositions \ref{prop:admMrank3-1}, \ref{prop:admMrank3-2} and Corollaries \ref{cor:C8divisors}, \ref{cor:C18divisors}). 
Laza \cite{Laz18} observed that the rank of a nonadmissible lattice $A(X)$ is at most $11$ and asked the existence of such a lattice with rank $11$. 
Recently Dolgachev--Markushevich \cite[Remark 7.4]{DM19} found an example by considering cubics containing mutually intersecting $10$ planes. 
As far as we know,  this is the only known example answering Laza's question. 
Based on Corollary \ref{non-empty-CM} and our criterion of admissibility, we realize infinitely many rank $11$ lattices as the algebraic cohomologies of cubics having no associated K3 surfaces (Corollary \ref{cor:anwserLaza}).

The notion of $M$-polarized cubic fourfolds (see Definition \ref{def:Mpolarizable}) is analogous to lattice polarized K3 surfaces defined by Dolgachev \cite{Dol96}.   Important connection between moduli of $M$-polarized cubics for admissible $M$ and moduli of lattice polarized K3 surfaces goes back to Hassett's work \cite{Has00}. 
We will not touch this very interesting topic in this paper.  
However, our results (non-emptyness and irreducibility of $\CC_M$, the new admissibility criterion, etc) should shed light on further study of the theory of lattice polarized cubics and K3 surfaces.

\subsection*{Acknowledgements}
This work is partially supported by the National Natural Science Foundation of China (Grant No. 11701414, No. 12071337, No. 11701413, No. 11831013).


\section{Lattices and glue}\label{sec:lattice}

In this section, we gather a series of basic facts on lattices, glue and extensions.
We refer to \cite{Ser73,Nik80,McM11,McM16,OY20} for more detailed discussions.

\subsection{Lattices}
\begin{defn}
A {\it lattice} is a finite rank free $\Z$-module $L$ together with a symmetric and non-degenerate bilinear form 
$$
(-.-)_{L}: L \times L\longrightarrow \mathbb{Z}.
$$
Let $L$ be a lattice and $m\in \Z$.
$L(m)$ denotes the same free $\Z$-module with the bilinear form $(-.-)_{L(m)}:=m(-.-)_{L}$.
Let $L_{1}$ and $L_{2}$ be two lattices.
$L_{1}\oplus L_{2}$ denotes the orthogonal direct sum of $L_{1}$ and $L_{2}$, 
namely, the bilinear form $(x+y. x'+y')_{L_{1}\oplus L_{2}}:=(x.x')_{L_{1}}+(y.y')_{L_{1}}$.
For the bilinear form, we will drop the subscript if without causing any confusion. 
\end{defn}

Let $(e_1,e_2,...,e_{r(L)})$ be a basis of a lattice $L$ of rank $r(L)$. 
We call $((e_i.e_j))_{1\leq i, j \leq r(L)}$ the {\it Gram matrix} of $L$ with respect to $(e_1,e_2,...,e_{r(L)})$. 
The {\it discriminant} of  a lattice $L$, denote by $\disc(L)$,
is the determinant of the Gram matrix with respect to an arbitrary basis of $L$.
A lattice $L$ is called {\it unimodular} if its discriminant $\disc(L)=\pm 1$.
An element $x\in L$ is called a {\it root} if $(x.x)=2$.
A lattice $L$ is called {\it even} if $(x.x)$ is even for all $x\in L$; 
otherwise, it is called {\it odd}.
The bilinear form $(-.-)$ on $L$ induces a symmetric and non-degenerate bilinear form on $L\otimes_{\Z} \mathbb{R}$. 
The latter intersection matrix can be diagonalized with only $\pm1$ on the diagonal.
The {\it signature} of $L$ is $\sign(L):=(s_{+},s_{-})$, 
where $s_{\pm}$ is the number of $\pm1$ on the diagonal.
So the rank $r(L)=s_{+}+s_{-}$.
A lattice $L$ is called {\it definite} if $r(L)=s_{\pm}$; 
otherwise, $L$ is called {\it indefinite}.

\begin{defn}
Let $L$ and $L^{'}$ be two lattices.
An {\it isometry} is a homomorphism $f: L\longrightarrow L^{'}$ 
such that $f$ preserves the bilinear forms.
We denote by $L\cong L^{'}$ if there is an isometry between them,
and $\mathrm{O}(L)$ the {\it orthogonal group} consisting of all isometries $f: L\longrightarrow L$.
\end{defn}

A sublattice $N$ of $L$ is called {\it primitive} if $L/N$ is torsion-free.
Let $N\subset L$ be a sulattice. The {\it primitive closure }$\overline{N}$ of $N$ in $L$ is the smallest primitive sublattice of $L$ containing $N$.
We denote by $N^{\perp}_L$ the orthogonal complement of $N$ in $L$, and we sometimes omit the subscript $L$ if there is no confusion. 
For a subset $S$ of $L$, 
we denote by $\langle S \rangle\subset L$ the sublattice generated by $S$.

Next, we collect some interesting examples which we will need in the subsequent sections.

\begin{example}
$(i)$\, $U$ denotes the hyperbolic plane with Gram matrix
$\small{\left(\begin{array}{cc} 
0 & 1 \\
1 & 0 
\end{array} \right)}$.

$(ii)$\, $A_{2}$ denotes the lattice with Gram matrix
$
\small{\left(\begin{array}{cc} 
2 & 1 \\
1 & 2 
\end{array} \right)}
$. 

$(iii)$\,$E_{8}$ denotes the unique unimodular positive definite even lattice of rank $8$. 

$(iv)$ For $s>0$, $t>0$ (resp. $s>0$, $t>0$, $s\equiv r\; \mathrm{mod}\;8$), there is a unique odd (resp. even) indefinite unimodular lattice $\mathrm{I}_{s, t}$ (resp. $\mathrm{II}_{s, t}$) of signature $(s, t)$ by a theorem of Milnor (cf. Serre \cite[Chapter V]{Ser73}).
\end{example}

\subsection{Glue}

\begin{defn}\label{def:quadratic}
A {\it quadratic form} on a finite abelian group $G$ is a map
$$
q_{G}: G\longrightarrow \mathbb{Q}/ 2\Z
$$
with a symmetric bilinear form
$
b_{G}: G\times G \longrightarrow \mathbb{Q}/ \Z
$
satisfying the following conditions:
\begin{enumerate}
\item[(1)] $q_{G}(nx)=n^{2}q_{G}(x)$ for all $n\in \Z$ and $x\in G$; and
\item[(2)] $q_{G}(x+y)-q_{G}(y)-q_{G}(y)\equiv 2b_{G}(x,y)\; (\mathrm{mod}\; 2\Z)$ for any $x, y\in G$. 
\end{enumerate}
If $q_{G}$ is a quadratic form on a finite abelian group $G$,
so is $-q_{G}$ with symmetric bilinear form $-b_{G}$.
\end{defn}

Give a lattice $L$ and we denote its dual 
$$
L^{\vee}:=\Hom_{\Z}(L, \mathbb{Z})\cong \{x\in L\otimes \mathbb{Q} \mid (x. L) \subset \Z\}.
$$ 
The non-degenerate bilinear form $(-.-)$ on a lattice $L$ determines natural inclusions
$$
L\subset L^{\vee} \subset L\otimes \mathbb{Q}
$$

\begin{defn}
The {\it glue group} (or {\it discriminant group}) of a lattice $L$,
$$
G(L):=L^{\vee}/ L,
$$
is a finite abelian group of order $|G(L)|=|\disc(L)|$.
We denote by $\ell(L)$ the minimal number of generators of the glue group $G(L)$.
\end{defn}

For any element $x\in L^{\vee}$, 
$\bar{x}$ denotes the image of $x$ in $G(L)$ under the projection $L^{\vee}\rightarrow L^{\vee}/ L$. 
The symmetric bilinear form $(-.-)$ on $L$ extends to be a $\mathbb{Q}$-valued bilinear form on $L^{\vee}$ which induces a symmetric bilinear form $b_{G(L)}$ on $G(L)$:
$$
\begin{array}{cccl}
b_{G(L)}:&G(L)\times G(L)&\longrightarrow& \mathbb{Q}/\Z \\ 
&(\bar{x}, \bar{y}) &\longmapsto& b_{G(L)}(\bar{x}, \bar{y})\equiv (x.y) \;\mathrm{mod}\; \Z.
\end{array}
$$
If $L$ is even, we call the quadratic form 
$$
\begin{array}{cccl}
q_{G(L)}:&G(L) &\longrightarrow& \mathbb{Q}/ 2\Z \\ 
&\bar{x}&\longmapsto& q_{G(L)}(\bar{x})\equiv (x.x) \;\mathrm{mod}\; 2\Z
\end{array}
$$
the {\it discriminant form} of $L$. 
For any prime number $p$, 
$q_{G(L)_{p}}$ denotes the restriction of $q_{G(L)}$ on the Sylow $p$-subgroup $G(L)_{p}$.

\begin{example}\label{A_2-example}
Consider the lattice
$A_{2}(-1)=\langle  e_{1}, e_{2}\rangle $ with Gram matrix 
$\small{\left(\begin{array}{cc} 
-2 & -1 \\
-1 & -2 
\end{array} \right)}$.
Then its glue group $G(A_{2}(-1))\cong \Z/3Z$ is generated by $\overline{\frac{-2e_{1}+e_{2}}{3}}$.
Moreover,
$$
q_{G(A_{2}(-1))}(\overline{\frac{-2e_{1}+e_{2}}{3}})\equiv \frac{4}{3} \;\mathrm{mod}\; 2\Z
$$
and
$$
b_{G(A_{2}(-1))}(\overline{\frac{-2e_{1}+e_{2}}{3}}, \overline{\frac{-2e_{1}+e_{2}}{3}})\equiv \frac{1}{3}\; \mathrm{mod}\, \Z.
$$
\end{example}

Let $L_{1}$ and $L_{2}$ be two lattices (resp. two even lattices) and let $\psi:G(L_1)\longrightarrow G(L_2)$ be an isomorphism of abelian groups. 
If $\psi$ preserves the bilinear forms $b_{G(L_i)}$ (resp. the quadratic forms $q_{G(L_i)}$), we say that $\psi$ is an isometry (resp. an isomorphism of quadratic forms). 
Any isometry $f: L_1\longrightarrow L_2$ of lattices (resp. even lattices) induces an isomorphism $\bar{f}: G(L_1)\longrightarrow G(L_2)$ of abelian groups which is an isometry (resp. isomorphism of quadratic forms). For an even lattice $L$, we denote by ${\rm O}(q_{G(L)})$ the group of automorphisms of $q_{G(L)}$.

\subsection{Extensions}
The glue group plays a fundamental role in classifying extensions of a lattice.
Let $L$ be a lattice.
An {\it overlattice} of $L$ is a lattice $M$ such that $L\subset M \subset L^{\vee}$ and $M/L$ is a finite abelian group.
Note that the subgroup $H_{M}:=M/L \subset G(L)$ is isotropic.
Moreover, the set of overlattices $L\subset M \subset L^{\vee}$ is bijectively corresponding to the set of isotropic subgroups $0\subset H_{M} \subset G(L)$.
Notice that $[M:L]=|H_{M}|$ and $[M:L]^{2} |\disc(M)|=|\disc(L)|=|G(L)|$.

\begin{defn}
Let $L_{1}$ and $L_{2}$ be two lattices.
A {\it gluing map} is an isomorphism $\varphi: H_{1} \longrightarrow H_{2}$ 
between subgroups $H_{i}\subset G(L_{i})$ such that 
$$
b_{G(L_{1})}(x,y)=-b_{G(L_{2})}(\varphi(x),\varphi(y)),
$$
for $x,y \in H_{1}$.
\end{defn}

Therefore, the subgroup ${H}_M:=\{(x, \varphi(x)) \mid x\in H_{1}\}\subset G(L_{1} \oplus L_{2})$ is isotropic, 
and thus $\varphi$ defines a lattice as 
\begin{equation}\label{eq:glue}
M=L_{1}\oplus_{\varphi} L_{2}
:=
\{(x, y)\in L_{1}^{\vee}\oplus L_{2}^{\vee}  \mid \bar{x} \in H_{1}, \bar{y}\in H_{2}, \bar{y}=\varphi(\bar{x})\}.
\end{equation}
Note that $\disc(L_{1})\disc(L_{2})=\disc(L_{1}\oplus_{\varphi} L_{2})|H_{M}|^{2}$.
Then $M$ is a primitive extension of $L_{1}\oplus L_{2}$ in the sense of that $M/L_{i}$ is torsion-free.
Moreover, the set of primitive extensions $L_{1}\oplus L_{2}\subset M$ 
is bijective to the set of gluing map $\phi: H_{1} \rightarrow H_{2}$ of subgroups of $G(L_{1})$ and $G(L_{2})$.

Finally, we conclude this section reviewing the extending isometries. 
An isometry $f\in \mathrm{O}(L)$ extends to an overlattice $M\supset L$ 
if $\bar{f}(H_{M})=H_{M}$. 
Also, there is a bijection between the set of extensions $f_{1}\oplus_{\varphi} f_{2}\in \mathrm{O}(M)$ of $f_{1}\oplus f_{2}\in \mathrm{O}(L_{1}\oplus L_{2})$ and the set of gluing maps $\varphi: H_{1}\rightarrow H_{2}$ with $\varphi\circ f_{1}=f_{2}\circ \varphi$.


\section{Special cubic fourfolds}\label{sec:special}

This section reviews some important aspects on lattice theory and Hodge theory of cubic fourfolds and mainly focuses on special cubic fourfolds;
we refer to Hassett \cite{Has00,Has16} and Huybrechts \cite{Huy19,Huy20} for more detailed discussions.

Let $X$ be a cubic fourfold.
It is well-known that the singular cohomology $H^{\ast}(X, \Z)$ is torsion-free 
and the Hodge diamond of $X$ is as following:
\begin{equation}\label{Hodge-diamond}
\begin{array}{cccc}
\begin{matrix}
&&&&&    1  \\
&&&&  0  &&  0 \\
&&& 0 && 1 && 0\\
&&0 && 0 &&  0 && 0\\
&0 && 1 && 21 && 1 && 0
\end{matrix}
\end{array}
\end{equation}
Moreover, by Poincar\'{e} duality and Hodge-Riemann bilinear relations, 
the middle cohomology $H^{4}(X, \Z)$ is an odd unimodular lattice of signature $(21,2)$ under the intersection form $(-.-)$.
Thus, 
$$
H^{4}(X, \Z) 
\cong 
\Lambda
:=E_{8}^{\oplus 2}\oplus U^{\oplus 2}\oplus \mathrm{I}_{3,0}.
$$
 Moreover, based on Beauville--Donagi \cite[Proposition 6]{BD85}, 
Hassett \cite[Proposition 2.1.2]{Has00} showed that the primitive cohomology
$$
H_{\mathrm{prim}}^{4}(X, \Z)
\cong  \Lambda_{0}:=\langle h^{2}\rangle^{\perp}_{\Lambda} \cong E_{8}^{\oplus 2}\oplus U^{\oplus 2}\oplus A_{2}, 
$$
where the square $h^2_X$ of the hyperplane class $h_X$ is corresponding to $h^{2}=(1, 1, 1)\in \mathrm{I}_{3,0}$.
Note that $\Lambda_{0}$ is an even lattice.
We denote by $A(X):=H^{4}(X, \Z)\cap H^{2,2}(X)$ the algebraic cohomology of $X$.
Then $A(X)$ is a positive definite lattice containing $h_X^2$ and $\langle h_X^2 \rangle_{A(X)}^{\perp}$ is an even lattice. 
Furthermore, by Voisin \cite[\S4, Proposition 1]{Voi86} and Hassett \cite{Has00}, 
$A(X)$ has no roots.

In the fundamental work \cite{Has00}, 
Hassett introduced the notion of special cubic fourfolds.

\begin{defn}[Hassett \cite{Has00}]
A {\it labelling of discriminant $d$} on a cubic fourfold $X$ is a positive definite rank-two primitive sublattice
$$
K \subset A(X) \text{ with }h^2_X\in K
$$
where $d$ is the discriminant of $K$.
A cubic fourfold $X$ is called {\it special (of discriminant $d$)} if $X$ contains a labelling (of discriminant $d$).
\end{defn}

In \cite[Theorem 18]{Voi07}, 
Voisin proved that the algebraic cohomology of a cubic fourfold is generated by the classes of algebraic cycles, i.e., the integral Hodge conjectrue holds for cubic fourfolds. 
Therefore, a cubic fourfold $X$ is special if and only if it contains an algebraic surface which is not homologous to a complete intersection. 

Next we will review the Hassett divisors parametrizing special cubic fourfolds 
in the moduli space of smooth cubic fourfolds.
Note that the Hilbert scheme of cubic hypersurfaces in $\mathbb{P}^{5}$ is $\mathbb{P}(H^{0}(\mathbb{P}^{5}, \mathcal{O}_{\mathbb{P}^{5}}(3)))\cong \mathbb{P}^{55}$; in addition, the smooth cubic hypersurfaces in $\mathbb{P}^{5}$ forming a Zariski open dense subset $\mathcal{U}\subset \mathbb{P}^{55}$. 
It is known that two cubic fourfolds are isomorphic if and only if they are projectively equivalent by the action of $\mathrm{PGL}(6, \mathbb{C})$.
As a result, 
the {\it moduli space of cubic fourfolds} is the GIT quotient
$$
\CC:=\mathcal{U}// \mathrm{PGL(6, \mathbb{C})}
$$
which is a quasi-projective variety of dimension $20$.
For a cubic fourfold $X$, we denote by $[X]$ its moduli point in $\CC$.
Hassett \cite[Theorem 1.0.1]{Has00} obtained the following structure and existence theorem (or classification theorem) of special cubic fourfolds.

\begin{thm}[Hassett \cite{Has00}]\label{thm:Cd}
Let $\CC_{d}\subset \mathcal{C}$ denote the special cubic fourfolds of discriminant $d$. 
Then $\CC_{d}$ is an irreducible divisor of $\CC$;
moreover, $\CC_{d}$ is non-empty if and only if 
\begin{equation*}\tag{$*$}\label{onestar}
d>6 \;\;\; \textrm{and}\;\;\; d\equiv 0, 2\; (\mathrm{mod}\; 6).
\end{equation*}
\end{thm}

Such a non-empty Noether–Lefschetz divisor $\CC_{d}$ is called a {\it Hassett divisor} of discriminant $d$.
We say that a special cubic fourfold $[X]\in \CC_{d}$ with a labelling $K_{d}$ has a {\it Hodge-theoretically associated K3 surface of degree $d$} in the sense of Hassett 
if there is a polarized K3 surface $(S, f)$ of degree $(f.f)=d$ and an Hodge isometry 
$H_{\mathrm{prim}}^{2}(S, \Z)(-1)\cong  K_{d}^{\perp}$.
In \cite[Theorem 5.1.3]{Has00}, 
Hassett showed that $[X]\in \mathcal{C}_{d}$ has a Hodge-theoretically associated K3 surface 
of degree $d$ if and only if  $d$ satisfies the following property
\begin{equation}\tag{$**$}\label{twostar}
4 \nmid d,\; 9\nmid d,\; p \nmid d \text{ for any odd prime } p \equiv 2\;  (\mathrm{mod}\, 3).
\end{equation}

We conclude this section by recalling Kuznetsov's conjecture \cite{Kuz10} for characterization of smooth rational cubic fourfolds.
For a cubic fourfold $X$, 
there is a semiorthogonal decomposition on the bounded derived category of coherent sheaves
$$
\D^{b}(X)=\langle \mathcal{A}_{X}, \CO_{X}, \CO_{X}(H), \CO_{X}(2H) \rangle, 
$$
where $\mathcal{A}_{X}:=\{E\in \D^{b}(X) \mid \Hom(\CO_{X}(iH),E[k])=0, \forall\, k\in \Z, i=0, 1, 2\}$ is called the {\it Kuznetsov component} of $X$.
In \cite[Conjecture 1.1]{Kuz10}, Kuznetsov conjectured that a cubic fourfold $X$ is rational if and only if $\mathcal{A}_{X}$ is equivalent to $\D^{b}(S)$ for some K3 surface $S$.
By \cite{AT14,Huy17,BLMNPS}, a cubic fourfold $[X]\in \mathcal{C}_{d}$ for some $d$ satisfying \eqref{twostar} if and only if $\mathcal{A}_{X}$ is equivalent to $\D^{b}(S)$ for some K3 surface $S$.
Therefore, Kuznetsov's conjecture can be restated as follows:

\begin{conj}\label{conj:Kuz}
A cubic fourfold $X$ is rational if and only if $[X]\in \CC_{d}$ with $d$ satisfying  \eqref{twostar}.
\end{conj}


\section{$M$-polarizable cubic fourfolds}

In this section, we discuss some basic properties (Lemma \ref{unique-CM-lem}, Proposition \ref{prop:CMo}, 
Corollary \ref{cor:irreducible}) of the moduli spaces $\mathcal{C}_M$, 
a generalization of Hassett divisors, 
of $M$-polarizable cubic fourfolds.

\begin{defn}\label{def:Mpolarizable}
Let $M$ be a positive definite lattice of rank $2\le r(M)\le 21$. We say a cubic fourfold $X$ is {\it $M$-polarizable} if there exists a primitive embedding $\iota: M \hookrightarrow A(X) \;\;\textrm{such that}\;\; h^2_X\in \iota(M)$. The pair $(X,\iota)$ is called an {\it $M$-polarized} cubic fourfold.  
We denote by $\mathcal{C}_M\subset \mathcal{C}$ the $M$-polarizable cubic fourfolds.
\end{defn}

\begin{rem}
(1) A cubic fourfold $X$ is a special cubic fourfold with a labelling $K_d$ of discriminant $d$ if and only if $X$ is a $K_d$-polarizable cubic fourfold. So lattice polarizable cubic fourfolds are generalization of special cubic fourfolds and $\mathcal{C}_M$ is a generalization of Hassett divisors $\CC_d$. We will study properties of $\CC_M$ toward generalizing Hassett's results (Theorem \ref{thm:Cd} and Corollary \ref{cor:rank2}) about $\CC_d$ to lattices $M$ of higher rank.

(2) Notice that the definition of $\mathcal{C}_M$ in this paper is different from that in \cite{Has16} (see also \cite{YY20}).  In fact, for a positive definite lattice $M$ together with a prescribed primitive embedding $\phi: M \hookrightarrow \Lambda$ such that $h^2\in \phi(M)$, the meaning of $\mathcal{C}_M$ in \cite{Has16} is the same as $\mathcal{C}_{(M, \mathfrak{o},\phi)}$ defined in Definition \ref{def:CMo} below, where $\mathfrak{o}=\phi^{-1}(h^2)$.

(3) The notion of $M$-polarized cubic fourfolds is analogous to lattice polarized K3 surfaces defined by Dolgachev \cite{Dol96}.  
For study of $M$-polarized cubic fourfolds and their moduli for explicit $M$ of higher rank
(see e.g. \cite{LPZ18,DM19,Aue20}).
\end{rem}

\begin{defn}
An element $\mathfrak{o}$ of a lattice $M$ is called {\it distinguished} if $(\mathfrak{o}.\mathfrak{o})=3$ and the orthogonal complement $\langle \mathfrak{o} \rangle^{\bot}_M$ is an even lattice.
\end{defn}

For any cubic fourfold $X$, the lattice $\langle h_X^2 \rangle^{\bot}_{A(X)}$ is even. Thus, for a positive definite lattice $M$, if there exists a $M$-polarizable cubic fourfold, then $M$ contains a distinguished element. By direct computation, we have the following

\begin{lem}\label{lem:eventest}
Let $\mathfrak{o}$ be an element of a lattice $M$ of rank $r(M)\ge 2$ with $(\mathfrak{o}.\mathfrak{o})=3$. Suppose $(\mathfrak{o}, e_1,...,e_{r(M)-1})$ is a basis of $M$. 
Then $\mathfrak{o}$ is a distinguished element of $M$ if and only if the integer $(e_i.e_i)-(\mathfrak{o}.e_i)$ is even for every $1\le i\le r(M)-1$.
\end{lem}

A {\it marking } of a cubic fourfold $X$ is an isometry $$\varphi: H^4(X,\mathbb{Z})\longrightarrow \Lambda$$ such that $\varphi(h_X^2)=h^2$. To study the structure of $\mathcal{C}_M$, we need the following

\begin{defn}\label{def:CMo}
Let $M$ be a positive definite lattice containing a distinguished element $\mathfrak{o}$.
\begin{enumerate}
\item[(i)] We define
$$
\CC_{(M, \mathfrak{o})}
:=
\{[X]\in \CC \mid  \; (X, \iota)\textrm{ is an } M\text{-polarized cubic fourfold with } \iota(\mathfrak{o})=h_{X}^{2} \}.
$$

\item[(ii)] We set 
$$
\mathcal{E}_{(M,\mathfrak{o})}:=\{\phi \mid \, \phi : M \hookrightarrow \Lambda \text{ is a primitive embedding with }\phi(\mathfrak{o})=h^2\}.
$$ 

\item[(iii)] 
For any $\phi \in \mathcal{E}_{(M,\mathfrak{o})}$, 
$$
\mathcal{C}_{(M,\mathfrak{o},\phi)}:=\{[X]\in \mathcal{C} \mid \, \exists \text{ marking }\varphi \text{ of } X \text{ such that }\; h^{2}\in \phi(M)\subset \varphi(A(X))\subset \Lambda\}.
$$
\end{enumerate}
\end{defn}

The following lemma is clear from Definition \ref{def:CMo}.
\begin{lem}\label{lem:union}
Let $M$ be a positive definite lattice containing a distinguished element $\mathfrak{o}$. Then 

\begin{enumerate}
\item $\mathcal{C}_{(M,\mathfrak{o})}=\cup_{\phi\in \mathcal{E}_{(M,\mathfrak{o})}} \mathcal{C}_{(M,\mathfrak{o},\phi)}$. 
\item For any $g\in \Gamma$ and any $\phi\in \mathcal{E}_{(M,\mathfrak{o})}$, 
we have $\mathcal{C}_{(M,\mathfrak{o},\phi)}=\mathcal{C}_{(M,\mathfrak{o},g\circ\phi)}$.
\end{enumerate}
\end{lem}

For rank two positive definite lattices $M$ with a distinguished element, the three notions $\mathcal{C}_M$, $\CC_{(M, \mathfrak{o})}$, $\mathcal{C}_{(M,\mathfrak{o},\phi)}$ give the same set of cubic fourfolds.

\begin{cor}[Hassett \cite{Has00, Has16}]\label{cor:rank2}
Let $M$ be a positive definite rank-two lattice of discriminant $d$.
Suppose that $M$ contains a distinguished element $\mathfrak{o}$.
Then one of the following two conditions must hold:
\begin{enumerate}
\item[(i)] $d>0,\, d\,\equiv\, 2\, (\mathrm{mod}\; 6)$
and there exists $v\in M$ such that
         $M=\langle\mathfrak{o}, v\rangle$ and its Gram matrix is
         $\small{
         \left(\begin{array}{cc} 
         3 & 1 \\
         1 & \frac{d+1}{3} 
         \end{array} \right)}
         $;
\item[(ii)] $d>0,\, d\,\equiv\, 0\, (\mathrm{mod}\; 6)$ and there exists $v\in M$ such that
         $M=\langle\mathfrak{o}, v\rangle$ and its Gram matrix is
         $\small{
         \left(\begin{array}{cc} 
         3 & 0 \\
         0 & \frac{d}{3} 
         \end{array} \right)}
         $.        
\end{enumerate}
In particular, if $M^{'}$ is another positive definite lattice containing a distinguished element $\mathfrak{o}^{'}$, then $(M^{'}, \mathfrak{o}^{'})$ and $(M, \mathfrak{o})$ are isometric if and only if $\disc(M)=\disc(M')$.
Moreover, for $d\geq 8$, 
$\mathcal{C}_d=\mathcal{C}_M=\CC_{(M, \mathfrak{o})}=\CC_{(M, \mathfrak{o},\phi)}$, where $\phi\in \mathcal{E}_{(M,\mathfrak{o})}$.
\end{cor}

Similarly, for a higher-rank positive definite lattice $M$,
we have the following:

\begin{lem}\label{unique-CM-lem}
Let $M$ be a positive definite lattice containing a distinguished element $\mathfrak{o}_1$.
If $\mathfrak{o}_2$ is another distinguished element of $M$, then there is an isometry $f\in \mathrm{O}(M)$ such that $f(\mathfrak{o}_1)=\mathfrak{o}_2$. In particular, $\mathcal{C}_M=\CC_{(M, \mathfrak{o}_1)}$.
\end{lem}

\begin{proof}
We consider the lattice $L:=\langle \mathfrak{o}_{1}, \mathfrak{o}_{2} \rangle\subset M$.
Its Gram matrix as follows:
$$
\left(\begin{array}{cc} 
         3 & a \\
         a & 3 
\end{array} \right).
$$
By definiteness, $a=\pm 2, \pm1, 0$.
Using Lemma \ref{lem:eventest}, we can rule out $a=0, \pm 2$.
Hence, $a=\pm 1$, and we have $L\cong K_{8}$. Replacing $\mathfrak{o}_2$ by $-\mathfrak{o}_2$ if necessary, we may and will assume $a=1$.

Let $N:=L^{\perp}_M$. Let $L^\prime:=N^\perp_M$. Then $L^\prime$ is the primitive closure of $L$ in $M$. Thus, either $[L^\prime:L]=2$ or $L^\prime=L$. Note that  $L^\prime \oplus_\phi N=M$ for some gluing map $\phi: H_{1}\longrightarrow H_{2}$, where $H_{1}\subset G(L^\prime)$ and $H_{2}\subset G(N)$. Next,  our goal is to find an isometry $f\in \mathrm{O}(M)$ such that $f(\mathfrak{o}_{1})=\mathfrak{o}_{2}$ by extending an isometry of $L^\prime \oplus N$ to that of $L^\prime\oplus_{\phi} N=M$.
Since the possible orders of the subgroup $H_{i}$ are $1, 2, 4, 8$, the arguments can be divided into three cases:

({\bf Case 1}) If the order $|H_{1}|=8$, then $L^\prime=L$, $H_{1}=G(L)=\langle \overline{\frac{-3\mathfrak{o}_{1}+\mathfrak{o}_{2}}{8}}\rangle\cong \Z/ 8\Z$. 
Thus, by (\ref{eq:glue}), there exists an element $\alpha\in N$ such that 
$$
\frac{\alpha}{8}\in N^{\vee}\; \textrm{and}\;  \frac{-3\mathfrak{o}_{1}+\mathfrak{o}_{2}}{8}+\frac{\alpha}{8}\in M.
$$
Set $\beta:=\frac{-3\mathfrak{o}_{1}+\mathfrak{o}_{2}+\alpha}{8}\in M$.
Then we have $(\mathfrak{o}_{1}.\beta)=-1$, $(\mathfrak{o}_{2}.\beta)=0$ (i.e., $\beta\in \langle \mathfrak{o}_{2}\rangle^{\perp}$) and $\beta+\mathfrak{o}_{2}\in \langle\mathfrak{o}_{1} \rangle^{\perp}_M$.
Therefore, we get
$$
(\beta+\mathfrak{o}_{2}.\beta+\mathfrak{o}_{2})=(\beta.\beta)+2(\beta.\mathfrak{o}_{2})+(\mathfrak{o}_{2}.\mathfrak{o}_{2})=(\beta.\beta)+3 \in 2\Z.
$$
This contradicts $\beta\in \langle \mathfrak{o}_{2}\rangle^{\perp}_M$, an even lattice.

({\bf Case 2}) Suppose $|H_{1}|<8$ and $L^\prime=L$. Then $|H_1|=1,2$ or $4$, and
we have
$$
H_{1}=\langle 0\rangle, \langle\overline{\frac{\mathfrak{o}_{1}+\mathfrak{o}_{2}}{2}}\rangle \;\textrm{or}\; \langle\overline{\frac{\mathfrak{o}_{1}+\mathfrak{o}_{2}}{4}}\rangle.
$$
We choose an isometry $f_1\in \mathrm{O}(L)$ such that $f_1(\mathfrak{o}_{1})=\mathfrak{o}_{2}$ and $f_1(\mathfrak{o}_{2})=\mathfrak{o}_{1}$.
Thus, $\bar{f}_1|{G(L)}=\Id_{G(L)}$. 
Consider $f_{2}:=\Id_{N}\in \mathrm{O}(N)$.
Then the isometry $f_{1}\oplus f_{2} \in \mathrm{O}(L\oplus N)$ extends to an isometry $f:=f_1\oplus_{\phi}f_2\in \mathrm{O}(L\oplus_{\phi} N)$ such that $f(\mathfrak{o}_{1})=\mathfrak{o}_{2}$.

({\bf Case 3}) Assume $|H_{1}|<8$ and $[L^\prime:L]=2$. Then $\disc(L^\prime)=2$, $L^\prime=\langle \mathfrak{o}_1,\frac{\mathfrak{o}_1+\mathfrak{o}_2}{2}\rangle$, $G(L^\prime)=\langle \frac{\overline{\mathfrak{o}_1+\mathfrak{o}_2}}{4}\rangle\cong \mathbb{Z}/2\mathbb{Z}$.
we have
$$
H_{1}=\langle 0\rangle\;\textrm{or}\; \langle\overline{\frac{\mathfrak{o}_{1}+\mathfrak{o}_{2}}{4}}\rangle.
$$
Agian, there exists an isometry $f_1\in \mathrm{O}(L^\prime)$ such that $f_1(\mathfrak{o}_{1})=\mathfrak{o}_{2}$ and $f_1(\mathfrak{o}_{2})=\mathfrak{o}_{1}$. As in (Case 2), the isometry $f_1\oplus_{\phi}{\rm id}_{G(N)}$ is what we want.
\end{proof}

\begin{rem}
For higher rank cases, $\CC_M=\CC_{(M,\mathfrak{o})}$ is in general not equal to $\CC_{(M,\mathfrak{o},\phi)}$ since $M$ might have inequivalent primitive embeddings into $\Lambda$ modulo the action of ${\rm O}(\Lambda,h^2)$. However, it turns out that $\CC_M =\CC_{(M,\mathfrak{o})}=\CC_{(M,\mathfrak{o},\phi)}$ in many higher rank cases (see Theorem \ref{mainthm-one-equal}).
\end{rem}

Next we recall the Torelli theorem for cubic fourfolds.  
Let $\mathcal{D}^{'}$ be one of the two connected components of 
$$
\mathcal{D}_{\Lambda}
:=
\{[\omega] \in \mathbb{P}(\Lambda_{0}\otimes \mathbb{C}) \mid (\omega.\omega)=0,\; (\omega.\bar{\omega})>0\}. 
$$
Note that $\Lambda_{0}=\langle h^{2}\rangle^{\perp}\subset \Lambda$. Then $\mathcal{D}^{'}$ is a bounded symmetric domain of type $\mathrm{IV}$ of dimension $20$. According to the Hodge diamond \eqref{Hodge-diamond},
it is well-known that $\mathcal{D}^\prime$ may be viewed as the classifying space of Hodge structures 
for cubic fourfolds. 
Denote by $\Gamma:=\mathrm{O}(\Lambda, h^{2})$ the orthogonal group of $\Lambda$ preserving the distinguished element $h^{2}$, 
and $\Gamma^{+}\subset \Gamma$ the subgroup stabilizing $\mathcal{D}^{'}$.
By Voisin's global Torelli theorem \cite{Voi86}, 
the period map 
$$
\begin{array}{cccl}
\wp:&\CC &\longrightarrow& \CD:= \Gamma^{+}\backslash \CD^{'} \\ 
&[X]&\longmapsto& H^{1}(X, \Omega_{X}^{3})
\end{array}
$$
is an open immersion of analytic spaces. 
Moreover, the global period domain $\CD$ is a quasi-projective variety of dimension $20$ and $\wp$ is an algebraic map (see \cite[Proposition 2.2.2]{Has00}). 
The image of the period map $\wp$ has been described explicitly in the work of Laza \cite[Theorem 1.1]{Laz10} and Looijenga \cite[Theorem 4.1]{Loo09}.

Building on \cite{Voi86,Has00,Laz10,Loo09},
Hassett gave the following structural result (\cite[Proposition 12 and p. 43]{Has16}); see also \cite[Proposition 5]{YY20}.

\begin{prop}\label{pre-version}
Let $M$ be a positive definite lattice of rank $2\leq r(M)\leq 21$ containing a distinguished element $\mathfrak{o}$ such that $\mathcal{E}_{(M,\mathfrak{o})}$ is non-empty. 
Suppose $\phi\in \mathcal{E}_{(M,\mathfrak{o})}$. 
Then the following statements hold:
\begin{enumerate}
\item[(i)] If $\CC_{(M,\mathfrak{o},\phi)}$ is non-empty, then $\CC_{(M,\mathfrak{o},\phi)}\subset \mathcal{C}$ has codimension $r(M)-1$ and
               there exists a cubic fourfold $[X]\in \CC_{(M,\mathfrak{o},\phi)}$ 
               such that $M\cong A(X)$; 
\item[(ii)]  $\CC_{(M,\mathfrak{o},\phi)}$ is non-empty 
                if and only if there is no sublattice 
                $K\subset M$, $\mathfrak{o}\in K$, with $K=K_{2}$ or $K_{6}$, 
where 
$
K_{2}=\small{\left(\begin{array}{cc} 
3 & 1 \\
1 & 1 
\end{array} \right)}
$ 
and
$
K_{6}=\small{\left(\begin{array}{cc} 
3 & 0 \\
0 & 2 
\end{array} \right)}$.
\end{enumerate}
\end{prop}

In \cite[Lemma 6]{YY20},
we gave the following useful criterion for checking the non-emptyness condition.

\begin{lem}\label{non-empty-criterion}
Let $M$ be a positive definite lattice of rank $2\leq r(M)\leq 21$ containing a distinguished element $\mathfrak{o}$ such that $\mathcal{E}_{(M,\mathfrak{o})}$ is non-empty.  
Suppose $\phi\in \mathcal{E}_{(M,\mathfrak{o})}$. 
Then the following conditions are equivalent:
\begin{enumerate}
\item[(1)] there is no sublattice $K\subset M$, $\mathfrak{o}\in K$, with $K=K_{2}$ or $K_{6}$;
\item[(2)]  $ M$ has no roots (i.e., $\nexists\; x\in M$ with $(x.x)=2$);
\item[(3)] $(x.x)\geq 3$ for all $0\neq x\in M$.
\end{enumerate}
If $M$ satisfies one of the above conditions, 
then there exists a cubic fourfold $[X]\in  \CC_{(M,\mathfrak{o},\phi)}$ such that $M\cong A(X)$; in particular, $\emptyset \neq \CC_{(M,\mathfrak{o},\phi)}\subset \CC_{(N,\mathfrak{o},\phi| N)}$ for any primitive sublattice $N\subset M$ with $\mathfrak{o}\in N$.
\end{lem}

\begin{rem}
(1) Let $M$ be a positive definite lattice of rank two containing a distinguished element. 
 By Lemma \ref{non-empty-criterion} and Corollary \ref{cor:rank2},  
 $M$ has no roots if and only if $\disc(M)\geq 8$.

(2)
It is in general nontrivial to determine existence/nonexistence of roots in positive definite lattices. In practice, we use the computer algebra system PARI/GP (see \cite{Th}) to find roots if roots do exist. On the other hand, our proofs for nonexistence of roots are free from computer (see the proofs of Theorems \ref{C20intC38}, \ref{subvar-Inter-all-HasD}).
\end{rem}

\begin{prop}\label{two-non-empty}
Let $M$ be a positive definite lattice containing a distinguished element $\mathfrak{o}$. Then the following two statements are equivalent:
\begin{enumerate}
\item $\CC_{(M, \mathfrak{o})}$ is non-empty;
\item $\mathcal{E}_{(M, \mathfrak{o})}$ is non-empty and $M$ has no roots.
\end{enumerate}
\end{prop}

\begin{proof}
Suppose $\CC_{(M, \mathfrak{o})}$ is non-empty. By definition, there is an $M$-polarized cubic fourfold $(X,\iota)$ with $\iota(\mathfrak{o})=h_{X}^{2}$. Since $A(X)$ has no roots, it follows that $M$ has no roots. Let $\varphi$ be a marking of $X$. Then $\varphi \iota \in \mathcal{E}_{(M, \mathfrak{o})}$. Thus, (2) is true.

Suppose $\mathcal{E}_{(M, \mathfrak{o})}$ is non-empty and $M$ has no roots. Let $\phi\in \mathcal{E}_{(M, \mathfrak{o})}$. By Proposition \ref{pre-version}, $\mathcal{C}_{\phi(M)}$ is non-empty. Then by Lemma \ref{lem:union}, $\CC_{(M, \mathfrak{o})}$ is non-empty.
\end{proof}

\begin{prop}\label{prop:CMo}
Let $M$ be a positive definite lattice containing a distinguished element $\mathfrak{o}$. Suppose $\CC_{(M, \mathfrak{o})}$ is non-empty. Then 
\begin{enumerate}
\item For any $\phi\in \mathcal{E}_{(M,\mathfrak{o})}$, $\mathcal{C}_{(M,\mathfrak{o},\phi)}$ is a closed subvariety of $\mathcal{C}$ of pure dimension $21-r(M)$ and it has at most two irreducible components;
\item $\mathcal{C}_{(M,\mathfrak{o})}$ is a closed subvariety of $\mathcal{C}$ of pure dimension $21-r(M)$.
\end{enumerate}
\end{prop}

\begin{proof} 
(1) By Proposition \ref{two-non-empty}, 
we may fix a primitive embedding $\phi: M\hookrightarrow \Lambda$ 
with $\phi(\mathfrak{o})=h^{2}$. 
Moreover, according to Lemma \ref{non-empty-criterion},
$\mathcal{C}_{(M,\mathfrak{o},\phi)}$ is non-empty.
Fix an element $\gamma_0\in \Gamma\setminus\Gamma^+$. We use $M_1$ and $M_2$ to denote the two sublattices $\phi(M)$ and $(\gamma_0\circ \phi)(M)$ of $\Lambda$ respectively.
The classifying space of Hodge structures on $M_i^{\perp}\subset \Lambda$ is determined by 
$\CD_{M_i}^{'}\subset \mathcal{D}^\prime$ which is one of the two connected components of
$$
\{[\omega] \in \mathbb{P}(M_i^{\perp} \otimes \CN) \mid (\omega.\omega)=0, (\omega.\bar{\omega})<0\}\subset \mathbb{P}(\Lambda_0 \otimes \CN).
$$
Following \cite[Subsection 2.2]{Has00}, 
one can show that  $\CD_{M_i}^{'}$ is a bounded symmetric domain of type $\mathrm{IV}$.
Let $\Gamma_{M_i}^{+}:=\{\gamma \in \Gamma^{+} \mid \gamma(M_i)=M_i\}$. 
Moreover,  the quotient $\Gamma_{M_i}^{+}\backslash \CD_{M_i}^{'}$
is a quasi-projective variety and the induced map
$$
\tau_{M_i}:
\Gamma_{M_i}^{+}\backslash \CD_{M_i}^{'} 
\longrightarrow
\Gamma^{+}\backslash \CD^{'}=\CD
$$
is an algebraically defined map.
Recall that $\mathcal{C}$ can be viewed as a Zariski open subset of $\CD$ via the period map $\wp$.  For $i=1,2$, we set $Z_i:=\Im(\tau_{M_i}) \cap \mathcal{C}$, where $\Im(\tau_{M_i})$ is the image of $\tau_{M_i}$. 
Note that $Z_i$ is an irreducible closed subvariety of $\mathcal{C}$ of dimension $21-r(M)$.

Let $[X]\in \mathcal{C}_{(M,\mathfrak{o},\phi)}$. 
By definition, $\exists $ an isometry $\varphi: H^4(X,\mathbb{Z})\longrightarrow \Lambda$ with $\varphi(h_X^2)=h^2$ such that $M_1\subset \varphi(A(X))$. Thus, we have $$h^2\in M_1\subset \varphi(A(X))\subset \Lambda$$ and $$h^2\in M_2\subset (\gamma_0\circ\varphi)(A(X))\subset \Lambda.$$ Let $\omega_X$ denote a nonzero element in  $ H^{1}(X, \Omega_{X}^{3})$. Then either $[(\varphi\otimes{\mathbb{C}})(\omega_X)]\in \CD^\prime$ or $[((\gamma_0\circ\varphi)\otimes{\mathbb{C}})(\omega_X)]\in \CD^\prime$.  Thus, either $[X]\in Z_1$ or $[X]\in Z_2$. Conversely, let $[X^\prime]\in Z_1\cup Z_2$. Reversing the above process, one can show that $[X^\prime]\in \mathcal{C}_{(M,\mathfrak{o},\phi)}$. Therefore, we have $ \mathcal{C}_{(M,\mathfrak{o},\phi)}=Z_1\cup Z_2$. This completes the proof of (1). For the proof of Corollary \ref{cor:irreducible} below, we remark that $Z_1$ is equal to $$\{[X]\in \mathcal{C} |\, \exists \text{ a marking }\varphi \text{ of } X \text{ such that }\; h^2\in \phi(M)\subset \varphi(A(X))\subset \Lambda,\, [\varphi_{\mathbb{C}}(\omega_X)]\in \CD^{\prime}\} .$$

(2) Primitive embeddings $\phi$ of $M$ into $\Lambda$ depends on two things: a) isomorphism classes of the orthogonal complement $\phi(M)^\perp\subset \Lambda$ and b) gluing maps between $G(M)$ and $G(\phi(M)^\perp)$. Since $\Lambda$ is unimodular, it follows that $\disc(\phi(M)^\perp)=\disc(M)$. The number of isomorphic classes of lattices with given rank and discriminant is finite (cf. \cite[Chapter 9, Theorem 1.1]{Cas78}). The number of all possible gluing maps in question is finite too. Thus, $\mathcal{E}_{(M,\mathfrak{o})}$ only contains finitely many orbits under the action $\phi\mapsto g\circ \phi $  by $\Gamma$. Then by (1) and Lemma \ref{lem:union}, $\mathcal{C}_{(M,\mathfrak{o})}$ is a finite union of closed subsets of $\mathcal{C}$ of pure dimension $21-r(M)$.
\end{proof}

\begin{rem}
Note that $\mathcal{C}_{(M,\mathfrak{o}, \phi)}$ and $\mathcal{C}_{(M,\mathfrak{o})}$ are in general not irreducible.
Let $T$ be the negative definite even lattice of rank two with Gram matrix 
$
\small{
\left(\begin{array}{cc} 
-14 & -5 \\
-5 & -58 
\end{array} \right)}
$.
Note that $\disc(T)=787$ is a prime 
and $\mathrm{O}(T)=\{\mathrm{id}, -\mathrm{id}\}$.
One can show that there is a primitive embedding $T\hookrightarrow \Lambda$ such that its orthogonal complement $M:=T^{\perp}\stackrel{\phi}{\hookrightarrow}  \Lambda$ is a positive definite rank $21$ primitive sublattice containing $h^{2}$ and $M$ has no roots.
Namely, we have $\phi\in \mathcal{E}_{(M, h^{2})}\neq \emptyset$ and $M$ has no roots.
Hence, Proposition \ref{two-non-empty} implies $\mathcal{C}_{(M, h^{2})}$ is non-empty.
Since $\mathrm{O}(T)=\{\mathrm{id}, -\mathrm{id}\}$, 
by Proposition \ref{prop:CMo}, 
$\mathcal{C}_{(M, h^{2}, \phi)}$ has exactly two cubic fourfolds.
Note that $r(M)+\ell(M)=21+1=22$ (compare with Theorem \ref{mainthm-one-equal}).
\end{rem}

\begin{cor}\label{cor:irreducible}
Let $M$ be a positive definite lattice containing a distinguished element $\mathfrak{o}$. Suppose $\CC_{(M, \mathfrak{o})}$ is non-empty. Then  $\mathcal{C}_{(M,\mathfrak{o})}$ is irreducible if and only if for any two $\phi_i\in \mathcal{E}_{(M,\mathfrak{o})}$, $i=1,2$, there exist $f\in {\rm O}(M, \mathfrak{o})$ and $\gamma\in \Gamma^+$ such that there is a commutative diagram 
$$
\xymatrix{
M \ar[d]_{\phi_1} \ar[r]^{f} & M \ar[d]^{\phi_2}\\
\Lambda \ar[r]^{\gamma} &\Lambda.
}
$$
\end{cor}

\begin{proof}
Suppose $\mathcal{C}_{(M,\mathfrak{o})}$ is irreducible. Let $\phi_i\in \mathcal{E}_{(M,\mathfrak{o})}$, $i=1,2$. By Lemma \ref{lem:union} and Proposition \ref{prop:CMo}, we have $\mathcal{C}_{(M,\mathfrak{o})}=\mathcal{C}_{(M,\mathfrak{o},\phi_1)}=\mathcal{C}_{(M,\mathfrak{o},\phi_2)}$. Choose $[X]\in \mathcal{C}_{(M,\mathfrak{o})}$ with $r(A(X))=r(M)$. By the remark in the end of the proof of Proposition \ref{prop:CMo} (1), there exists a marking $\varphi_i$ of $X$ such that $$h^2\in \phi_i(M)=\varphi_i(A(X))\subset \Lambda \text{ and } [(\varphi_i\otimes \mathbb{C})(\omega_X)]\in \CD^{\prime}.$$ Then $\gamma:=\varphi_{2}\varphi_1^{-1}\in \Gamma^+$, $f:=\phi_2^{-1}\gamma \phi_1\in {\rm O}(M,\mathfrak{o})$, and $\gamma \phi_1=\phi_2 f$.

Conversely, suppose for any two $\phi_i\in \mathcal{E}_{(M,\mathfrak{o})}$, $i=1,2$, there exist $f\in {\rm O}(M, \mathfrak{o})$ and $\gamma\in \Gamma^+$ such that $\gamma \phi_1=\phi_2 f$. By Lemma \ref{lem:union} (2), we have $\mathcal{C}_{(M,\mathfrak{o},\phi_1)}=\mathcal{C}_{(M,\mathfrak{o},\gamma \phi_1)}$. 
Clearly, $\mathcal{C}_{(M,\mathfrak{o},\phi_2)}=\mathcal{C}_{(M,\mathfrak{o},\phi_{2} f)}$. 
Thus, $\mathcal{C}_{(M,\mathfrak{o},\phi_1)}=\mathcal{C}_{(M,\mathfrak{o},\phi_2)}$. Then by Lemma \ref{lem:union} (1), we have $\mathcal{C}_{(M,\mathfrak{o})}=\mathcal{C}_{(M,\mathfrak{o},\phi)}$ for any $\phi\in  \mathcal{E}_{(M,\mathfrak{o})}$. Following the proof of Proposition \ref{prop:CMo} (1), one can show that $\mathcal{C}_{(M,\mathfrak{o},\phi)}$ is irreducible.
In fact, notice that both $\phi$ and $\gamma_0\phi$ are elements of $\mathcal{E}_{(M,\mathfrak{o})}$, and then $Z_1=Z_2$ by the existence of $f^\prime \in {\rm O}(M, \mathfrak{o})$ and $\gamma^\prime\in \Gamma^+$ such that $\gamma^\prime \phi=(\gamma_0\phi) f^\prime$. 
Therefore, $\mathcal{C}_{(M,\mathfrak{o})}$ is irreducible.
\end{proof}


\section{Non-emtpyness and irreducibility}

This section is mainly dedicated to the proof of Theorem \ref{mainthm-one}. Along the way, we obtain some results (Corollaries \ref{non-empty-CM}, \ref{non-empty-CM-special}) about the structure of the algebraic cohomologies of cubic fourfolds. 

For a rank-two positive definite lattice $K_d$ of discriminant $d$ containing a distinguished element, Hassett (Theorem \ref{thm:Cd}, Corollary \ref{cor:rank2}) gave very precise description of both non-emptyness and irriducibility of the moduli space $\mathcal{C}_{K_d}$ of $K_d$-polarizable cubic fourfolds. 
It is important, 
especially for studying irreducible components of intersection of Hassett divisors 
and finding new rational cubic fourfolds (see Section \ref{sec:intersection}), to generalize such precise description to higher rank positive definite lattices. The main result of this section is stated as follows:

\begin{thm}[see Theorem \ref{mainthm-one}]\label{mainthm-one-equal}
Let  $M$ be a positive definite lattice of rank $r(M)\ge 2$ containing a distinguished element $\mathfrak{o}$.
If the following conditions hold
\begin{enumerate}
\item $r(M)+\ell(M)\leq 20$; and
\item $M$ has no roots,
\end{enumerate}
then $\CC_{(M,\mathfrak{o})}$ is a non-empty irreducible closed subvariety of codimension $r(M)-1$ in $\CC$.
Moreover, $\mathcal{C}_M=\CC_{(M,\mathfrak{o})}=\CC_{(M,\mathfrak{o}, \phi)}$ for any $\phi\in \mathcal{E}_{(M,\mathfrak{o})}$.
\end{thm}

As a direct consequence this theorem, we have the following

\begin{thm}\label{prop:main}
Let  $M$ be a positive definite lattice containing a distinguished element $\mathfrak{o}$ with $2\le r(M)\le 10$.
If $M$ has no roots, then $\mathcal{C}_M=\CC_{(M,\mathfrak{o})}$ is a non-empty irreducible closed subvariety of codimension $r(M)-1$ in $\CC$.
\end{thm}

To prove Theorem \ref{mainthm-one-equal}, 
the issues on existence and uniqueness of primitive embeddings will be studied correspondently. Thanks to Nikulin \cite{Nik80} and technique of glue, 
it turns out that quite strong results (Propositions \ref{exist-primi-embed} and \ref{unique-primi-embed}) for these two issues hold. 

\begin{prop}\label{exist-primi-embed}
Let $M$ be a positive definite lattice containing a distinguished element $\mathfrak{o}$. 
If $r(M)+\ell(M)< 23$ and $r(M)\leq 21$,
then there exists a primitive embedding $\iota: M\hookrightarrow \Lambda$
such that $\iota(\mathfrak{o})=h^{2}$.
\end{prop}

\begin{proof}
First of all, let us setup some notations for later use:
\begin{itemize}
\item the lattice $L_{1}:=\langle \mathfrak{o} \rangle$ is odd and of signature $(1,0)$;
\item the lattice $L_{2}:=A_{2}(-1)$ is even and of signature $(0,2)$; and 
\item the lattice $N:=\langle \mathfrak{o} \rangle^{\perp}_M$ is even and of signature $(r(M)-1,0)$.
\end{itemize}

For one thing,
we notice that the glue group $G(L_{1})=\langle \overline{\frac{\mathfrak{o}}{3}}\rangle\cong \mathbb{Z}/3\mathbb{Z}$ 
and its symmetric bilinear form
$$
b_{G(L_{1})}(\overline{\frac{\mathfrak{o}}{3}}, \overline{\frac{\mathfrak{o}}{3}})\equiv \frac{1}{3}\; \mathrm{mod}\, \Z.
$$
Combining with Example \ref{A_2-example},
there exists an isometry
$$ 
\begin{array}{cccl}
\psi:&G(L_{1})&\longrightarrow& G(L_{2})\\
&\overline{\frac{\mathfrak{o}}{3}}&\longmapsto& \overline{\frac{-2e_{1}+e_{2}}{3}}
\end{array}.
$$ 
For another,
we have inclusions $$L_{1}\oplus N\subset M\subset (L_{1}\oplus N)^{\vee}=L_1^\vee\oplus N^\vee,$$
so the subgroup
$$
H_{12}:=\frac{M}{L_{1}\oplus N}
\subset \frac{L_{1}^{\vee}\oplus N^{\vee}}{L_{1}\oplus N}
=G(L_{1})\oplus G(N)
$$
is isotropic since $M$ is an overlattice of $L_{1}\oplus N$.
Next, we consider the isometry 
$$\psi\oplus \Id_{G(N)}: G(L_{1})\oplus G(N) \longrightarrow G(L_{2})\oplus G(N).$$
Hence, the subgroup
$$
H_{12}^{'}:=(\psi\oplus \Id_{G(N)})(H_{12}) 
\subset G(L_{2})\oplus G(N)
= \frac{(L_{2}\oplus N)^{\vee}}{L_{2}\oplus N}
$$
is aslo isotropic.
As a consequence, there exists an overlattice  
$M^{'} \subset (L_{2}\oplus N)^{\vee}$ of $L_{2}\oplus N$
such that 
$$
H_{12}^{'}=\frac{M^{'}}{L_{2}\oplus N}.
$$
Note that $G(M)\cong H_{12}^{\perp}/ H_{12}\cong H_{12}^{'\, \perp}/H_{12}^{'}\cong G(M^{'})$ 
and hence $\ell(M)=\ell(M^{'})$.
In addition, the lattice $M^{'}$ is even since the lattice $L_{2}\oplus N$ is even and $[M^\prime: L_2\oplus N]=|H^\prime_{12}|=|H_{12}|\in\{1,3\}$. 
Since $\sign(L_2)=(0,2)$ and $\sign(N)=(r(M)-1,0)$, it follows that $\sign(M^{'})=\sign(L_2\oplus N)=(r(M)-1, 2)$.
By the hypothesis $r(M)+\ell(M)<23$ and $r(M)\leq 21$,
we have $20-(r(M)-1)\geq 0$ and $(20-(r(M)-1))+(4-2)>\ell(M^{'})$.
Therefore, by \cite[Corollary 1.12.3]{Nik80},
there exists a primitive embedding 
\begin{equation}\label{eq:Mprime}
\iota^{'}: M^{'}\hookrightarrow \mathrm{II}_{20, 4}.
\end{equation}
We denote by $P:=\iota^{'}(M^{'})^{\perp}_{\mathrm{II}_{20, 4}}$ the orthogonal complement of $\iota^{'}(M^{'})$ in $\mathrm{II}_{20, 4}$. We may and will view $M^\prime$ as a sublattice of $\mathrm{II}_{20, 4}$ via $\iota^{'}$.
Then there are inclusions 
$$
L_{2}\oplus N\oplus P
\subset M^{'}\oplus P
\subset \mathrm{II}_{20, 4}
\subset (M^{'}\oplus P)^{\vee}
\subset (L_{2}\oplus N\oplus P)^{\vee}=L_{2}^\vee\oplus N^\vee\oplus P^\vee.
$$
We set the isotropic subgroup 
$$
H_{123}:=
\frac{\mathrm{II}_{20, 4}}{ L_{2}\oplus N\oplus P}
\subset G(L_{2})\oplus G(N)\oplus G(P)
=\frac{(L_{2}\oplus N\oplus P)^{\vee}}{L_{2}\oplus N\oplus P}. 
$$
Likewise, considering the isometry
$$
f:=\psi\oplus \Id_{G(N)}\oplus \Id_{G(P)}: G(L_{1})\oplus G(N)\oplus G(P) \longrightarrow G(L_{2})\oplus G(N)\oplus G(P),
$$
we have that the subgroup 
$$
H_{123}^{'}:=f^{-1}(H_{123}) \subset G(L_{1})\oplus G(N)\oplus G(P)=\frac{(L_{1}\oplus N\oplus P)^{\vee}}{L_{1}\oplus N\oplus P}
$$ 
is isotropic.
Consequently, there is an overlattice $J\subset (L_{1}\oplus N\oplus P)^\vee$ of $L_{1}\oplus N\oplus P$ such that 
$$
H_{123}^{'}=\frac{J}{L_{1}\oplus N\oplus P}\; .
$$
Note that $\sign(J)=(21,2)$, and $J$ is an odd lattice since the sublattice $L_{1}\subset J$ is odd.
Moreover, we have

$$
|\disc(J)|=\frac{|\disc(L_{1}\oplus N\oplus P)|}{|f^{-1}(H_{123})|^{2}}
=\frac{|\disc(L_{2}\oplus N\oplus P)|}{|H_{123}|^{2}}=\disc(\mathrm{II}_{20, 4})=1,
$$
so $J$ is unimodular. 
Thus, $J$ is isometric to $\Lambda$ by the theorem of Milnor.
Noticng that the primitive closure $L_3$ of $N\oplus P$ in $J$ is isometric to  
the primitive closure $L_4$ of $N\oplus P$ in $\mathrm{II}_{20,4}$.
Moreover, the primitive closure of $L_{1}\oplus N$ in $J$ is $M$. 
Then the lattice $L_3$ is even since the lattice $L_4\subset \mathrm{II}_{20,4}$ is even. 
We observe that $\langle \mathfrak{o}\rangle^\perp_J=L_3$. 
Therefore, there exists an isometry 
$$
\varphi: J\longrightarrow \Lambda
$$
such that $\varphi(\mathfrak{o})=h^{2}$. Then the restriction 
$$
\iota:=\varphi | M: M\longrightarrow \Lambda
$$ is a primitive emedding such that $\iota(\mathfrak{o})=h^2$.
\end{proof}

\begin{rem}
The condition $r(M)+\ell(M)\leq 23$ is necessary for the existence of a primitive embedding $M\hookrightarrow \Lambda$.
In fact, for such an embedding,
one has
$$
\ell(M)=\ell(M_{\Lambda}^{\perp})\leq \mathrm{min}\{r(M), 23-r(M) \},
$$
and hence $r(M)+\ell(M)\leq 23$. 
The remaining case $r(M)+\ell(M)= 23$ can be handled in a similar way. In fact, one may use \cite[Theorem 1.12.2]{Nik80} instead of \cite[Corollary 1.12.3]{Nik80} to determine the existence of the primitive imbedding $\iota^\prime$ in (\ref{eq:Mprime}). But this would be more involved and we will not use it in the sequel. So we would like to leave it to interested readers.
\end{rem}

\begin{cor}\label{non-empty-CM}
Let $M$ be a positive definite lattice containing a distinguish element $\mathfrak{o}$. If the following conditions hold
\begin{enumerate} 
\item $r(M)+\ell(M)<23$ and $2\leq r(M)\leq 21$; and
\item $M$ has no roots,
\end{enumerate}
then there exists a cubic fourfold $X$ 
and an isometry $f: M\longrightarrow A(X)$ such that $f(\mathfrak{o})=h_{X}^{2}$.
\end{cor}

\begin{proof}
This is a direct consequence of Proposition \ref{exist-primi-embed} 
and Lemma \ref{non-empty-criterion}.
\end{proof}

In particular, we have

\begin{cor}\label{non-empty-CM-special}
Let $M$ be a positive definite lattice of rank $2\leq r(M)\leq 11$ containing a distinguished element $\mathfrak{o}$.
If $M$ has no roots, then there exists a cubic fourfold $X$ and an isometry $f: M\longrightarrow A(X)$ such that $f(\mathfrak{o})=h_{X}^{2}$.
\end{cor}

\begin{rem}
An analogous result for the structure of N\'{e}ron--Severi groups of K3 surfaces holds (see \cite[Corollary 2.9 (i), Remark 2.11]{Mor84}).
\end{rem}

The following lemma  will be used in the proof of Proposition \ref{unique-primi-embed}.

\begin{lem}[{cf. \cite[Proposition 1.2.3]{Nik80}}]\label{lem:q2}
Let $L_{1}$ and $L_{2}$ be two even lattices and let $f:G(L_1)\longrightarrow G(L_2)$ be an isomorphism of abelian groups. 
Then the following two statements are equivalent:
\begin{enumerate}
\item $f$ is an isomorphism of quadratic forms;
\item $f$ is an isometry and $q_{G(L_2)}(f(x))=q_{G(L_1)}(x)$ for all $x\in G(L_1)_2$.
\end{enumerate}
\end{lem}

\begin{proof}
Suppose that (1) holds. 
Note that $q_{G(L_i)}$, $i=1,2$, 
are quadratic forms on $G(L_i)$ with symmetric bilinear forms $b_{G(L_i)}$. 
Since $f$ preserves the quadratic forms $q_{G(L_i)}$, 
it follows that $f$ preserves the bilinear forms $b_{G(L_i)}$ (see Definition \ref{def:quadratic} (2)). 
Thus, $f$ is an isometry. 
Moreover, $f$ preserves the restriction of  $q_{G(L_i)}$ to Sylow-$2$ subgroups $G(L_i)_2$. Therefore, (2) holds.

Suppose that (2) holds. 
Let $p$ be an odd prime number. 
We only need to show that $f$ preserves the restriction of  $q_{G(L_i)}$ to Sylow-$p$ subgroups $G(L_i)_p$ of $G(L_i)$. 
Note that $2q_{G(L_i)}(x_i)\equiv 2b_{G(L_i)}(x_i,x_i)$ (mod $2\mathbb{Z}$) for any $x_i\in G(L_i)$. Let $x\in G(L_1)_p$. 
Since $f$ is an isometry, it follows that 
$$
2q_{G(L_1)}(x)\equiv 2b_{G(L_1)}(x,x)\equiv 2b_{G(L_2)}(f(x),f(x))\equiv 2 q_{G(L_2)}(f(x))\; ({\rm mod}\; 2\mathbb{Z}).
$$ 
Then $q_{G(L_1)}(x)=q_{G(L_2)}(f(x))$ since $p$ and $2$ are coprime. 
Thus, (1) holds.
\end{proof}

Inspired by \cite[Proposition 3.2.4]{Has00},
we now obtain the following crucial result on uniqueness of the primitive embedding.

\begin{prop}\label{unique-primi-embed}
Let $M$ be a positive definite lattice containing a distinguished element $\mathfrak{o}$.
Suppose that $\iota_{i}: M\hookrightarrow \Lambda$ ($i=1, 2$) are two primitive embeddings with $\iota_{i}(\mathfrak{o})=h^{2}$.
If $r(M)+\ell(M)\leq 20$,
then there exists an element $\gamma\in \Gamma^{+}\subset {\rm O}(\Lambda)$ such that
the following diagram commutes
$$
\xymatrix@C=0.5cm{
& M \ar[ld]_{\iota_{1}} \ar[rd]^{\iota_{2}}&  \\
\Lambda \ar[rr]_{\gamma} && \Lambda.
}
$$ 
\end{prop}

\begin{proof}
We denote by $P_i$ ($i=1,2$)  the even lattices $\iota_{i}(M)^{\perp}_\Lambda$.
Then the signatures 
$$
\sign(P_{1})=\sign(P_{2})=(21-r(M), 2).
$$ 
Since $\Lambda$ is unimodular, it follows that  
$$
b_{G(P_{1})}\cong -b_{G(M)}\cong b_{G(P_{2})}.
$$ 
As a result, by \cite[Theorem 1.11.3]{Nik80}, 
the two quadratic forms $q_{G(P_{1})}$ and $q_{G(P_{2})}$ are isomorphic. Thus, $P_{1}$ and $P_{2}$ have the same genus.
Since $r(P_{i})=23-r(M)$ and $\ell(P_{i})=\ell(M)$,
by the hypothesis $r(M)+\ell(M)\leq 20$, 
we have 
$$
r(P_{i})\geq \ell(P_{i})+3 \geq \ell(P_{i})+2.
$$
According to Nikulin \cite[Theorem 1.14.2]{Nik80},
the two even indefinite lattices $P_{1}$ and $P_{2}$ are isometric.
Note that we have the primitive extensions $\iota_{i}(M)\oplus P_{i}\subset \Lambda$. 
Since the orders $|G(\iota_{i}(M))|=|G(P_{i})|$,
there are gluing maps $\phi_{i}:G(\iota_{i}(M)) \longrightarrow G(P_{i})$ such that
$$
\iota_{i}(M)\oplus_{\phi_{i}} P_{i}=\Lambda.
$$
Besides, there is an isomorphism $\psi: G(P_{1}) \stackrel{\simeq}{\longrightarrow} G(P_{2})$ 
of abelian groups
such that the following diagram commutes
\begin{equation}\label{diagram4.8-1}
\xymatrix{
G(\iota_{1}(M)) \ar[d]_{\bar{g}} \ar[r]^{\phi_{1}} & G(P_{1}) \ar[d]^{\psi}\\
G(\iota_{2}(M)) \ar[r]^{\phi_{2}} & G(P_{2}),
}
\end{equation}
where $\bar{g}$ is indued by the isometry $g: \iota_{1}(M)\longrightarrow \iota_{2}(M)$ with $g\circ \iota_{1}=\iota_{2}$. Since $\phi_{1}$ and $\phi_{2}$ are gluing maps and $\bar{g}$ is an isometry, it follows that $\psi$ is an isometry.

We claim that $\psi$ is an isomorphism of discriminant forms. 
In fact, by Lemma \ref{lem:q2}, it suffices to show that $\psi$ preserves the restriction of $q_{G(P_i)}$ to Sylow-2 subgroups $G(P_i)_2$.
Let $L:=\langle \iota_{i}(\mathfrak{o}) \rangle=\langle h^2 \rangle$ and $N_{i}:=L^{\perp}_{\iota_{i}(M)}$, $i=1,2$. Note that the Sylow-2 subgroup $G(L)_2$ is trivial since $\disc(L)=3$. Moreover, $[\iota_i(M): L\oplus N_i]\in \{1,3\}$ (in other words, $\iota_i(M)$ is obtained by gluing $L$ and $N_i$ along finite abelian groups of order $1$ or $3$). Thus, there is a well-defined isometry
$$
\varphi_{i}: G(N_{i})_{2} \longrightarrow G(\iota_{i}(M))_{2}
$$
such that $\varphi_i(\bar{x})=\bar{x}$ for all $x\in N_i^\vee$ with $\bar{x}\in G(N_{i})_{2}$ (see the proof of \cite[Lemma 4.1]{OY20}). We denote by $\overline{N_{i}\oplus P_{i}}$ the primitive closure of $N_{i}\oplus P_{i}$ in $\Lambda$. Note that 
$$
\overline{N_{i}\oplus P_{i}}=L^\perp_\Lambda=\langle h^2 \rangle^\perp_\Lambda
$$ 
is an even lattice.
Let $x\in N_{i}^{\vee}$ with $\bar{x}\in G(N_{i})_{2}$.
Let $y\in P_{i}^{\vee}$ with $(\phi_{i}\circ \varphi_{i})(\bar{x})=\bar{y}\in G(P_i)_2$.
Then $(x, y)\in \overline{N_{i}\oplus P_{i}}=\langle h^2 \rangle^\perp_\Lambda$,
which implies that $((x,y).(x,y))$ is even.
Since $(x.x)_{N_{i}^{\vee}}+(y.y)_{P_{i}^{\vee}}=((x,y).(x,y))$ is even, 
it follows that
$$
q_{G(N_{i})_{2}}(\bar{x})+q_{G(P_{i})_{2}}(\bar{y})= 0\in \mathbb{Q}/ 2\Z.
$$
So we have
$\phi_{i}\circ \varphi_{i}: q_{G(N_{i})_{2}} \stackrel{\simeq}{\longrightarrow}-q_{G(P_{i})_{2}}$.
Note that we have the following commutative diagram
\begin{equation}\label{diagram4.8-2}
\xymatrix{
G(N_{1})_{2} \ar[d]_{\overline{g|N_1}} \ar[r]^{\varphi_{1}} & G(\iota_{1}(M))_2 \ar[d]_{\bar{g}} \ar[r]^{\phi_{1}} &  G(P_{1})_2 \ar[d]^{\psi} \\
G(N_{2})_{2} \ar[r]^{\varphi_{2}} & G(\iota_{1}(M))_2 \ar[r]^{\phi_{2}} & G(P_{2})_2\; . 
}
\end{equation}
Since $\overline{g|N_1}$ is induced from the isometry $g|N_1: N_1\longrightarrow N_2$, it follow that $\overline{g|N_1}$ is an isomorphism of discriminant forms. 
Thus, $\psi$ preserves the restriction of $q_{G(P_i)}$ to Sylow-2 subgroups $G(P_i)_2$ by the commutativity of the diagram (\ref{diagram4.8-2}). 
This completes the proof of the claim.

Recall that we have proved that $P_1$ and $P_2$ are isometric. 
Let $\mu: P_1\longrightarrow P_2$ be an isometry. 
Then $\psi\circ \bar{\mu}^{-1}: G(P_2)\longrightarrow G(P_2)$ is an isomorphism of discriminant forms, i.e., $\lambda:=\psi\circ \bar{\mu}^{-1}\in {\rm O}(q_{G(P_2)})$. 
Since we have
$$
r(P_{2})\geq \ell(P_{2})+3 \geq \ell(P_{2})+2,
$$
by \cite[Theorem 1.14.2]{Nik80}, 
the induced homomorphism $\mathrm{O}(P_{2}) \longrightarrow \mathrm{O}(q_{G(P_{2})})$ given by $f\mapsto \bar{f}$ is surjective.
By this, there is an isometry $\tilde{\lambda}\in \mathrm{O}(P_{2})$ 
such that $\overline{\tilde{\lambda}}=\lambda$. 
Let $\tilde{\psi}:=\tilde{\lambda}\circ \mu$.
Considering the following commutative diagrams
$$
\xymatrix@C=0.5cm{
& P_{1}\ar[ld]_{\mu} \ar[rd]^{\tilde{\psi}}&  \\
P_{2} \ar[rr]_{\tilde{\lambda}} && P_{2}
}
\textrm{and}
\xymatrix@C=0.5cm{
& q_{G(P_{1})}\ar[ld]_{\bar{\mu}} \ar[rd]^{\psi}&  \\
q_{G(P_{2})}\ar[rr]_{\lambda} && q_{G(P_{2})}\; ,
}
$$
we see that $\overline{\tilde{\psi}}=\psi$.
By the diagram \eqref{diagram4.8-1},
the isometry 
$$
g\oplus \tilde{\psi}: \iota_{1}(M)\oplus P_{1}\longrightarrow \iota_{2}(M)\oplus P_{2}
$$ 
extends to an isometry
$$
\delta: 
\iota_{1}(M)\oplus_{\phi_{1}} P_{1}
\longrightarrow 
\iota_{2}(M)\oplus_{\phi_{2}} P_{2}.
$$
As a result, $\delta\in \mathrm{O}(\Lambda)$ with $ \delta \circ \iota_{1}=\iota_{2}$ such that $\delta(h^{2})=h^{2}$. 
Since $r(P_{2})\geq \ell(P_{2})+3$, 
by Nikulin \cite[Corollary 1.13.5]{Nik80},
there is a lattice $S$ such that $P_{2}\cong U\oplus S$.
Then the isometries 
$$
\chi:=(-\Id_{U})\oplus \Id_{S} \in \mathrm{O}(U\oplus S)=\mathrm{O}(P_{2}) \text{ and } \Id_{\iota_{2}(M)}\in {\rm O}(\iota_2(M))
$$
extend to 
$$
\sigma:=\Id_{\iota_{2}(M)}\oplus_{\phi_{2}} \chi \in \mathrm{O}(\iota_{2}(M)\oplus_{\phi_{2}} P_{2}).
$$ 
Note that $\sigma \in \Gamma \setminus \Gamma^{+}$.
Then we have the following two cases:
\begin{enumerate}
\item if $\delta \in \Gamma^{+}$, 
then $\gamma:=\delta$ is what we wanted; 
\item if $\delta \in \Gamma \setminus  \Gamma^{+}$,
then $\gamma:= \sigma \circ \delta \in \Gamma^{+}$ as in the following diagram
$$
\xymatrix@C=0.5cm{
&M \ar[ld]_{\iota_{1}}  \ar[d]_{\iota_{2}} \ar[rd]^{\iota_{2}} &&  \\
\Lambda \ar[r]^{\delta} & \Lambda \ar[r]^{\sigma} & \Lambda.
}
$$ 
\end{enumerate}
This completes the proof of Proposition \ref{unique-primi-embed}.
\end{proof}

\begin{rem}
If $r(M)+\ell(M)=21$ and $\iota_{i}(M)_{\Lambda}^{\perp}$ contains $U(m)$ ($m\in \{1, 2\}$) as a direct summand,
then we can also obtain $\gamma \in \Gamma^{+}$ such that $ \gamma \circ \iota_{1}=\iota_{2}$.
In fact, if $r(M)+\ell(M)=21$, by the same arguments, 
we can obtain $\delta\in \Gamma$ such that $\delta \circ \iota_{1}=\iota_{2}$. 
The second condition guarantees the existence of $\chi$ and $\sigma$ in the proof of Proposition \ref{unique-primi-embed} (cf. \cite[Proposition 5.6]{Dol96}).
For instance, if $M$ is the positive definite lattice of rank $11$ with the Gram matrix $(a_{ij})$ such that $a_{ii}=3$ and $a_{ij}=1$ for $i\neq j$, then $r(M)+\ell(M)=21$;
moreover, for any primitive embedding $M\hookrightarrow \Lambda$, 
the orthogonal complement contains $U(2)$ as a direct summand (see \cite[Lemma 7.1]{DM19}).
\end{rem}

Finally, we are in the position to finish the proof of Theorem \ref{mainthm-one}.
Actually, we will give the proof of Theorem \ref{mainthm-one-equal} which clearly implies Theorem \ref{mainthm-one}.

\begin{proof}[Proof of Theorem \ref{mainthm-one-equal}]
Under the hypothesis of the theorem, 
Corollary \ref{non-empty-CM} yields the non-emptyness of $\CC_{(M, \mathfrak{o})}$.
By Propositions \ref{prop:CMo}, \ref{unique-primi-embed} and Corollary \ref{cor:irreducible}, $\CC_{(M, \mathfrak{o})}\subset \mathcal{C}$ is an irreducible closed subvariety of codimension $r(M)-1$, and $\CC_{(M, \mathfrak{o})}=\CC_{(M, \mathfrak{o},\phi)}$ for any $\phi\in\mathcal{E}_{(M, \mathfrak{o})}$. Moreover, by Lemma \ref{unique-CM-lem}, $\CC_M=\CC_{(M, \mathfrak{o})}$.
\end{proof}


\section{Fermat cubic fourfold in Hassett divisors}

In this section, we investigate the algebraic cohomology of the Fermat cubic fourfold. In particualr, we give an explicit description (Proposition \ref{algcohom-Fermat}) of this lattice in terms of $21$ planes. Then we show that the Fermat cubic fourfold is contained in all Hassett divisors (Theorem \ref{mainthm-two}), and hence we give a new proof of Hassett's existence theorem of special cubic fourfolds.

\subsection{Algebraic cohomology of Fermat cubic fourfold}
Let
$$
X_{F}:=\{(z_{1}:z_{2}: \cdots: z_{6})\in \mathbb{P}^{5} \mid z_{1}^{3}+z_{2}^{3}+z_{3}^{3}+z_{4}^{3}+z_{5}^{3}+z_{6}^{3}=0\}
$$ 
be the Fermat cubic fourfold.
It is an important concrete cubic fourfold, 
and its algebraic cohomology attains the maximal rank $21$ (cf. \cite[Proposition 11]{Bea14}).
Planes on $X_F$ turn out to be crucial for our study of the algebraic cohomology $A(X_F)$. 
Segre \cite{Seg44} showed that the Fermat quartic surface contains exactly $48$ lines. 
Analogously, we have the following observation.

\begin{lem}\label{Fermat-405planes}
Fermat cubic fourfold $X_{F}$ contains exactly $405$ planes.
\end{lem}

\begin{proof}
It follows from the equation that $X_{F}$ contains $405$ planes as follows:
$$
\{z_{l}-\epsilon_{1}z_{i}=z_{m}-\epsilon_{2}z_{j}=z_{n}-\epsilon_{3}z_{k}=0\},
$$
where $\epsilon_{1}^{3}=\epsilon_{2}^{3}=\epsilon_{3}^{3}=-1$ and $[\{l, i\}; \{m, j\}; \{n, k\}]$ is an unordered partition of the index set $\{1, 2, 3, 4, 5, 6\}$ into three unordered pairs.

Next we will show that $X_{F}$ contains only $405$ planes.
Following Segre \cite[Section 2]{Seg44},
a plane 
$$
z_{l}=a_{1}z_{i}+a_{2}z_{j}+a_{3}z_{k},\;
z_{m}=b_{1}z_{i}+b_{2}z_{j}+b_{3}z_{k},\;
z_{n}=c_{1}z_{i}+c_{2}z_{j}+c_{3}z_{k},
$$
lies on the Fermat cubic fourfold if and only if 
\begin{equation}\label{equ-exact405P}
z_{i}^{3}+z_{j}^{3}+z_{k}^{3}
=
-(a_{1}z_{i}+a_{2}z_{j}+a_{3}z_{k})^{3}-(b_{1}z_{i}+b_{2}z_{j}+b_{3}z_{k})^{3}-(c_{1}z_{i}+c_{2}z_{j}+c_{3}z_{k})^{3}
\end{equation}
where $\{i, j, k, l, m, n\}$ is a permutation of $\{1, 2, 3, 4, 5, 6\}$.
By direct computation of the determinants of Hessian matrices for both sides of the equation \eqref{equ-exact405P},
we have
$$
z_{i}z_{j}z_{k}
=\mathrm{C}\cdot(a_{1}z_{i}+a_{2}z_{j}+a_{3}z_{k})(b_{1}z_{i}+b_{2}z_{j}+b_{3}z_{k}) (c_{1}z_{i}+c_{2}z_{j}+c_{3}z_{k})
$$
where $\mathrm{C}:=(a_{3}b_{2}c_{1}-a_{2}b_{3}c_{1}-a_{3}b_{1}c_{2}+a_{1}b_{3}c_{2}+a_{2} b_{1}c_{3}-a_{1}b_{2}c_{3})^2$.
As a result, the claim follows.
\end{proof}

Based on the above lemma,
we will show that the algebraic cohomology of the Fermat cubic fourfold is generated by the cohomology classes of planes. 
For this goal, we need the following useful result.

\begin{lem}[{cf. Voisin \cite[\S 3, Appendix]{Voi86}}]\label{intersect-planes}
If a cubic fourfold $X$ contains two distinct planes $P_{1}$ and $P_{2}$,
then one of the following conditions holds
\begin{enumerate}
\item $([P_{1}].[P_{2}])=0$ if $P_{1}\cap P_{2}$ is empty;
\item $([P_{1}].[P_{2}])=1$ if $P_{1}\cap P_{2}$ is a point; 
\item $([P_{1}].[P_{2}])=-1$ if $P_{1}\cap P_{2}$ is a line.
\end{enumerate}
\end{lem}

\begin{proof}
The first is clear. For the second, since $P_{1}$ and $P_{2}$ meet only at a point $p$, it follows that $P_{1}$ and $P_{2}$ meet transversely at $p$ and thus $([P_{1}].[P_{2}])=1$.
For the third one, suppose $P_{1}\cap P_{2}=l$ is a line.
Suppose that $W\subset \mathbb{P}^{5}$ is the unique 3 dimensional linear subspace generated by $P_{1}$ and $P_{2}$.
By suitable linear transformation of the coordinates, 
we may assume $W=\{z_{1}=z_{2}=0\}\subset \mathbb{P}^{5}$, $P_{1}=\{z_{1}=z_{2}=z_{3}=0\}$ and $P_{2}=\{z_{1}=z_{2}=z_{4}=0\}$.
Let $X:=\{F=0\}$ be the cubic fourfold.
Then the homogenous polynomial $F=z_{1}F_{1}+z_{2}F_{2}+G(z_{3}, z_{4}, z_{5}, z_{6})$
and $P_{1}, P_{2}\subset \{ G=0\}$.
Hence, $G=z_{3}z_{4}H$, where $H$ is a homogenous polynomial of degree 1.
By the smoothness of $X$, $z_{3}\nmid H$ and $z_{4}\nmid H$.
Let $P_{2}=\{z_{1}=z_{2}=H=0\}$. 
We have $W\cap X=P_{1}\cup P_{2}\cup P_{3}$ and $h_{X}^{2}=[P_{1}]+[P_{2}]+[P_{3}]$.
Set $a=([P_{1}].[P_{2}])$, $b=([P_{1}].[P_{3}])$ and $c=([P_{2}].[P_{3}])$.
Since $([P_{i}].h_{X}^{2})=1$, so we have $a=b=c=-1$.
\end{proof}

Now we can determine explicitly the generators for the algebraic cohomology of the Fermat cubic fourfold.

\begin{prop}\label{algcohom-Fermat}
The algebraic cohomology of the Fermat cubic fourfold is generated by the cohomology classes of $21$ planes $P_i$ (i=1,...,21) whose defining equations are given in the Appendix \eqref{21-planes}.
\end{prop}

\begin{proof}
We denote by $M:=\langle \xi_{1}, \xi_{2}, \cdots, \xi_{21} \rangle \subset A(X_{F})$
the sublattice generated by the cohomology classes $\xi_{i}$ of the $21$ planes $P_{i}$.
According to Lemma \ref{intersect-planes}, 
the intersection matrix $((\xi_{i}.\xi_{j}))$ is \eqref{intersect-matrix-21-planes}.
By direct computation, the determinant of this matrix is $27$. 
Thus, $r(M)=21$ and $(\xi_{1}, \xi_{2}, \cdots, \xi_{21})$ is a basis of $M$.
Since the rank of $A(X_{F})$ is not bigger than $21$, we also have $r(A(X_{F}))=21$. By $(h_{X_F}^2.\xi_{i})=1$ ($i=1,...,21$), the square $h_{X_F}^{2}$ of hyperplane class is expressed as 
\begin{equation}\label{eq:h2Fermat}
h_{X_F}^{2}=(0,-1, -1, 1, 0, -1, 0, 1, -1, 1, -1, 1, 1, 1, -1, 0, -1, 0, 2, 2, 0)
\end{equation}
under the basis $(\xi_1,...,\xi_{21})$ of $M$.

Next, we will show that $M=A(X_{F})$. Suppose otherwise. Then $A(X_F)$ is an overlattice of $M$ with $l:=[A(X_F):M]>1$. Since $\disc(M)=27$, it follows that $l=3$ and $\disc(A(X_F))=3$.

We claim that $M$ has only one nontrivial overlattice.
To show this, 
we pick up a basis of the dual $M^{\vee}$ of $M$, 
$$
M^{\vee}=\langle \xi_{1}, \xi_{2}, \cdots, \xi_{19}, \eta, \theta \rangle\subset M\otimes \mathbb{Q},
$$
where 
$\eta=-\frac{2}{3}\xi_{2}-\frac{2}{3}\xi_{3}-\frac{2}{3}\xi_{4}-\frac{1}{3}\xi_{5}-\frac{2}{3}\xi_{6}-\frac{2}{3}\xi_{7}-\frac{1}{3}\xi_{8}-\frac{2}{3}\xi_{10}-\frac{2}{3}\xi_{14}-\frac{1}{3}\xi_{15}-\frac{1}{3}\xi_{16}-\frac{1}{3}\xi_{17}+\frac{1}{3}\xi_{20}$
and  
$\theta
=-\frac{4}{9}\xi_{1}-\frac{5}{9}\xi_{2}-\frac{2}{9}\xi_{3}+\frac{1}{9}\xi_{4}+\frac{1}{9}\xi_{5}+\frac{2}{9}\xi_{7}+\frac{1}{9}\xi_{8}-\frac{1}{3}\xi_{9}+\frac{2}{9}\xi_{10}-\frac{2}{9}\xi_{11}-\frac{4}{9}\xi_{12}-\frac{1}{9}\xi_{13}-\frac{2}{3}\xi_{14}-\frac{2}{9}\xi_{15}-\frac{4}{9}\xi_{16}-\frac{7}{9}\xi_{17}-\frac{7}{9}\xi_{18}-\frac{5}{9}\xi_{19}-\frac{1}{9}\xi_{20}+\frac{1}{9}\xi_{21}$.
Hence, the glue group of $M$,
$$
G(M)=M^{\vee}/ M=\langle \bar{\eta}, \bar{\theta} \rangle\cong \Z/3\Z \oplus \Z/9\Z,
$$
and its intersection matrix
$$
\left(
\begin{array}{cc}
b_{G(M)}(\bar{\eta},\bar{\eta}) & b_{G(M)}(\bar{\eta},\bar{\theta}) \\
b_{G(M)}(\bar{\theta},\bar{\eta}) & b_{G(M)}(\bar{\theta},\bar{\theta})
\end{array}
\right)
=
\left(\begin{array}{cc}
\frac{2}{3} & \frac{1}{3} \\
\frac{1}{3} & \frac{5}{9}
\end{array}\right).
$$
We shall determine the isotropic vectors $n\bar{\eta}+m\bar{\theta}$ ($0\leq n\leq 2$ and $0\leq m\leq 8$), 
that is
$$
\left(\begin{array}{cc}
n & m
\end{array}
\right)
\left(\begin{array}{cc}
\frac{2}{3} & \frac{1}{3} \\
\frac{1}{3} & \frac{5}{9}
\end{array}
\right)
\left(\begin{array}{c}
n \\
m
\end{array}
\right)
\equiv 0 \;\mathrm{mod}\; \Z.
$$
So the isotropic vectors are $0$, $3\bar{\theta}$ and $6\bar{\theta}$.
Therefore, 
the glue group $G(M)$ of $M$ has only one non-trivial isotropic subgroup generated by 
$3\bar{\theta}$ and thus 
$M$ has only one nontrivial overlattice 
$$
\overline{M}:=
\langle \xi_{1}, \xi_{2}, \cdots, \xi_{19}, \xi_{20}, 3\theta \rangle.
$$
Thus, we have $A(X_{F})=\overline{M}$. 
However, the lattice $\overline{M}$ contains a root $e$ such that $(h_{X_{F}}^{2}.e)=0$, 
e.g., $e:=2\xi_{1}+2\xi_{2}+\xi_{3}-\xi_{4}-\xi_{7}-\xi_{8}+\xi_{9}-\xi_{10}+\xi_{11}+\xi_{12}+2\xi_{14}+\xi_{15}+\xi_{16}+3\xi_{17}+3\xi_{18}+\xi_{19}+3\theta$ 
with $(e.e)=2$.
This contradicts to \cite[\S 4, Proposition 1]{Voi86}.
\end{proof}

\begin{rem}\label{rem:FermatC14}
(1) In fact, Voisin proved that the algebraic cohomology of a special cubic fourfold can be generated either by smooth surfaces (cf. \cite[the proof of Theorem 5.6]{Voi17}) or by possibly singular rational surfaces (\cite{Voi07} and \cite[Remark 5.9]{Voi17}). 
A natural question arises (cf. \cite[Question 15]{Has16} and \cite[\S 6.1]{Nue17}): 
{\it is the algebraic cohomology of a special cubic fourfold generated by the classes of smooth rational surfaces}?
Our Proposition \ref{algcohom-Fermat} implies that this question holds for Fermat cubic fourfold.

(2) Let $X$ be a cubic fourfold containing a plane $P$.
Then there is a primitive sublattice $\langle h_X^2, [P]\rangle \subset A(X)$ of discriminant $8$ and thus $[X]\in \CC_{8}$.
Conversely, if $[X]\in \CC_{8}$ then $X$ contains a plane by Voisin \cite[\S 3]{Voi86}. 
Let $X$ be a cubic fourfold containing two disjoint planes $P$ and $P^\prime$.
Then $X$ is rational and the primitive sublattice $\langle h_X^2, h_{X}^{2}-[P]-[P^\prime]\rangle \subset A(X)$ is of discriminant $14$.
Therefore, $[X]\in \CC_{8}\cap \CC_{14}$ and $X$ has a Hodge-theoretically associated K3 surface. 
In particular, the Fermat cubic fourfold $[X_F]\in\CC_{8}\cap \CC_{14}$ is rational, and it has a (unique) Hodge-theoretically associated K3 surface $S$ with transcendental lattice 
$
T(S)=\left(\begin{array}{cc} 
6 & 3 \\
3 & 6 
\end{array} \right)
$. 
\end{rem}

\subsection{Proof of Theorem \ref{mainthm-two}}
This subsection is going to show Theorem \ref{mainthm-two}, 
that is, the Fermat cubic fourfold $[X_{F}]\in \CC_{d}$ for all integers $d$ satisfying \eqref{onestar}.
To this end, 
we would like to apply two classical results in number theory: Lagrange's four-square theorem and an analogue result of Ramanujan.

\begin{lem}[Lagrange]\label{Lagrange}
Any positive integer can be expressed by the form
$x^{2}+y^{2}+z^{2}+u^{2}$
for some integers $x, y, z, u$.
\end{lem}

\begin{lem}[{Ramanujan \cite[Section 7]{Ram}}]\label{Ramanujan}
Any positive integer $l$ except for $l=1$ and $17$ can be expressed as the form
$2x^{2}+2y^{2}+2z^{2}+3u^{2}$ 
for some integers $x, y, z, u$.
\end{lem}

Before entering into the detailed arguments, 
we first sketch the basic idea for the proof of Theorem \ref{mainthm-two}.
By definition, a cubic fourfold $X$ is a special cubic fourfold of discriminant $d$,
i.e., $[X]\in \CC_{d}$, 
if and only if there exists an element $v \in A(X)$ such that the rank $2$ sublattice $\langle h_{X}^{2}, v \rangle \subset A(X)$ is primitive and moreover its discriminant
\begin{equation}\label{idea-equation}
\disc\, \langle h_{X}^{2}, v \rangle
=3(v.v)-(h_{X}^{2}.v)^{2}=d.
\end{equation}
In other words,
to smainthm-twohow that $[X]\in \CC_{d}$ for any integer $d=6k$ ($k\geq 2$) or $d=6k+2$ ($k\geq 1$), 
it is enough to find an element $v\in A(X)$ such that the rank $2$ sublattice $\langle h_{X}^{2}, v \rangle \subset A(X)$ is primitive and $d$ is represented by the {\it integral quadratic form} as \eqref{idea-equation}.

Now we are ready to prove Theorem \ref{mainthm-two}.

\begin{proof}[Proof of Theorem \ref{mainthm-two}]
From now on, we would consider the Fermat cubic fourfold $X_{F}$.
First of all, we pick the basis $(\xi_{1}, \xi_{2}, \cdots, \xi_{21})$ of $A(X_{F})$ as in the proof Proposition \ref{algcohom-Fermat}. Recall that $h_{X_F}^2$ is given by (\ref{eq:h2Fermat}).

Inspired by Lagrange's four-square theorem (see Lemma \ref{Lagrange}) and the result of Ramanujan (see Lemma \ref{Ramanujan}),
we will pick five pairs of planes $(P'_{11}, P'_{12}), \cdots, (P'_{51}, P'_{52})$ and another plane $P'$; see Appendix (\ref{5+1-planes}) for the explicit defining equations.
Under the basis $(\xi_{1}, \xi_{2}, \cdots, \xi_{21})$, 
their cohomology classes are expressed as \eqref{11planes}.
We denote by $\alpha_{i}:=[P'_{i1}]-[P'_{i2}]$ ($i=1,2,3,4,5$) and $\beta:=[P']$ the cohomology classes.
Next, we shall show the primitivity of the sublattice $\langle h_{X_F}^{2}, \alpha_{1}, \alpha_{2}, \alpha_{3}, \alpha_{4}, \alpha_{5}, \beta\rangle \subset A(X_{F})$.
For this goal,
we consider the coordinates of $h_{X_F}^{2}$, $\alpha_{1}$, $\alpha_{2}$, $\alpha_{3}$, $\alpha_{4}$, $\alpha_{5}, \beta$ under the basis $(\xi_{1}, \xi_{2}, \cdots, \xi_{21})$.
In fact, the coordinates form a $7\times 21$ matrix as follows:
$$
\tiny{
\left(
\begin{array}{ccccccccccccccccccccc}
 0& -1& -1& 1&	0& -1&  0&1&  -1&  1& -1& 1& 1& 1& -1& 0 & -1& 0&	2& 2& 0 \\
 0& -1& -1& 1&-1& -1&  0& 1&  -1& 1& -1& 1& 1& 1&	 -1&-1& -1& 0&	2& 1& 0\\
-1& -1& -1& 0& 0&  0& -1& 1&  -1& 1&  0& 0& 0& 0&	  0&	1& 0& 0&	1& 1& 0 \\
-2& -4& -5& 2& 0& -2& -1& 3&	-4& 3& -4& 5& 4& 3& -2&	1& -4& -2& 5&	5& -1\\
-1& -3& -3& 2& 0& -2&  0& 2&	-3& 3& -2& 3& 2& 3& -2&	0& -4& -1& 4&	3& -1\\
-2& -3& -3& 2& 0& -3& -1& 3&	-3& 2& -3& 3& 3& 2&  -2&	0& -3& -1& 4&	5& 0 \\
-3& -6& -8& 3& 1& -4& 1& 3& -6& 6& -6& 6& 7& 4& -3& 1& -7& -3& 9& 8& -2
\end{array}
\right).
}
$$
Considering the columns $1, 2, 3, 4, 5, 6, 8$ of the above matrix,
they form a $7\times 7$ matrix and 
its determinant is $1$, which implies the primitivity we wanted. 
By Lemma \ref{intersect-planes}, 
its Gram matrix is as follows: 
\begin{equation}\label{matrix}
\left(
\begin{array}{ccccccc} 
3 & 0 & 0 & 0 & 0 & 0 & 1 \\
0 & 4 & 0 & 0 & 0 & 0 & 0 \\
0 & 0 & 4 & 0 & 0 & 0 & 0 \\
0 & 0 & 0 & 4 & 0 & 0 & 0 \\
0 & 0 & 0 & 0 & 4 & 0 & 0 \\
0 & 0 & 0 & 0 & 0 & 6 & 0 \\
1 & 0 & 0 & 0 & 0 & 0 & 3  
\end{array}
\right).
\end{equation}

Now we will discuss the core of the proof.
Suppose
$v:=x_{1}\alpha_{1}+x_{2}\alpha_{2}+x_{3}\alpha_{3}+x_{5}\alpha_{5}+y\beta$ for integers $y$ and $x_{i}$ ($i=1,2, \cdots ,5$).
Then its self-intersection
$$
(v.v)=4x_{1}^{2}+4x_{2}^{2}+4x_{3}^{2}+4x_{4}^{2}+6x_{5}^{2}+3y^{2},
\;\;
\textrm{and}
\;\;
(h_{X_F}^{2}.v)=y.
$$
Considering the rank $2$ sublattice 
$$
\langle h_{X_F}^{2}, v \rangle\subset A(X_{F}),
$$ 
and then its discriminant 
\begin{equation}\label{our-form-for-d}
\disc\, \langle h_{X_F}^{2}, v \rangle
= 3(4x_{1}^{2}+4x_{2}^{2}+4x_{3}^{2}+4x_{4}^{2}+6x_{5}^{2})+8y^{2}.
\end{equation}

In the next, the proof is concluded by the following two cases:

\;

\paragraph{\bf{Case 1}: $d=6k$ for $k\geq 2$.}

Let $y:=0$ in \eqref{our-form-for-d}.
We shall find integers $x_{i}$ ($i=1,2,3,4,5$) such that 
$3(4x_{1}^{2}+4x_{2}^{2}+4x_{3}^{2}+4x_{4}^{2}+6x_{5}^{2})=6k=d$.
That means, we have
\begin{equation}\label{Ram-Lag-form}
k=2x_{1}^{2}+2x_{2}^{2}+2x_{3}^{2}+2x_{4}^{2}+3x_{5}^{2}.
\end{equation}
Considering the following two detailed cases:
\begin{itemize}
\item Suppose $k=2m$ ($m\geq 1$) and $x_{1}=1$. 
Hence the rank-two sublattice $\langle h_{X_F}^{2}, v \rangle\subset A(X_{F})$ is primitive.
By the result of Ramanujan (cf. Lemma \ref{Ramanujan}), 
there exist integers $x_{2}, x_{3}, x_{4}, x_{5}$ such that $k=2m$ is expressed as the form \eqref{Ram-Lag-form}
\item Suppose $k=2m+1$ ($m\geq 1$)  and $x_{5}=1$. 
Therefore, the rank-two sublattice $\langle h_{X_F}^{2}, v \rangle\subset A(X_{F})$ is also primitive.
By Lagrange's four-square theorem (cf. Lemma \ref{Lagrange}),
there exist integers $x_{1}, x_{2}, x_{3}, x_{4}$ such that $k=2m+1$ is expressed as the form \eqref{Ram-Lag-form}
\end{itemize}
In summary, for any $d=6k$ with $k\geq 2$, 
there exists an element $v\in A(X_{F})$ such that the sublattice $\langle h_{X_F}^{2}, v \rangle\subset A(X_{F})$ is primitive
and its discriminant $\disc \, \langle h_{X_F}^{2}, v \rangle=d=6k$.

\;

\paragraph{\bf{Case 2}: $d=6k+2$ for $k\geq 1$.}

Since the Fermat cubic fourfold $[X_{F}]\in \CC_{8} \cap \CC_{14}$, 
we may assume $k\geq 3$.
Let $y=1$.
Then the sublattice $\langle h_{X_F}^{2}, v \rangle\subset A(X_{F})$ is primitive and
its discriminant
$$
\disc\, \langle h_{X_F}^{2}, v \rangle
= 3(4x_{1}^{2}+4x_{2}^{2}+4x_{3}^{2}+4x_{4}^{2}+6x_{5}^{2})+8.
$$
Set $l:=k-1\geq 2$. 
Then $d=6k+2=6l+8$ and hence {\bf Case 2} follows immediately from {\bf Case 1}.
\end{proof}

Theorem \ref{mainthm-two} yields immediately a new proof of the Hassett's existence theorem  \cite[Theorem 4.3.1]{Has00} of special cubic fourfolds.

\begin{cor}[Hassett]
For any integer $d$ satisfying \eqref{onestar},
there exists a special cubic fourfold of discriminant $d$.
\end{cor}
 
Notice that our new proof is quite different from Hassett's proof. His approach \cite[Section 4]{Has00} uses the Torelli theorem for cubic fourfolds, the surjectivity of the period map for K3 surfaces, limiting Hodge structure, etc. 
Our approach here uses an explicit description (Proposition \ref{algcohom-Fermat}) of the algebraic cohomology of the Fermat cubic fourfold and two results (Lemmas \ref{Lagrange}, \ref{Ramanujan}) in number theory.  
An advantage of our approach is the following: 
for any $\CC_{d}$ with $d$ satisfying \eqref{onestar}, 
we give an explicit example of cubic fourfold contained in $\CC_{d}$.

As a final remark,  we noted that, in \cite{Awa20}, Awada studied to what extent one can intersect Hassett divisors in $\mathcal{C}$ and obtained non-emptyness of intersection of 20 Hassett divisors $\mathcal{C}_{d_i}$ ($i=1,...,20$) under some arithmetic conditions for the indexes $d_i$.

\section{Intersection of Hassett divisors}\label{sec:intersection}

In this section, based on Theorems \ref{mainthm-one} and \ref{mainthm-two}, we study intersection of 
Hassett divisors. 
We give an algorithm to determine all irreducible components of intersection of any two Hassett divisors (Algorithm \ref{alg:int}).
In particular, we use Algorithm \ref{alg:int} to determine all irreducible components of $\CC_{20}\cap\CC_{38}$ (Theorem \ref{C20intC38}), and we find new examples of rational cubic fourfolds parametrized by two divisors in $\CC_{20}$ (Corollary \ref{new-rat-cubics}) using explicit description of such components.

\subsection{Irreducible components of $\CC_{d_1}\cap \CC_{d_2}$: an algorithm}

\begin{lem}\label{lem:subset}
Let $M_i$, $i=1,2$ be two positive definite lattices of rank $r(M_i)\ge 2$ containing distinguished elements $\mathfrak{o}_i$ such that $\mathcal{C}_{M_i}\neq \emptyset$. 
Then $\mathcal{C}_{M_1}\subset \mathcal{C}_{M_2}$ if and only if there exists a primitive embedding $\varphi: M_2 \hookrightarrow M_1$ such that $\varphi(\mathfrak{o}_2)=\mathfrak{o}_1$.
\end{lem}

\begin{proof}
By Lemma \ref{unique-CM-lem},  $\mathcal{C}_{M_i}=\mathcal{C}_{(M_i,\mathfrak{o}_i)}$, $i=1,2$. Suppose $\mathcal{C}_{(M_1,\mathfrak{o}_1)}\subset \mathcal{C}_{(M_2,\mathfrak{o}_2)}$. 
Since $\mathcal{C}_{(M_1,\mathfrak{o}_1)}\ne \emptyset$, there exists a $M_{1}$-polarizable cubic fourfold $X$ and an isometry $\psi_1: M_1\longrightarrow A(X)$ such that $\psi_1 (\mathfrak{o}_1)=h_X^2$. Then $[X]\in \mathcal{C}_{(M_2,\mathfrak{o}_2)}$, which implies that there exists a primitive embedding $\psi_2: M_2\hookrightarrow A(X)$ such that $\psi_2 (\mathfrak{o}_2)=h_X^2$. Then $\psi_1^{-1}\psi_2$ is a primitive embedding from $M_2$ to $M_1$ mapping $\mathfrak{o}_2$ to $\mathfrak{o}_1$. Conversely, if such a primitive embedding exists, then clearly $\mathcal{C}_{(M_1,\mathfrak{o}_1)}\subset \mathcal{C}_{(M_2,\mathfrak{o}_2)}$.
\end{proof}

\begin{rem}
Note that the condition $\varphi(\mathfrak{o}_2)=\mathfrak{o}_1$ is necessary. Let $M_1:=\langle \mathfrak{o}_1, e_1,e_2 \rangle$ with Gram matrix $\left(
\begin{array}{cccc} 
3 & 1 & 4 \\
1& 3 &3  \\
4 & 3 & 10
\end{array}
\right)$, and let $M_2\subset M_1$ be the primitive sublattice generated by $\{e_1, 2e_2+\mathfrak{o}_1\}$ with Gram matrix $\left(\begin{array}{cc} 
         3 & 7 \\
         7 & 59 
         \end{array} \right)$. By Lemma \ref{lem:eventest}, $e_1$ is a distinguished element of $M_2$ (but not $M_1$). Then one can show that $\mathcal{C}_{M_1}$ is not contained in $\mathcal{C}_{M_2}=\mathcal{C}_{128}.$
\end{rem}

\begin{prop}\label{prop:int}
Let $M_i$, $i=1,...,k$ be positive definite lattices of rank $r(M_i)\ge 2$ containing distinguished elements $\mathfrak{o}_i$ such that the intersection 
$$
Z:=\bigcap_{i=1}^{i=k}\mathcal{C}_{M_i}\subset \CC
$$ 
is non-empty. Then there exist a positive integer $n$ and positive definite lattices $M_j^\prime$, $j=1,...,n$ containing distinguished elements $\mathfrak{o}_j^\prime$ with $\mathcal{C}_{M_j^\prime}\ne \emptyset$ such that  
\begin{enumerate}
\item $Z$ is the union of $\mathcal{C}_{M_j^\prime}$, $j=1, 2,\ldots,n$; 

\item $r(M_j^\prime)\leq r(M_1)+...+r(M_k)-k+1$, for any $1\leq j \leq n$; and

\item for any $i\neq j$, there exists no primitive embedding $\varphi: M_i^\prime \hookrightarrow M_j^\prime$ with $\varphi(\mathfrak{o}^\prime_i)=\mathfrak{o}^\prime_j$.
\end{enumerate}
Moreover, if $r(M_j^\prime)+l(M_j^\prime)\leq 20$ for all $1\leq j\leq n$, then $Z$ has exactly $n$ irreducible components $ \mathcal{C}_{M_1^\prime},..., \mathcal{C}_{M_n^\prime}$. 
\end{prop}

\begin{proof}
By induction, it suffices to prove the case $k=2$. From now on, we will assume $k=2$. We set $\mathcal{L}$ to be the set of pairs $(M,\mathfrak{o})$ consisting of a positive definite lattice $M$ and a distinguished element $\mathfrak{o}\in M$ such that 
\begin{enumerate}
\item[(a)] $\mathcal{C}_M$ is non-empty;
\item[(b)] $M$ has two primitive sublattices $N_i$ containing $\mathfrak{o}$, $i=1,2$;
\item[(c)] there exist isometries $\psi_i$ from $M_i$ to $N_i$ mapping $\mathfrak{o}_i$ to $\mathfrak{o}$; and
\item[(d)] $M$ is the primitive closure (inside $M$) of the sublattice generated by $N_1$ and $N_2$.
\end{enumerate}
Let $r_i=r(M_i)$, $i=1,2$. 
Note that $r(M)\leq r_1+r_2-1$ for any $(M,\mathfrak{o})\in \mathcal{L}$. 
Next we show that $\mathcal{L}$ is a finite set (up to isometries). 
Choose a basis $\{\alpha_{i1},...,\alpha_{ir_i}\}$ of $M_i$ such that $\alpha_{i1}=\mathfrak{o}_i$. 
Let $(M,\mathfrak{o})\in \mathcal{L}$. 
By (c) and (d), there exists a subset $\{\beta_1,...,\beta_s \}$ of $\{\alpha_{22},...,\alpha_{2r_2}\}$ such that 
$$
\psi_1(\alpha_{11}),...,\psi_1(\alpha_{1r_1}),\psi_2(\beta_1),...,\psi_2(\beta_s)
$$
is a $\mathbb{Q}$-basis of $M\otimes\mathbb{Q}$. Let $N$ denote the sublattice of $M$ generated by the elements 
$$
\psi_1(\alpha_{11}),...,\psi_1(\alpha_{1r_1}),\psi_2(\beta_1),...,\psi_2(\beta_s).
$$ 
Let $A$ denote the Gram matrix of $N$ with respect to this basis. Since $A$ is positive definite and the diagonal entries are bounded, it follows that $A$ has only finitely many possibilities, which implies both $N$ and its overlattice $M$ have only finitely many possibilities, up to isometry. 
Therefore, we may assume $\mathcal{L}=\{(M_1^\prime,\mathfrak{o}_1^\prime)$,.., $(M_n^\prime,\mathfrak{o}_n^\prime)\}$ for some positive integer $n$. 
Then 
$$
\bigcup_{j=1}^{n} \mathcal{C}_{(M_j^\prime,\mathfrak{o}_j^\prime)}\subset Z=\mathcal{C}_{(M_1,\mathfrak{o_1})}\cap \mathcal{C}_{(M_2,\mathfrak{o}_2)}.
$$ 
Let $[X]\in Z$. Then there exist primitive embeddings $\phi_i: M_i\hookrightarrow A(X)$ with $\phi_i(\mathfrak{o}_i)=h_X^2$. 
Let $M_X$ be the primitive closure of the sublattice $\phi_1(M_1)+\phi_2(M_2)$ of $A(X)$. 
Clearly, up to isometry, the pair $(M_X,h_X^2)$ is contained in $\mathcal{L}$. Then we have
$$
[X]\in \mathcal{C}_{(M_X,h_X^2)}\subset \bigcup_{j=1}^{n} \mathcal{C}_{(M_i^\prime,\mathfrak{o}_i^\prime)}.
$$
Thus, we get
$$
Z= \bigcup_{j=1}^{n} \mathcal{C}_{(M_j^\prime,\mathfrak{o}_j^\prime)}= \bigcup_{j=1}^{n} \mathcal{C}_{M_j^\prime}.
$$ 
By Lemma \ref{lem:subset}, after removing superfluous lattices we may and will assume that the statement (3) in the proposition holds. 
If $r(M_j^\prime)+l(M_j^\prime)\leq 20$ for all $1\leq j\leq n$, then by Theorem \ref{mainthm-one}, all $\mathcal{C}_{M_j^\prime}$ are irreducible, which implies that $Z$ has exactly $n$ irreducible components $ \mathcal{C}_{M_1^\prime},..., \mathcal{C}_{M_n^\prime}$.
\end{proof}

\begin{rem}
By Theorem \ref{mainthm-two},
for any given finitely many Hassett divisors $C_{d_{i}}$, $i=1,2, \cdots, k$, 
the intersection $Z:=C_{d_{1}}\cap C_{d_{2}} \cap \cdots \cap C_{d_{k}}$ is non-empty. 
Therefore, one can use Proposition \ref{prop:int} to determine the irreducible components.
\end{rem}

In particular, we have the following 

\begin{cor}\label{cor:twoCdint}
Let $Z\subset \mathcal{C}$ be the intersection of two different Hassett divisors $\mathcal{C}_{d_1}$ and $\mathcal{C}_{d_2}$. 
Then there exist a positive integer $n$ and  rank $3$ positive definite lattices $M_j$, $j=1,...,n$  with $\mathcal{C}_{M_j}\ne \emptyset$ such that $Z$ has exactly $n$ irreducible components $ \mathcal{C}_{M_1},..., \mathcal{C}_{M_n}$.
\end{cor}

The following algorithm is based on Corollary \ref{cor:twoCdint} and the proof of Proposition \ref{prop:int}.

\begin{alg}\label{alg:int}
Input: a pair of distinct integers $(d_1,d_2)$ such that both $\mathcal{C}_{d_1}$ and $\mathcal{C}_{d_2}$ are non-empty. To find all irreducible components of $\mathcal{C}_{d_1}\cap \mathcal{C}_{d_2}$, proceed as follows.

(1) Let $M_\tau$ denote a rank 3 lattice with a basis $(\mathfrak{o}_{\tau},\mu_{\tau},\nu_{\tau})$ such that the corresponding Gram matrix (by abuse of notation, we use $M_\tau$ to denote this matrix too) is: 
\begin{enumerate}
\item[(i)] If  $d_1=6n_{1}+2, d_2=6n_{2}+2, 1\leq n_1<n_2$, then 
$$
M_{\tau}=\left(
\begin{array}{cccc} 
3 & 1 & 1 \\
1& 2n_{1}+1 & \tau  \\
1 & \tau & 2n_{2}+1
\end{array}
\right),
$$
where $\tau\in \mathbb{Z}$;

\item[(ii)] If $d_1=6n_{1}+2, d_2=6n_{2}, 1\leq n_1,2\leq n_2$, then 
$$
M_{\tau}=\left(
\begin{array}{cccc} 
3 & 1 & 0 \\
1& 2n_{1}+1 & \tau  \\
0 & \tau & 2n_{2}
\end{array}
\right),
$$
where $\tau\geq 0$ and $\tau\in \mathbb{Z}$;

\item[(iii)] If  $d_1=6n_{1}, d_2=6n_{2}, 2\leq n_1<n_2$, then  
$$
M_{\tau}=\left(
\begin{array}{cccc} 
3 & 0 & 0 \\
0& 2n_{1} & \tau  \\
0 & \tau & 2n_{2}
\end{array}
\right),
$$
where $\tau\geq 0$ and $\tau\in \mathbb{Z}$.
\end{enumerate}

(2) Find the set $\mathcal{T}_1$ of integers $\tau$ such that $\disc(M_{\tau})>0$ (this implies that $M_{\tau}$ is positive definite). Note that $\mathcal{T}_{1}$ is a finite set.

(3) Find the set $\mathcal{T}_2:=\{\tau \in \mathcal{T}_1 \mid M_{\tau} \text{ has no roots}\}$. Suppose $\mathcal{T}_2$ has exactly $k$ elements: $\tau_1,...,\tau_k$. Let $\mathcal{M}:=\{M_{\tau_1},...,M_{\tau_k} \}$. 

(4) For each $\tau\in\mathcal{T}_2$, find the set $\mathcal{N}_{\tau}:=\{N |\, N \text{ is a nontrivial overlattice of }M_{\tau}$, $K_{\tau,1} \text{ and } K_{\tau,2} \text{ are primitive sublattices of }N, \langle \mathfrak{o}_{\tau}\rangle_{N}^\perp \text{ is even, and }N \text{ has no roots}\}.$ Here $K_{\tau,1}:=\langle\mathfrak{o}_{\tau},\mu_{\tau}\rangle\subset M_{\tau}$ and $K_{\tau,2}:=\langle\mathfrak{o}_{\tau},\nu_{\tau}\rangle\subset M_{\tau}$.

(5) Find the set $\mathcal{M}^\prime=\{M_1^\prime,...,M_n^\prime\}$ of representatives of isometric classes of lattices in the union $\mathcal{M}\cup \mathcal{N}_{\tau_1}\cup...\cup \mathcal{N}_{\tau_k}$. Then output that $\mathcal{C}_{d_1}\cap \mathcal{C}_{d_2}$ has exactly $n$ irreducible components $\mathcal{C}_{M_1^\prime},...,\mathcal{C}_{M_n^\prime}$.
\end{alg}

\begin{rem}\label{alg:int-rem}
(a) By Lemma \ref{lem:eventest}, $\mathfrak{o}_{\tau}$ is a distinguished element of $M_{\tau}$ in all the three cases (i)-(iii) of the step (1). 
Note that in cases (ii) and (iii), there exists an isometry from $M_{\tau}$ to $M_{-\tau}$ mapping $\mathfrak{o}_{\tau},\mu_{\tau},\nu_{\tau}$ to $\mathfrak{o}_{-\tau},-\mu_{-\tau},\nu_{-\tau}$ respectively. Thus, we may require $\tau\geq 0$ in these two cases.

(b) In step (3), the $k$ lattices in $\mathcal{M}$ are of mutually different discriminants. By Theorem \ref{mainthm-one}, $\mathcal{C}_{M_{\tau_i}}$ is non-empty and irreducible for any $1\leq i\leq k$. Thus, in step (5), $n\ge k$, and we may choose $M_{\tau_i}$ as $M_{i}^\prime$ for $i=1,...,k$. 

(c) By Theorem \ref{prop:main}, 
Algorithm \ref{alg:int} can be adapted to compute irreducible components of $\mathcal{C}_M\cap \mathcal{C}_{d_2}$ for any positive definite lattice $M$ of rank $r(M)<10$ with $\mathcal{C}_M\ne \emptyset$.
\end{rem}

The following lemma is useful for determination of overlattices in the step (4) of Algorithm \ref{alg:int}.

\begin{lem}\label{lem:detMrank3}
Let $M$ be a rank three positive definite lattice containing a distinguished element $\mathfrak{o}$. Then $\disc(M)\equiv 0,1\; ({\rm mod}\; 4)$.
\end{lem}

\begin{proof}
Let $L:=\langle \mathfrak{o} \rangle_M^\perp $. Then $M\cong \langle \mathfrak{o} \rangle\oplus_{\psi} L$, where $\psi $ is a gluing map from $H_1\subset G(\langle \mathfrak{o} \rangle)\cong \mathbb{Z}/3 \mathbb{Z}$ to some subgroup $H_2\subset G(L)$ with $|H_1|=|H_2|\in\{1,3\}$. Thus, either $3\disc(L)=\disc(M)$ or $\disc(L)=3 \disc(M)$. The discriminant of any rank two positive definite even lattice must be either $0$ or $3$ $({\rm mod}\;4)$. Note that $L$ is such a lattice, which implies $\disc(M)\equiv 0, 1$ (mod 4).
\end{proof}

\begin{thm}\label{C20intC38}
The intersection $\CC_{20}\cap \CC_{38}$ has exactly $17$ 
irreducible components $\CC_{M_{\tau}}$ which are given by the rank $3$ lattices
$$
M_{\tau}:=
\left(
\begin{array}{cccc} 
3 & 1 & 1 \\
1& 7 & \tau  \\
1 & \tau & 13
\end{array}
\right)
$$
where $-8\leq \tau\leq 8$. Moreover, the discriminants $\disc(M_{-8})$, $\disc(M_{-7})$, ... , $\disc(M_{8})$ are $45$, $ 92$, $133$, $168$, $197$, $220$, $237$, $248$, $253$, $252$, $245$, $232$, $213$, $188$, $157$, $120$, $77$ respectively.
\end{thm}

\begin{proof}
We follow Algorithm \ref{alg:int}. Let $d_1=20$ and $d_2=38$. Then $n_1=3$, $n_2=6$, and as in case (i) of step (1), $$M_{\tau}=\left(
\begin{array}{cccc} 
3 & 1 & 1 \\
1& 7 & \tau  \\
1 & \tau & 13
\end{array}
\right)$$ and $\tau\in \mathbb{Z}$. 
Thus, 
$$
\disc(M_\tau)=253+2\tau-3\tau^2,
$$ 
and we have $\mathcal{T}_1=\{\tau|\,\tau\in\mathbb{Z},\, -8\leq \tau\leq 9\}$. 
To find $\mathcal{T}_2$,  
we consider nonzero element $v:=x \mathfrak{o}_\tau+y\mu_{\tau}+z\nu_{\tau}$, $x, y, z\in \mathbb{Z}$.
Then 
\begin{eqnarray*}
(v.v)&=& 3x^{2}+7y^{2}+13z^{2}+2xy+2xz+2\tau yz\\
      &=& x^2+(x+y)^2+(x+z)^2+(6y^2+12z^2+2\tau yz)\geq 3
\end{eqnarray*}
for each integer $-8\leq \tau\leq 8$ (in fact, for these values of $\tau$, we have that $(6y^2+12z^2+2\tau yz)\geq 2$ if either $y$ or $z$ is nonzero). Note that if $\tau=9$ then $(v.v)=2$ for $x=0$, $y=-1$ and $z=1$. Thus, we have $\mathcal{T}_2=\{\tau|\,\tau\in\mathbb{Z},\, -8\leq \tau\leq 8\}$. 
To complete the proof of the theorem, we only need to show that $\mathcal{N}_\tau=\emptyset$ for all $\tau\in \mathcal{T}_2$. 

{\bf Cases $\tau\in \{-8,2\}$}: 
Suppose $\tau=-8$. 
Then $\disc(M_{-8})=45=3^{2}\cdot 5$. 
Suppose $N\in \mathcal{N}_{-8}$.  
Since $N$ is a nontrivial overlattice of $M_{-8}$, it follows that $[N:M_{-8}]=3$ and $\disc(N)=5$. 
Then the minimum of $N$ is less than $3$ (cf. \cite{Mo44}), contradiction. 
Thus, $\mathcal{N}_{-8}=\emptyset$. Similarly, $ \mathcal{N}_{\tau}=\emptyset$ for $\tau=2$.

{\bf Cases $\tau\in \{\pm 7,\pm 5,\pm 3,-1\}$}: 
Suppose $\tau=-7$.  
Then $\disc(M_{-7})=92=2^{2}\cdot 23$. 
Suppose $N\in \mathcal{N}_{-7}$.  
Then $[N:M_{-7}]=2$ and $\disc(N)=23$, which contradicts Lemma \ref{lem:detMrank3}. 
Thus, $\mathcal{N}_{-7}=\emptyset$. 
Similarly, $\mathcal{N}_{\tau}=\emptyset$ for $\tau=7,\pm 5,\pm 3,-1$.

{\bf Cases $\tau\in \{\pm 6,\pm 4,-2,0,8\}$}: 
Suppose $\tau=-6$. 
Then $\disc(M_{-6})=92=7\cdot19$, a square free integer, which implies $\mathcal{N}_{-6}=\emptyset$. 
Similarly, $\mathcal{N}_{\tau}=\emptyset$ for $\tau=6,\pm 4,-2,0,8$.

{\bf Case $\tau=1$}: 
This is the most complicated case. 
Note that $\disc(M_{1})=252=2^{2} \cdot 3^{2}\cdot 7$. 
Suppose $N\in \mathcal{N}_{1}$.  
Then $[N:M_{1}]=2$, $3$ or $6$, and $\disc(N)=63$, $28$, or $7$ respectively. 
By Lemma \ref{lem:detMrank3}, $\disc(N)$ cannot be $63$ or $7$. 
Then $[N:M_{1}]=3$ and $\disc(N)=28$. 
Let $\eta:=\frac{2}{3}\mathfrak{o}_{1}-\frac{1}{9}\mu_{1}+\frac{1}{9}\nu_1\in M_{1}^{\vee}$. 
Then $G(M_{1})_3=\langle \bar{\eta}\rangle=\mathbb{Z}/9\mathbb{Z}$, 
and $b_{G(M_{1})}(\bar{\eta},\bar{\eta})=\frac{5}{9}\in \mathbb{Q}/\mathbb{Z}$. 
Note that $G(M_{1})_3$ has exactly $2$ nonzero isotropic vectors $3\bar{\eta}$ and $6\bar{\eta}$. 
Then $(\mathfrak{o}_1,\mu_1, 3\eta)$ is a basis of $N$, and $N$ has a root $-2\mathfrak{o}_1+3\eta$, a contradiction. Thus, $\mathcal{N}_{1}=\emptyset$.
\end{proof}

\begin{rem}\label{rem:Bound}
Using the steps (1)-(3) of  Algorithm \ref{alg:int}, we can give lower bounds for the number of irreducible components of $\mathcal{C}_{d_1}\cap \mathcal{C}_{d_2}$.

Case (i): $d_1=6n_1+2$, $d_2=6n_2+2$, $n_{2}>n_{1}\geq 1$. 
Then $M_{\tau}$ is positive definite if and only if 
its discriminant $12n_{1}n_{2}+4n_{1}+4n_{2}+1+2\tau-3\tau^{2}>0$,
i.e., 
$$
\frac{1-\sqrt{36n_{1}n_{2}+12n_{1}+12n_{2}+4}}{3}
< \tau <
\frac{1+\sqrt{36n_{1}n_{2}+12n_{1}+12n_{2}+4}}{3}.
$$
We set $A:=\sqrt{36n_{1}n_{2}+12n_{1}+12n_{2}+4}$, $B:=\sqrt{36n_{1}n_{2}+18n_{1}+12n_{2}+6}$, $D:=\lfloor \frac{B}{3}-1\rfloor $. Note that 
$$
\frac{1-B}{3}<\frac{1-A}{3}<1-\frac{B}{3}<0<\frac{B}{3}-1<\frac{1+A}{3}<\frac{1+B}{3},
$$ and if $|\tau|\leq \frac{B}{3}-1$, then $M_{\tau}$ has no roots. 
Then the set $\mathcal{T}_2$ is equal to one of the following four sets (depending on the values of $n_1$ and $n_2$): $\{-D,-D+1,..., D\},\; \{-D-1,-D,...,D\}$,  $\{-D,-D+1,...,D+1\}$, $\{-D-1,-D,...,D+1\}$. In particular, 
$$
2D+1\le |\mathcal{T}_2|\le 2D+3,
$$ and the intersection $\mathcal{C}_{6n_1+2}\cap \mathcal{C}_{6n_2+2}$ has at least $2D+1$ irreducible components.

Case (ii): $d_1=6n_1+2$, $d_2=6n_2$, $n_1\geq 1$, $n_2\geq 2$. 
Then $M_{\tau}$ is positive definite if and only if 
$\disc(M_{\tau})=12n_{1}n_{2}+4n_{2}-3\tau^{2}>0$. 
We set $D:=\lfloor\sqrt{4n_1n_2+\frac{4n_2}{3}}-1\rfloor$. Then the set $\mathcal{T}_2$ is equal to one of the following two sets: $\{0,1,..., D\},\; \{0,1,...,D+1\}$. In particular, 
$$
D+1\le |\mathcal{T}_2|\le D+2,
$$ 
and $\mathcal{C}_{6n_1+2}\cap \mathcal{C}_{6n_2}$ has at least $D+1$ irreducible components.

Case (iii): $d_1=6n_1$, $d_2=6n_2$, $n_2>n_1\geq 2$. 
Then $M_{\tau}$ is positive definite if and only if 
$\disc(M_{\tau})=3(4n_{1}n_{2}-\tau^{2})>0$. 
We set $D:=\lfloor2\sqrt{n_1n_2}-1\rfloor$. 
Then the set $\mathcal{T}_2$ is equal to one of the following two sets: $\{0,1,..., D\},\; \{0,1,...,D+1\}$. 
In particular, 
$$
D+1\le |\mathcal{T}_2|\le D+2,
$$ 
and $\mathcal{C}_{6n_1}\cap \mathcal{C}_{6n_2}$ has at least $D+1$ irreducible components.
\end{rem}

The following theorem shows that new irreducible components might be found by computing overlattices in the step (4) of Algorithm \ref{alg:int}.

\begin{thm}\label{C20intC12}
The intersection $\CC_{20}\cap \CC_{12}$ has exactly $6$ irreducible components 
which are given by the following rank $3$ lattices
$$
M_{\tau}:=
\left(
\begin{array}{cccc} 
3 & 1 & 0 \\
1& 7 & \tau  \\
0 & \tau & 4
\end{array}
\right)
\;\;
\textrm{and} 
\;\;
N_0:=
\left(
\begin{array}{cccc} 
3 & 1 & 1 \\
1& 3 &-3  \\
1 & -3 & 7
\end{array}
\right)
$$
where $0\leq \tau \leq 4$ and $N_{0}$ is an overlattice of $M_{0}$. Moreover, the discriminants $\disc(M_0)$, ... , $\disc(M_4)$, $\disc(N_0)$ are $80$, $77$, $68$, $53$, $32$, $20$ respectively.
\end{thm}

\begin{proof}
Again we follow Algorithm \ref{alg:int}. Let $d_1=20$ and $d_2=12$. Then $n_1=3$, $n_2=4$, and as in case (ii) of step (1),
$$
M_{\tau}=
\left(
\begin{array}{cccc} 
3 & 1& 0 \\
1& 7 & \tau  \\
0 & \tau & 4
\end{array}
\right),
$$
where $\tau\ge 0$ and $\tau\in\mathbb{Z}$. Then $\disc(M_\tau)=80-3\tau^2$, $\mathcal{T}_1=\{0,1,2,3,4,5\}$, and $\mathcal{T}_2=\{0,1,2,3,4\}$. Note that if $\tau=5$ and $v=\mu_\tau-\nu_\tau$, then $(v.v)=1<3$ and $M_\tau$ contains a root (see Lemma \ref{non-empty-criterion}). Like in the proof of Theorem \ref{C20intC38}, one can show that $\mathcal{N}_{\tau}=\emptyset$ for $\tau\in \{1,2,3,4\}$. Then we only need to consider the case $\tau=0$. Suppose  $N\in\mathcal{N}_0$. Since $\disc(M_0)=80=2^4\cdot 5$, it follows that $\disc(N)=20$ and $[N:M_0]=2$. Let $\eta:=\frac{1}{4}\mathfrak{o}_0+\frac{1}{4}\mu_0$, $\theta:=\frac{1}{4}\nu_0$. Then $G(M_0)_2=\langle\overline{\eta},\overline{\theta} \rangle\cong (\mathbb{Z}/4\mathbb{Z})^{\oplus 2}$. 
Note that $G(M_0)_2$ has exactly three order $2$ isotropic vectors: $2\overline{\eta}$, $2\overline{\theta}$, $2\overline{\eta}+2\overline{\theta}$. Since $K_{\tau,1}$ and $K_{\tau,2}$ must be primitive sublattices of $N$, it follows that $N/M_0=\langle 2\overline{\eta}+2\overline{\theta} \rangle$. Then $(\mathfrak{o}_0,2\eta+2\theta-\mu_0,\mu_0)$ is a basis of $N_0:=N$ with the desired Gram matrix. Since the six lattices $M_{\tau}$ ($\tau=0,1,2,3,4$) and $N_0$ have different discriminants, we have that $\mathcal{M}^\prime=\{M_0,M_1,M_2,M_3,M_4,N_0\}$. Thus, $\CC_{20}\cap \CC_{12}$ has exactly 6 irreducible components.
\end{proof}

\subsection{New rational cubic fourfolds}
In \cite{FL20} a birational involution $\sigma_V$ of $\CC_{20}$ was produced via the Cremona transformation defined by the Veronese surface. 
By considering the action of $\sigma_V$ on the three intersections $\CC_{20}\cap \CC_{26}$, $\CC_{20}\cap \CC_{38}$ and $\CC_{20}\cap \CC_{42}$, Fan--Lai \cite[Theorem 3.13]{FL20} found examples of rational cubics which were unknown before parametrized by three divisors in $\CC_{20}$.
Similarly, using Theorem \ref{C20intC38}, 
we find new rational cubics parametrized by two other divisors in $\CC_{20}$:

\begin{cor}\label{new-rat-cubics}
Let $M^\prime_{-4}$ and $M_4^\prime$ be the positive definite lattices with the Gram matrices  
$$
\left(
\begin{array}{cccc} 
3 & 4 & 7 \\
4& 12 & -3 \\
7 & -3 & 49
\end{array}
\right)
\;
\textrm{and} 
\;
\left(
\begin{array}{cccc} 
3 & 4 & -1 \\
4& 12 & 5 \\
-1 & 5 & 17
\end{array}
\right)
$$
of determinants $197$ and $213$ respectively.
Then $\mathcal{C}_{M^\prime_{-4}}$ and $\mathcal{C}_{M_4^\prime}$ are two non-empty irreducible divisors in $\CC_{20}$. Moreover, $\mathcal{C}_{M^\prime_{-4}}$ and $\mathcal{C}_{M_4^\prime}$ parametrize rational cubic fourfolds and are not in $\CC_{d}$ for all $d\in \{8,14,18,26,38,42\}$.
\end{cor}

\begin{proof}
By Theorem \ref{mainthm-one},  $\mathcal{C}_{M^\prime_{-4}}$ and $\mathcal{C}_{M_4^\prime}$ are two non-empty irreducible divisors in $\CC_{20}$. 
Next we show that  $\mathcal{C}_{M^\prime_{-4}}$ and $\mathcal{C}_{M_4^\prime}$ parametrize rational cubics. 
The idea is the same as the proof of \cite[Theorem 3.13]{FL20}. Let $\tau=\pm 4$ and $M_{\tau}$ be the positive definite lattice  with a basis $(\mathfrak{o}_{\tau},\mu_{\tau},\nu_{\tau})$ such that the Gram matrix as in Theorem \ref{C20intC38}.
By Theorem \ref{C20intC38}, $\CC_{M_{\tau}}$ is an irreducible component of $\CC_{20}\cap\CC_{38}$. Note that $\CC_{M_{\tau}}\cap \CC_{8}\neq \CC_{M_{\tau}}$, and  the Gram matrix of $M_{\tau}$ with respect to the basis $(\mathfrak{o}_{\tau},\mathfrak{o}_{\tau}+\mu_{\tau},\nu_{\tau})$ is 
$$
A:=\left(
\begin{array}{cccc} 
3 & 4 & 1 \\
4& 12 & \tau+1  \\
1 & \tau+1 & 13
\end{array}
\right).
$$
Let $X_\tau$ be a very general cubic in $\CC_{M_{\tau}}$ such that $A(X_\tau)\cong M_{\tau}$. 
Then the Gram matrix of $A(X_\tau)$ with respect to a suitable basis $(h_{X_\tau}^2,\mu,\nu)$ (for some $\mu,\nu\in A(X_\tau)$) is $A$. 
By \cite[(3.12)]{FL20},  $A(X_\tau^\prime)$ of $X_\tau^\prime:=F_V(X_\tau)$ has a basis $(h^2_{X_\tau^\prime},\mu^\prime,\nu^\prime)$ with the following Gram matrix 
$$
A^\prime:=\left(\begin{array}{ccc}
3 & 4 & 3-\tau \\
4 & 12 & \tau+1 \\
3-\tau & \tau+1 & 13+(2-\tau)^2
\end{array}\right),
$$
which implies that $[X_\tau^\prime]\in \CC_{M_{\tau}^\prime}$ ($\tau=\pm 4$). 
Note that $X_\tau$ is rational since $[X_\tau]\in \CC_{38}$ (\cite{RS19a}). 
Then $X_\tau^\prime$ is rational. 
Therefore, a very general cubic in $\CC_{M_{\tau}^\prime}$ is rational. 
From \cite[Theorem 1]{KT19} we deduce that all cubics in $\CC_{M_{\tau}^\prime}$ are rational. 
By computation (cf. the proof of Theorem \ref{mainthm-two}), for $d\in \{8,14,18,26,38,42\}$, there exists no rank two primitive sublattice $K_d\subset A(X_\tau^\prime)$, $h^2_{X_\tau^\prime}\in K_d$, of discriminant $d$, which implies that $[X^\prime_\tau]\notin \CC_d$.
\end{proof}

\subsection{Intersection of all Hassett divisors} 
Let $\mathcal{Z}\subset \CC$ be the intersection of all Hassett divisors, 
i.e.,
$$ 
\mathcal{Z}=\bigcap_{d>6, \;d\equiv 0, 2\; (\mathrm{mod}\; 6)} \;\mathcal{C}_{d}.
$$
Note that all cubic fourfolds in $\mathcal{Z}$ are rational. By Theorem \ref{mainthm-two}, $\mathcal{Z}$ is non-empty. Suprisingly, the dimension ${\rm dim}(\mathcal{Z})$ turns out to be large.

\begin{thm}\label{subvar-Inter-all-HasD}
${\rm dim}(\mathcal{Z})\ge 13$.
\end{thm}

\begin{proof}
We consider a rank $8$ lattice 
$
M=\langle \mathfrak{o}, \delta_{1}, \delta_{2}, \delta_{3}, \delta_{4}, \delta_{5}, \delta_{6}, \delta_{7} \rangle
$
with the following Gram matrix:
$$
\left(
\begin{array}{cccccccc} 
3 & 1 & 1 & 0 & 0 & 0 & 0 & 1 \\
1 & 3 & 1 & 0 & 0 & 0 & -2 & 0 \\
1 & 1 & 3 & 0 & 0 & 0 & -2 & 0 \\
0 & 0 & 0 & 4 & 0 & 0 & 0 & 0  \\
0 & 0 & 0 & 0 & 4 & 0 & 0 & 0  \\
0 & 0 & 0 & 0 & 0 & 4 & 0 & 0  \\
0 & -2 & -2 & 0 & 0 & 0 &  6 & 0 \\
1 & 0 & 0 & 0 & 0 & 0 & 0 & 3 
\end{array}
\right).
$$
First of all,
by Sylvester's criterion,
the lattice $M$ is positive definite. By Lemma \ref{lem:eventest}, $\mathfrak{o}$ is a distinguished element of $M$.
For any nonzero element 
$$
v=x\mathfrak{o}+\sum_{i=1}^{7}y_{i}\delta_{i}
$$
of $M$, the value 
\begin{eqnarray*}
(v.v)&=&
3x^{2}+3y_{1}^{2}+3y_{2}^{2}+4y_{3}^{2}+4y_{4}^{2}+4y_{5}^{2}+6y_{6}^{2}+3y_{7}^{2}+2xy_{1}+2xy_{2}+2y_{1}y_{2} \\
&+& 2xy_{7}-4y_{1}y_{6}-4y_{2}y_{6}\\
&=&(x+y_{1})^{2}+(x+y_{2})^{2}+(x+y_{7})^{2}+4y_3^2+4y_4^2+4y_5^2+2y_7^2\\
&+& (2y_1^2+2y_2^2+6y_6^2+2y_1y_2-4y_1y_6-4y_2y_6)\\
&\geq& 3.
\end{eqnarray*}
Here we use the fact that the quadratic form $ 2y_1^2+2y_2^2+6y_6^2+2y_1y_2-4y_1y_6-4y_2y_6$ is positive definite.  
Note that $r(M)+l(M)\le r(M)+r(M)=16<20$. 
Based on the above observations, according to Theorem \ref{mainthm-one}, 
$\CC_{M}\subset \CC$ is a non-empty irreducible subvariety of codimension $7$.

Note that the primitive sublattice $\langle \mathfrak{o}, \delta_{1}+\delta_{7} \rangle \subset M$ is of discriminant $14$. 
Moreover, the Gram matrix of the primitive sublattice $\langle \mathfrak{o}, \delta_{1}-\delta_{2}, \delta_{3}, \delta_{4}, \delta_{5}, \delta_{6}, \delta_{7}\rangle\subset M$ is the same as the matrix (\ref{matrix}). According to the proof of Theorem \ref{mainthm-two}, for any integer $d>6$ and $d\equiv 0,2\; ({\rm mod}\; 6)$, there exists $v\in M$ such that $\langle \mathfrak{o},v\rangle\subset M$ is a rank two primitive sublattice of discriminant $d$. Therefore, $\CC_M\subset \mathcal{Z}$. Then ${\rm dim}(\mathcal{Z})\ge {\rm dim}(\CC_M)=13$. This completes the proof of Theorem \ref{subvar-Inter-all-HasD}.
\end{proof}

\begin{quest}
What is the dimension of $\mathcal{Z}$?
\end{quest}

\begin{rem}
According to Proposition \ref{prop:int} and Remark \ref{alg:int-rem} (c), 
with the aid of computer, 
we can obtain that $\mathcal{C}_{8}\cap \mathcal{C}_{12}\cap \mathcal{C}_{18}$ has exactly $38$ irreducible components which are given by $38$ rank-four lattices.
Moreover, those irreducible components are not contained in $\mathcal{Z}$.
Combining with Theorem \ref{subvar-Inter-all-HasD}, we have $13\leq \dim(\mathcal{Z})\leq 16$.
Based on computer experiments, 
we speculate that the dimension of $\mathcal{Z}$ is $16$.
\end{rem}

By Proposition \ref{prop:int}, there exist positive definite lattices $M_i$, $i=1,...,n$ containing distinguished elements $\mathfrak{o}_i$ with $\CC_{M_i}\neq \emptyset$ such that $\mathcal{Z}=\bigcup_{i=1}^{n}\CC_{M_i}$ and for any $i\neq j$, there exists no primitive embeding $\varphi: M_{i} \hookrightarrow M_{j}$ with $\varphi(\mathfrak{o}_i)=\mathfrak{o}_j$. It is an interesting challenge to find these lattices.


\section{Admissible lattices}\label{sec:admissible}
 In this section, we introduce the notion of admissible lattices (Definition \ref{def:admissiblelattice}) and obtain a new criterion for admissibility (Theorems \ref{thm:admM1}, \ref{thm:admM2}). 
 As an illustration, we give explicit description of rank $3$ admissible lattices (Propositions \ref{prop:admMrank3-1}, \ref{prop:admMrank3-2}, Corollaries \ref{cor:C8divisors}, \ref{cor:C18divisors}). 
Moreover, we find infinitely many examples of rank $11$, the maximal value, nonadmissible lattices $M$ which can be realized as the algebraic cohomologies of cubic fourfolds (Corollary \ref{cor:anwserLaza}). 
Existence of such $M$ was asked by Laza \cite{Laz18}.

\begin{defn}\label{def:admissiblelattice}
Let $M$ be a positive definite lattice of rank $r(M)\ge 2$ with a distinguished element $\mathfrak{o}$. We say $M$ is an {\it admissible} lattice if $M$ has a rank $2$ primitive sublattice $K$, $\mathfrak{o}\in K$ such that the discriminant $\disc(K)$ satisfies the property \eqref{twostar}.
\end{defn}

Admissible lattices are closely related to geometry of lattice polarizable cubic fourfolds. 
According to Hassett \cite{Has00}, a cubic fourfold has a Hodge-theoretically associated K3 surface if and only if its algebraic cohomology is an admissible lattice.  
Kuznetsov's conjecture means that the rationality of a cubic fourfold is equivalent to the admissibility of its algebraic cohomology. 
Rank $2$ admissible lattices were classified by Hassett \cite{Has00}. 
However, classification of admissible lattices of higher rank is unknown. 
In order to characterize the admissibility of higher rank lattices in terms of some numerical conditions,
we will intensively use the notion of integral quadratic forms (in fact, the basic idea has been adopted in the proof of Theorem \ref{mainthm-two}).

Next we recall some basics on integral quadratic form for later use. 
By an {\it integral quadratic form} $f=f(x_1,...,x_n)$ in $n$ variables $x_1,...,x_n$ we shall mean a function
$$
f(x_1,...,x_n)=\sum_{i=1}^{n}a_{ii}x_i^{2}+\sum_{1\le i<j\le n}a_{ij}x_{i}x_{j}
$$
where $a_{ij}\in\mathbb{Z}$ $(1\le i\le j\le n)$. 
We say that $f$ is {\it primitive} if the greatest common divisor (gcd) of the coefficients $a_{ij}$ $(1\le i\le j\le n)$ is $1$. 
If $f(x_{1},...,x_{n})>0$ whenever at least one of $x_{1},...,x_{n}$ is nonzero, we call $f$ is {\it positive definite}. 
An integer $d$ is {\it represented} by a form $f$ if $f(y_{1},...,y_{n})=d$ for some $(y_{1},...,y_{n})\in \mathbb{Z}^n$. 
If ${\rm gcd}(y_{1},...,y_{n})=1$, we say that $d$ is {\it properly represented} by $f$. 
Two forms $f(x_{1},...,x_{n})$ and $g(x_{1},...,x_{n})$ are {\it equivalent} if there is an matrix $B=(b_{ij})\in{\rm GL}(n,\Z)$ with $\det(B)=\pm 1$ such that 
$$
f(\sum_{i=1}^nb_{1i}x_{i},...,\sum_{i=1}^nb_{ni}x_{i})=g(x_{1},...,x_{n}).
$$

\begin{defn}\label{associated-form}
Let $M$ be a positive definite lattice of rank $r\ge 2$ with a distinguished element $\mathfrak{o}$. Let $\mathbf{b}:=(\mathfrak{o}, \mu_1,...,\mu_n)$ be a basis of $M$, where $n=r-1$. 
We say the positive definite integral quadratic form
$$
f_{(M,\mathbf{b})}(x_1,...,x_n):=\disc (\langle \mathfrak{o}, \sum_{i=1}^{n}x_i \mu_i\rangle)
$$
is the {\it associated form} of the pair $(M, \mathbf{b})$.
\end{defn} 
Note that if $\mathbf{b}^\prime:=(\mathfrak{o}, \mu_1^\prime,...,\mu_n^\prime)$ is another basis of $M$, then the forms $f_{(M,\mathbf{b})}(x_1,...,x_n)$ and $f_{(M,\mathbf{b}^\prime)}(x_1,...,x_n)$ are equivalent.

\begin{lem}\label{lem:Mandform}
Let $M$ be a positive definite lattice of rank $r\ge 2$ with a basis $\mathbf{b}:=(\mathfrak{o},\mu_1,...,\mu_n)$ such that $\mathfrak{o}$ is a distinguished element, where $n=r-1$. 
Let $f:=f_{(M,\mathbf{b})}(x_1,...,x_n)$ be the associated form of $(M,\mathbf{b})$. 
Then the following statements are equivalent:
\begin{enumerate}
\item The lattice $M$ is admissible.
\item The form $f$ properly represents an integer satisfying  \eqref{twostar}.
\item The form $f$ represents an integer satisfying  \eqref{twostar}.
\end{enumerate}
\end{lem}

\begin{proof}
The equivalence between (1) and (2) follows from the definition of $f_{(M,\mathbf{b})}(x_1,...,x_n)$. Obviously, (2) implies (3). 
Conversely, suppose (3) holds. 
Then there exists $(y_1,...,y_n)\in \Z^n$ such that $d:=f(y_1,...,y_n)$ satisfies  \eqref{twostar}. 
Let $c:={\rm gcd}(y_1,...,y_n)$. 
Then 
$$
f(\frac{y_1}{c},...,\frac{y_n}{c})=\frac{d}{c^2}
$$ 
also satisfies \eqref{twostar}, which implies (2).
\end{proof}

The following number theoretical result will play an important role in the proof of Theorems \ref{thm:admM1}, \ref{thm:admM2}.

\begin{prop}\label{prop:representprime}
Let $f=f(x_1,...,x_n)$ ($n\ge 2$) be a primitive positive definite integral quadratic form. 
Suppose  $f(y_1,...,y_n)\equiv 1\;  (\mathrm{mod}\, 3)$ for some $(y_1,...,y_n)\in \mathbb{Z}^n$. 
Then $f$ represents a prime $p \equiv 1\;  (\mathrm{mod}\, 3)$.
\end{prop}

\begin{proof}
Let $u:=(y_1,...,y_n)\in \Z^n$ and let $a:=f(u)$. 
By the assumption, we have $a\equiv 1\;(\mathrm{mod}\, 3)$. 
Consider the prime factorization 
$$
a= p_{1}^{m_{1}}\cdots p_{s}^{m_{s}}.
$$
Since the form $f$ is primitive, it follows that for any $1\le i\le s$, there exists $(b_{i1},...,b_{in})\in \Z^n$ such that $p_i\nmid f(b_{i1},...,b_{in})$. 
By Chinese remainder theorem, for each $1\le j\le n$, there exists $z_j\in \Z$ such that $z_j\equiv b_{ij} \;  (\mathrm{mod}\, p_i)$, $i=1,...,s$. Let $v:=(z_1,...,z_n)$ and let $c:= f(v)$. 
Then we have $p_i \nmid c$ for all $1\le i\le s$. 
Thus, $a$ and $c$ are coprime. 
Let $b:=f(u+v)-f(u)-f(v)$. 
Since ${\rm gcd}(a,9c)=1$, the form  
$$
g(x,y):=f(xu+3yv)=ax^{2}+3bxy +9cy^{2}
$$ 
is a primitive positive definite integral quadratic form in two variables $x$ and $y$. 
Then by \cite[Theorem 9.12]{Cox89}, $g$ represents infinitely many prime numbers. Note that $g(x,y)\equiv 0,1\;  (\mathrm{mod}\, 3)$ for any $(x,y)\in \Z^2$. 
Thus, there exists $(d,e)\in\Z^2$ such that $f(du+3ev)=g(d,e)$ is a prime $p \equiv 1\;  (\mathrm{mod}\, 3)$.
\end{proof}

\begin{rem}
Addington--Thomas \cite[Proposition 3.3]{AT14} proved case $n=2$ of Proposition \ref{prop:representprime} based on \cite[Theorem 9.12]{Cox89}.  
It is also a key ingredient in the proof of \cite[Theorem 3.1]{AT14} characterizing the admissibility for $A(X)$ of a cubic fourfold $X$ via the algebraic Mukai lattice of $\mathcal{A}_{X}$ (cf. Remark \ref{more--AT-rem}).
\end{rem}

The following elementary observation on the lattice invariant $\ell(N)$ will be useful.

\begin{lem}\label{lem:rl}
Let $N$ be a rank $r$ lattice.
Then there is a unique (up to isometry) rank $r$ lattice $N_0$ such that $\ell(N_0)<r$ and $N\cong N_0(m)$ for some positive integer $m$. Moreover, $m$ is characterized by the following equivalent conditions:
\begin{enumerate}
\item $m$ is the greatest positive integer such that $m | (u.v)$ for all $u,v\in N$.
\item $m$ is equal to the greatest common divisor of all entries $a_{ij}$ for any Gram matrix $A=(a_{ij})_{1\le i,j\le r}$ of $N$.
\item $m=s_r$, where  $G(N)\cong \Z/s_1\Z\oplus ...\oplus \Z/s_{r} \Z$ for some positive integers $s_i$ ($1\le i\le r$) with $s_{i+1} | s_{i}$, $i=1,...,r-1$.
\end{enumerate} 
\end{lem}

\begin{proof}
Since both $N$ and its dual $N^\vee$ are free $\Z$-modules of rank $r$, it follows that there is a $\Z$-basis $\alpha_1,...,\alpha_r$ of $N^\vee$ and positive integers $s_1,...,s_r$ with $s_{j+1}|s_j$ ($1\le j<r$) such that $s_1\alpha_1,...,s_r\alpha_r$ is a $\Z$-basis of $N$ (cf. \cite[Chapter 11, Theorem 5.1]{Cas78}). Then $G(N)\cong \Z/s_1\Z\oplus ...\oplus \Z/s_{l} \Z$ and $s_j=1$ for $j\ge l+1$,
where $l:=\ell(N)$. From these observations, one can obtain equivalence among (1), (2), (3) by direct computation. Let $N_0$ be the same free $\Z$-module as $N$ but with the bilinear form $(-.-)_{N_0}:=\frac{1}{m}(-.-)_{N}$. Then $N_0$ is a well-defined rank $r$ lattice with $\ell(N_0)<r$ and $N\cong N_0(m)$.
\end{proof}

\begin{set}\label{setup}
Let $M$ be a positive definite lattice of rank $r\ge 2$ with a basis $\mathbf{b}:=(\mathfrak{o},\mu_1,...,\mu_n)$ such that $\mathfrak{o}$ is a distinguished element, where $n=r-1$. 
Let $f:=f_{(M,\mathbf{b})}(x_1,...,x_n)$ be the associated form of $(M,\mathbf{b})$. 
We may write 
\begin{equation}\label{eq:fandg}
f(x_1,...,x_n)=\lambda g(x_1,...,x_n),
\end{equation}
where $\lambda$ is the greatest common divisor of all the coefficients of the form $f$ and $g(x_1,...,x_n)$ is a primitive integral quadratic form. 
Let $N:=\langle \mathfrak{o}\rangle^\perp_M$. 
By Lemma \ref{lem:rl}, there is a positive integer $m$ and a lattice $N_0$ such that $N\cong N_0(m)$ and $\ell(N_0)<n$. 
We set $\displaystyle m^\prime:=\frac{m}{3^a}$, where $a$ is the non-negative integer such that $3^a |m$ but $3^{a+1}\nmid m$.
\end{set}

Our numerical criterion (Theorem \ref{mainthm-intro-3}) for the admissibility of a positive definite lattice $M$ is divided into two cases:
(i) $M\neq \langle \mathfrak{o}\rangle \oplus N$;  (ii) $M= \langle \mathfrak{o}\rangle \oplus N$.

\begin{thm}\label{thm:admM1}
Under Setup \ref{setup}, we suppose that $M\neq \langle \mathfrak{o}\rangle \oplus N$. 
Then the following statements are equivalent:
\begin{enumerate}
\item The lattice $M$ is admissible.
\item The form $f$ represents an integer satisfying \eqref{twostar}.
\item The integer $\lambda$ satisfies \eqref{twostar}.
\end{enumerate}
Moreover, $\lambda=2m^\prime$ (resp. $\lambda=m^\prime$) if the lattice $N_0$ is even (resp. odd).
\end{thm}    

\begin{proof}
Let $A=(a_{ij})$ be the Gram matrix of $M$ with respect to the basis $\mathbf{b}=(\mathfrak{o},\mu_{1},\cdots,\mu_{n})$. 
Without loss of generality, 
by Corollary \ref{cor:rank2},
we may assume that $a_{12}=a_{21}=1$ and $a_{1j}=a_{j1}=0$ ($j=3,\cdots,r$),
i.e.,
\begin{equation}\label{essential-matrix-one}
A=\begin{pNiceMatrix}
3 & 1&  0& \cdots  &0 \\
1 & a_{22} & a_{23}&\cdots  &a_{2r} \\
0 & a_{32} & a_{33} & \ddots & \vdots  \\
\vdots & \vdots & \vdots  & \ddots & a_{nr} \\ 
0 & a_{r2}& \cdots &a_{rn} & a_{rr} 
\end{pNiceMatrix},
\end{equation}
where $a_{22}$ is odd and $a_{jj}$ is even ($j=3,\cdots,r$).
Then the associated form
\begin{equation}\label{eq:f1}
f=f_{(M,\mathbf{b})}(x_1,...,x_n)=(3a_{22}-1)x_1^2+3(\sum_{i=3}^r a_{ii}x_{i-1}^2+\sum_{2\le i<j\le r}2a_{ij}x_{i-1}x_{j-1}).
\end{equation}
Thus, we have 
\begin{equation}\label{eq:lambda}
\lambda={\rm gcd}(3a_{22}-1, 3a_{33},...,3a_{rr},6a_{23},...,6a_{2r},6a_{34},...,6a_{3r},...,6a_{nr}).
\end{equation}
Note that $\mathbf{b}^\prime:=(3\mu_1-\mathfrak{o}, \mu_2,...,\mu_r)$ is basis of $N$, and the Gram matrix is
\begin{equation}\label{matrix-one-cor}
\begin{pNiceMatrix}
9a_{22}-3 & 3a_{23}&\cdots  &3a_{2r} \\
3a_{32} & a_{33} & \cdots & a_{3r} \\
\vdots & \vdots  & \ddots & \vdots  \\ 
3a_{r2}& a_{r3} & \cdots  & a_{rr} 
\end{pNiceMatrix}.
\end{equation}
By Lemma \ref{lem:rl}, the integer $m$ is equal to the greatest common divisor of all entries of the Gram matrix \eqref{matrix-one-cor} of $N$ with respect to $\mathbf{b}^\prime$, i.e., 
\begin{equation}\label{eq:m}
m={\rm gcd}(9a_{22}-3,a_{33},...,a_{rr},3a_{23},...,3a_{2r},a_{34},...,a_{3r},...,a_{nr}).
\end{equation}
Obviously $3\nmid \lambda$. 
Then from (\ref{eq:lambda}) and (\ref{eq:m}), one obtain that $\lambda=2m^\prime$ (resp. $\lambda=m^\prime$) if the lattice $N_0$ is even (resp. odd). 
Since $a_{22}$ is odd and $a_{33},...,a_{rr}$ are even,
it follows that $\lambda$ is even. 

Now suppose (3) holds. 
Then $\lambda\equiv 2 \, ({\rm mod}\, 3)$. 
By (\ref{eq:fandg}) and (\ref{eq:f1}), we have 
$$
\lambda g(1,0,...,0)=f(1,0,...,0)=3a_{22}-1,
$$
which implies $g(1,0,...,0)\equiv 1 \, ({\rm mod}\, 3)$. 
Thus, by Proposition \ref{prop:representprime}, the form $g$ represents a prime $p\equiv 1 \, ({\rm mod}\, 3)$. 
Then the form $f$ represents the integer $\lambda p$ which satisfies  \eqref{twostar}. 
Therefore, (3) implies (2). 
Conversely, it is clear that (2) implies (3). 
By Lemma \ref{lem:Mandform}, (1) and (2) are equivalent.                                                                                                                                                                                                                                                                                                                                                                                                                                                                                                                                                                     \end{proof}

\begin{thm}\label{thm:admM2}
Under Setup \ref{setup}, we suppose that $M= \langle \mathfrak{o}\rangle \oplus N$. 
Then the following statements are equivalent:

\begin{enumerate}
\item The lattice $M$ is admissible.
\item The form $f$ represents an integer satisfying  \eqref{twostar}.
\item The integer $\lambda$ satisfies  \eqref{twostar} and $g(y_1,...,y_n)\equiv 1\;  (\mathrm{mod}\, 3)$ for some $(y_1,...,y_n)\in \mathbb{Z}^n$.
\end{enumerate}
Moreover, $\lambda=6m$ (resp. $\lambda=3m$) if the lattice $N_0$ is even (resp. odd).
\end{thm}    

\begin{proof}
The proof is similar to that of Theorem \ref{thm:admM1}. 
Let $A=(a_{ij})$ be the Gram matrix of $M$ with respect to the basis $\mathbf{b}=(\mathfrak{o},\mu_1,...,\mu_n)$. 
We may assume that $a_{1j}=a_{j1}=0$ ($j=2,...,r$),
namely, 
\begin{equation}\label{essential-matrix-two}
A=\begin{pNiceMatrix}
3 & 0 &  0& \cdots  &0 \\
0 & a_{22} & a_{23}&\cdots  &a_{2r} \\
0 & a_{32} & a_{33} & \ddots & \vdots  \\
\vdots & \vdots & \vdots  & \ddots & a_{nr} \\ 
0 & a_{r2}& \cdots &a_{rn} & a_{rr} 
\end{pNiceMatrix},
\end{equation}
where $a_{jj}$ is even ($j=2,3,\cdots,r$); or equivalently, $\mathbf{b}^\prime:=(\mu_1,...,\mu_n)$ is a basis of $N$.
Then the associated form of $(M, \mathbf{b})$
\begin{equation}\label{eq:f2}
f=f_{(M,\mathbf{b})}(x_1,...,x_n)=3(\sum_{i=2}^r a_{ii}x_{i-1}^2+\sum_{2\le i<j\le r}2a_{ij}x_{i-1}x_{j-1}).
\end{equation}
Thus, we have 
\begin{equation}\label{eq:lambda2}
\lambda=3{\rm gcd}(a_{22}, a_{33},...,a_{rr},2a_{23},...,2a_{2r},2a_{34},...,2a_{3r},...,2a_{nr})
\end{equation}
and
\begin{equation}\label{eq:m2}
m={\rm gcd}(a_{22}, a_{33},...,a_{rr},a_{23},...,a_{2r},a_{34},...,a_{3r},...,a_{nr}).
\end{equation}
Then from (\ref{eq:lambda2}) and (\ref{eq:m2}), one obtain that $\lambda=6m$ (resp. $\lambda=3m$) if the lattice $N_0$ is even (resp. odd). 

Suppose (2) holds. 
Then the form $f$ represents an integer $d$ with the property  \eqref{twostar}. 
Thus, the form $g$ represents the integer $\displaystyle \frac{d}{\lambda}$. 
Since $a_{jj}$ is even ($j=2,3,\cdots,r$), it follows that $6|\lambda$. 
Hence, $\displaystyle 2\nmid \frac{d}{\lambda}$, $\displaystyle 3\nmid \frac{d}{\lambda}$, $\displaystyle p \nmid \frac{d}{\lambda} \text{ for any odd prime } p \equiv 2\;  (\mathrm{mod}\, 3)$. 
Then $\displaystyle \frac{d}{\lambda}\equiv 1\;  (\mathrm{mod}\, 3)$. 
Therefore, (2) implies (3). Again, the other implications are consequences of Proposition \ref{prop:representprime} and Lemma \ref{lem:Mandform}.
 \end{proof}

\begin{rem}\label{more--AT-rem}
(1) Let $T\subset\Lambda$ be a primitive even sublattice of signature $(20-n,2)$ such that $T\subset \langle h^2 \rangle^\perp_{\Lambda}$. By \cite[Theorem 3.1]{AT14}, $M$ is admissible if and only if $M^\prime$ contains the lattice $U$ (the lattice $\langle h^2 \rangle^\perp_{\Lambda}$ admits a unique primitive embedding into ${\rm II}_{20,4}$, which in turn gives the even lattice $M^\prime:=T^\perp_{{\rm II}_{20,4}}$ of signature $(n,2)$). Note that our Theorems \ref{thm:admM1}, \ref{thm:admM2} give a new criterion of admissibility of $M$ without referring to $M^\prime$, 
and in principle provide an approach for handling \cite[Question 2.1]{Laz18}.  

(2) Let $\mathcal{A}_X$ denote the Kuznetsov component of a cubic fourfold $X$. 
In \cite{AT14} the condition $[X]\in \CC_d$ with $d$ satisfying \eqref{twostar} was shown to be equivalent to the existence of a Hodge isometry between the Mukai lattice of $\mathcal{A}_{X}$ and that of some K3 surface $S$;
it is also equivalent to that the Fano variety of lines is birational to a moduli space of stable sheaves a K3 surface (see \cite{Add16}).
As a twisted version, Huybrechts \cite[Theorem 1.3]{Huy17} proved that the Mukai lattice of $\mathcal{A}_{X}$ is Hodge isometric to the twisted Mukai lattice of a twisted K3 surface $(S,\alpha)$ if and only if $[X]\in \CC_{d}$ with
\begin{equation}\tag{$**^\prime$}
d=k^2d_0, \text{ where } k, d_0 \text{ are integers and }  d_0 \text{ satisfies } \eqref{twostar}.
\end{equation}
It is of interest to find a numerical criterion for the condition ($**^\prime$) analogous to Theorems \ref{thm:admM1}, \ref{thm:admM2}.
\end{rem}

Let $M$ be a positive definite lattice containing a distinguished element $\mathfrak{o}$.
Noticing that, for any rank $r(M)$, 
Theorems \ref{thm:admM1}, \ref{thm:admM2} can be rephrased in terms of Gram matrix of $M$. 
From the proof of Theorems \ref{thm:admM1} and \ref{thm:admM2},
we can always choose a basis of $M$ with the Gram matrix as \eqref{essential-matrix-one} or \eqref{essential-matrix-two}.
Taking rank $3$ lattices for illustration, we have the following

\begin{prop}\label{prop:admMrank3-1}
Let $M$ be a rank $3$ positive definite lattice with a distinguished element $\mathfrak{o}$. 
Suppose $M\neq \langle\mathfrak{o} \rangle\oplus \langle\mathfrak{o} \rangle^\perp_M$. Then the lattice $M$ has a basis $(\mathfrak{o},\mu_1,\mu_2)$ with the Gram matrix of the following shape
$$
A=\left(
\begin{array}{cccc} 
3 & 1 & 0 \\
1& 2n_{1}+1 & \tau  \\
0 & \tau & 2n_{2}
\end{array}
\right),
$$ 
where $n_1\ge 0$, $\tau\ge 0$, $n_2>0$ are integers. Moreover, the lattice $M$ is admissible if and only if $\lambda:={\rm gcd}(6n_1+2, 6\tau, 6n_2)$ satisfies  \eqref{twostar}.
\end{prop}

\begin{prop}\label{prop:admMrank3-2}
Let $M$ be a rank $3$ positive definite lattice with a distinguished element $\mathfrak{o}$. Suppose $M= \langle\mathfrak{o} \rangle\oplus \langle\mathfrak{o} \rangle^\perp_M$. Then the lattice $M$ has a basis $(\mathfrak{o},\mu_1,\mu_2)$ with the Gram matrix of the following shape
$$
A=\left(
\begin{array}{cccc} 
3 & 0 & 0 \\
0& 2n_{1} & \tau  \\
0 & \tau & 2n_{2}
\end{array}
\right),
$$ 
where $n_1> 0$, $\tau\ge 0$, $n_2>0$ are integers. 
Moreover, the lattice $M$ is admissible if and only if the following conditions hold
\begin{enumerate}
\item $\lambda:={\rm gcd}(6n_1, 6\tau, 6n_2)$ satisfies  \eqref{twostar}, and 
\item the mod $3$ reduction of the polynomial $\displaystyle \frac{6n_1}{\lambda}x^2+\frac{6\tau}{\lambda}x y+\frac{6n_2}{\lambda}y^2\in \Z[x,y]$ is not equal to $2x^2$, $2y^2$, $2x^2\pm 2xy+2y^2$. 
\end{enumerate}
\end{prop}

\begin{rem}
(1) The moduli spaces $\CC_d$ of special cubic fourfolds with $d$ satisfying \eqref{twostar} are closely related to the moduli spaces of polarized K3 surfaces (see \cite{Has00}). 
For higher rank admissible lattices $M$, it is interesting to explore analogous relation between the moduli spaces $\CC_M$ of $M$-polarizable cubic fourfolds and the moduli spaces of lattice polarized K3 surfaces. 

(2) The admissibility of a rank-two positive definite lattice is determined by its discriminant.
For higher rank lattices, however, the discriminant is not enough.
For example, by Propositions \ref{prop:admMrank3-1} and \ref{prop:admMrank3-2},
we can find an admissible rank 3 lattice $M_{(1,2,2)}$ and a nonadmissible rank 3 lattice $M_{(2,0,1)}$ with the same discriminant $72$ and $\ell(M_{(1,2,2)})=\ell(M_{(2,0,1)})=2$ (see Corollary \ref{cor:C18divisors});
compare with \cite{Laz18}.
Moreover, for any concrete positive definite lattice, 
our Theorems \ref{thm:admM1} and \ref{thm:admM2} can effectively decide the admissibility. 
\end{rem}

As two examples, we classify all divisors $\CC_M$ in $\CC_8$ and $\CC_{18}$.

\begin{cor}\label{cor:C8divisors}
Let $\mathcal{S}_8\subset \Z^2$ be the set of pairs $(\tau,n)$ of integers satisfying 
\begin{equation}\label{eq:C8divisors}
0\le \tau\le 4,\, n\ge 2,\, (\tau, n)\neq (3,2), (4,2), (4,3).
\end{equation}
For each $(\tau,n)\in\mathcal{S}_8$, we denote by $M_{(\tau,n)}$ the rank $3$ positive definite lattice having a basis $(\mathfrak{o}_{(\tau,n)},\mu_{(\tau,n)},\nu_{(\tau,n)})$ with the Gram matrix 
$$
A_{(\tau,n)}=\left(
\begin{array}{cccc} 
3 & 1 & 0 \\
1& 3 & \tau  \\
0 & \tau & 2n
\end{array}
\right)
$$
and $\disc(M_{(\tau,n)})=16n-3\tau^2$. Let $(\tau,n), (\tau^\prime,n^\prime)\in\mathcal{S}_8$. Then the following statements hold:

\begin{enumerate}
\item $\CC_{M_{(\tau,n)}}$ is a non-empty irreducible divisor of $\CC_8$.
\item $\CC_{M_{(\tau,n)}}=\CC_{M_{(\tau^\prime,n^\prime)}}$ if and only if $(\tau,n)=(\tau^\prime,n^\prime)$.
\item $M_{(\tau,n)}$ is admissible if and only if one of the following conditions is true
\begin{enumerate}
\item $\tau=1,3$;
\item $\tau=0,2,4$ and $n$ is odd.
\end{enumerate}
\end{enumerate}
Moreover, if $M$ is a rank $3$ positive definite lattice with a distinguished element and $\emptyset \neq \CC_M\subset \CC_8$, then $\CC_M=\CC_{M_{(\tau,n)}}$ for some $(\tau,n)\in\mathcal{S}_8$.
\end{cor}

\begin{proof}
Clearly $\mathfrak{o}_{(\tau,n)}$ is a distinguished element of $M_{(\tau,n)}$. Since $M_{(\tau,n)}$ has no roots (cf. Remark \ref{rem:Bound}),  Theorem \ref{mainthm-one} implies (1).  If $\CC_{M_{(\tau,n)}}=\CC_{M_{(\tau^\prime,n^\prime)}}$, then $M_{(\tau,n)}\cong M_{(\tau^\prime,n^\prime)}$ and $\disc(M_{(\tau,n)})=\disc(M_{(\tau^\prime,n^\prime)})$, which implies $(\tau,n)=(\tau^\prime,n^\prime)$ (notice that $\disc(M_{(0,n)})=\disc(M_{(4,n+3)})$ but $M_{(0,n)}\ncong M_{(4,n+3 )}$). 
Thus, (2) holds. By Proposition \ref{prop:admMrank3-1},  we have (3). For the last claim, note that $M$ has a basis, say $(\mathfrak{o},\mu,\nu)$, with the Gram matrix of the same shape as $A_{(\tau,n)}$. Moreover, the numerical condition \eqref{eq:C8divisors} can be achieved by using a combination of the following two types of change of basis: $(\mathfrak{o}, \mu,\nu)\longmapsto (\mathfrak{o}, \mu,\nu+k (\mathfrak{o}-3\mu))$ and $(\mathfrak{o}, \mu,\nu)\longmapsto (\mathfrak{o}, \mu,-\nu)$.
\end{proof}

\begin{rem}\label{rem:C8rationalcubics}
The discriminant $\disc(M)$ of divisors $\CC_M$ in $\CC_8$ may be any positive integer $m$ with
$$
m\ge 16, \,m\equiv 4, 5, 13, 16\, ({\rm mod}\, 16).
$$
Hassett \cite{Has99} constructed countably infinitely many divisors in $\CC_8$ which parametrize rational cubic fourfolds. Those divisors correspond to admissible lattices $M_{(\tau,n)}$ with $(\tau,n)\in\mathcal{S}_8$, $\tau=1,3$ (see \cite[Lemma 4.4]{Has99}).  
\end{rem}

The classification for all divisors $\CC_M$ in $\CC_{18}$ can be proved by similar arguments.

\begin{cor}\label{cor:C18divisors}
Let $\mathcal{S}_{18}\subset \Z^2$ be the set of pairs $(\tau,n)$ of integers satisfying 
\begin{equation}
0\le \tau\le 3,\, n\ge 1.
\end{equation}
For each $(\tau,n)\in\mathcal{S}_{18}$, we denote by $M_{(1,\tau,n)}$ (resp. $M_{(2,\tau,n)}$) the rank $3$ positive definite lattice having a basis $(\mathfrak{o}_{(1,\tau,n)},\mu_{(1,\tau,n)},\nu_{(1,\tau,n)})$ (resp. $(\mathfrak{o}_{(2,\tau,n)},\mu_{(2,\tau,n)},\nu_{(2,\tau,n)})$) with the Gram matrix 
$$
A_{(1, \tau,n)}=\left(
\begin{array}{cccc} 
3 & 1 & 0 \\
1& 2n+1 & \tau  \\
0 & \tau & 6
\end{array}
\right) (\text{resp. } 
A_{(2, \tau,n)}=\left(
\begin{array}{cccc} 
3 & 0 & 0 \\
0& 2n+2 & \tau  \\
0 & \tau & 6
\end{array}
\right))
$$
and $\disc(M_{(1, \tau,n)})=36n-3\tau^2+12$ (resp. $\disc(M_{(2, \tau,n)})=36n-3\tau^2+36$). Let $(\tau,n), (\tau^\prime,n^\prime)\in\mathcal{S}_{18}$ and $i,i^\prime\in \{1,2\}$. Then the following statements hold:

\begin{enumerate}
\item $\CC_{M_{(i, \tau,n)}}$ is a non-empty irreducible divisor of $\CC_{18}$.
\item $\CC_{M_{(i,\tau,n)}}=\CC_{M_{(i^\prime,\tau^\prime,n^\prime)}}$ if and only if $(i,\tau,n)=(i^\prime,\tau^\prime,n^\prime)$.
\item $M_{(i,\tau,n)}$ is admissible if and only if one of the following conditions is true
\begin{enumerate}
\item $i=1$;
\item $i=2$ and $\tau=1,2$;
\item $i=2$, $\tau=0, 3$, and $3\mid n$.
\end{enumerate}
\end{enumerate}
Moreover, if $M$ is a rank $3$ positive definite lattice with a distinguished element and $\emptyset \neq \CC_M\subset \CC_{18}$, then $\CC_M=\CC_{M_{(i,\tau,n)}}$ for some $(i,\tau,n)$ with $(\tau,n)\in\mathcal{S}_{18}, i\in\{1,2\}$.
\end{cor}

\begin{rem}\label{rem:C18rationalcubics}
The discriminant $\disc(M)$ of divisors $\CC_M$ in $\CC_{18}$ may be any positive integer $m$ with
$$
m=21 \text{ or } m\ge 36, \,m\equiv 0, 9, 12, 21, 24, 33\, ({\rm mod}\, 36).
$$
In \cite{AHTVA19} it was shown that there is a countably infinite union of divisors in $\CC_{18}$ which parametrize rational cubic fourfolds. 
Those rational cubic fourfolds are contained in divisors  $\CC_{K_{a,b}}\subset \CC_{18}$ with the rank $3$ positive definite lattices
$$
K_{a,b}=\left(
\begin{array}{cccc} 
3 & 6 & a \\
6& 18 & 1  \\
a & 1 & b
\end{array}
\right),
$$
where $a\equiv b\,({\rm mod}\, 2)$, $a=-1,0,1$. By comparing discriminants (or change of basis), one see that $K_{-1,b}$, $K_{0,b}$, $K_{1,b}$ are isometric to admissible lattices $M_{(1,3,\frac{b-1}{2})}$, $M_{(2,1,\frac{b-2}{2})}$, $M_{(1,1,\frac{b-1}{2})}$ respectively.
\end{rem}

Laza \cite[Proposition 2.5, Remark 2.6]{Laz18} observed that if a cubic fourfold $X$ satisfies $\ell(A(X))=r(A(X))$ (resp. $\ell(A(X))<r(A(X))-1$), then the transcendental lattice of $X$ cannot (resp. can) be embedded into the lattice ${\rm II}_{3,19}(-1)$ primitively. 
Combining \cite[Proposition 2.9]{Laz18} and \cite{AT14} (see also \cite{Add16}), 
for a cubic fourfold  $X$ with $r(A(X))\ge 12$, the algebraic cohomology $A(X)$ is an admissible lattice.
Laza \cite{Laz18} asked whether $r(A(X))=11$ is achieved for some cubic fourfold $X$ with nonadmissible $A(X)$. 
We give the following affirmative answer:

\begin{cor}\label{cor:anwserLaza}
Let $n_i$ ($i=1, 2, ... ,10$) be positive integers with $n_i\ge 2$ for all $i$. Let $M$ be a rank $11$ positive definite lattice with a basis $\mathbf{b}:=(\mathfrak{o},\mu_1,\mu_2, ..., \mu_{10})$ such that the Gram matrix is equal to the following diagonal matrix 
$$
{\rm diag}(3,2n_1,2n_2, ..., 2n_{10}).
$$
Then there exists a cubic fourfold $X$ and an isometry $\phi: M\longrightarrow A(X)$ such that $\phi(\mathfrak{o})=h_{X}^{2}$. 
Moreover, the lattice $M$ is not admissible if one of the  following three statements is true:
\begin{enumerate}
\item $3\; |\; n_i$ for all $1\le i\le 10$.
\item $2 \;|\; n_i$ for all $1\le i\le 10$.
\item There exists an odd prime $p\equiv 2\;({\rm mod}\; 3)$ such that $p\; | \;n_i$ for all $1\le i\le 10$.
\end{enumerate}
\end{cor}

\begin{proof}
Clearly $\mathfrak{o}$ is a distinguished element and $M$ has no roots. 
Thus, by Corollary \ref{non-empty-CM-special}, there exists a cubic fourfold $X$ and an isometry 
$\phi: M\longrightarrow A(X)$ such that $\phi(\mathfrak{o})=h_{X}^{2}$. 
Note that the form
$$
f_{(M,\mathbf{b})}(x_1,...,x_{10})=6(n_1x_1^2+n_2x_2^2+...+n_{10}x_{10}^2)=\lambda g(x_1,...,x_{10}),
$$
where $\lambda={\rm gcd}(6n_1,6n_2,...,6n_{10})$ and $g(x_1,...,x_{10})$ is a primitive integral quadratic form.
As a result, it follows Theorem \ref{thm:admM2} that $M$ is nonadmissible.
\end{proof}

\begin{rem}
(i) For $X$ in Corollary \ref{cor:anwserLaza}, if the statement (1) holds, then $\ell(A(X))=r(A(X))$. 
On the other hand, if (1) does not hold but (2) or (3) holds, then $\ell(A(X))=r(A(X))-1$. 

(ii) For other examples, we consider the positive definite lattice $M$ of rank $3\leq n \leq 11$ with the Gram matrix $(a_{ij})$ such that $a_{ii}=3$ and $a_{ij}=1$ for $i\neq j$.
We observe that $M$ contains a distinguished element and $M$ has no roots. 
According to Corollary \ref{non-empty-CM-special}, 
there exists an $M$-polarizable cubic fourfold $X$ containing mutually intersecting $n-1$ planes with $M\cong A(X)$.
By elementary computations, we get $\lambda=4$ and $\ell(A(X))=r(A(X))-1$.
By Theorem \ref{thm:admM1}, 
$A(X)$ is nonadmissible; see \cite[Remark 7.4]{DM19} for $n=11$.

(iii) In \cite{Laz18} Laza introduced the {\it algebraic index} $\kappa_X:=\frac{\rho_X}{d_X}$ of a cubic fourfold $X$, where $\rho_X:=r(A(X))-1$ and $d_X:=\disc(A(X))$. He showed that $\kappa_X\leq 1$ for {\it potentially irrational} cubic fourfolds $X$ (equivalently, cubic fourfolds $X$ with nonadmissible $A(X)$). Moreover, based on \cite{LPZ18},  he showed that $\kappa_{X_0}=1$ for some potentially irrational cubic fourfolds with $\rho_{X_0}=6$ and $G(A(X_0))\cong (\Z/2\Z)^{\oplus 6}$. Now we consider $M:=A(X_0)\oplus D_4(2)$, where $D_4$ denotes the root lattice given by the corresponding Dynkin diagram (in particular, $r(D_4)=4$, $\disc(D_4)=4$). Then by Corollary \ref{non-empty-CM-special} and Theorem \ref{thm:admM1}, there exists a  potentially irrational cubic fourfold $X_1$ with $A(X_1)\cong M$. Moreover, $\ell(A(X_1))=\rho_{X_1}=10$ and $\kappa_{X_1}=\frac{2^{10}}{2^6\cdot 2^{6}}=\frac{1}{4}$. 
\end{rem}


\appendix

\section{Planes on Fermat cubic fourfold}

\subsection*{Basis and Gram matrix of $A(X_{F})$}

The $21$ planes $P_i$ $(i=1,...,21)$ on Fermat cubic fourfold $X_{F}$ in Proposition \ref{algcohom-Fermat} are:
\begin{equation}\label{21-planes}
\begin{array}{rclcrcl}
P_{1}&:=& \{e^{-\frac{2\pi i}{3}}z_{1}+z_{5}=e^{-\frac{2\pi i}{3}}z_{2}+z_{4}=e^{-\frac{2\pi i}{3}}z_{3}+z_{6}=0 \} \\
P_{2}&:=&\{e^{-\frac{2\pi i}{3}}z_{1}+z_{5}=z_{2}+z_{4}=e^{-\frac{2\pi i}{3}}z_{3}+z_{6}=0\} \\
P_{3}&:=& \{e^{-\frac{2\pi i}{3}}z_{1}+z_{4}=e^{-\frac{2\pi i}{3}}z_{2}+z_{6}=e^{\frac{2\pi i}{3}}z_{3}+z_{5}=0 \} \\
P_{4}&:=&\{e^{-\frac{2\pi i}{3}}z_{1}+z_{4}=e^{-\frac{2\pi i}{3}}z_{2}+z_{6}=z_{3}+z_{5}=0 \} \\
P_{5}&:=& \{e^{-\frac{2\pi i}{3}}z_{1}+z_{4}=e^{\frac{2\pi i}{3}}z_{2}+z_{4}=e^{-\frac{2\pi i}{3}}z_{3}+z_{6}=0 \} \\
P_{6}&:=& \{e^{\frac{2\pi i}{3}}z_{1}+z_{4}=z_{2}+z_{6}=z_{3}+z_{5}=0 \} \\
P_{7}&:=& \{e^{\frac{2\pi i}{3}}z_{1}+z_{4}=e^{-\frac{2\pi i}{3}}z_{2}+z_{5}=z_{3}+z_{6}=0 \} \\
P_{8}&:=& \{z_{1}+z_{4}=e^{-\frac{2\pi i}{3}}z_{2}+z_{5}=e^{\frac{2\pi i}{3}}z_{3}+z_{6}=0 \} \\
P_{9}&:=& \{e^{\frac{2\pi i}{3}}z_{1}+z_{4}=e^{\frac{2\pi i}{3}}z_{2}+z_{3}=e^{-\frac{2\pi i}{3}}z_{5}+z_{6}=0 \} \\
P_{10}&:=& \{z_{1}+z_{4}=e^{-\frac{2\pi i}{3}}z_{2}+z_{3}=e^{-\frac{2\pi i}{3}}z_{5}+z_{6}=0 \} \\
P_{11}&:=& \{e^{\frac{2\pi i}{3}}z_{1}+z_{3}=e^{-\frac{2\pi i}{3}}z_{2}+z_{6}=z_{4}+z_{5}=0 \} \\
P_{12}&:=& \{e^{-\frac{2\pi i}{3}}z_{1}+z_{3}=e^{\frac{2\pi i}{3}}z_{2}+z_{5}=e^{\frac{2\pi i}{3}}z_{4}+z_{6}=0 \} \\
P_{13}&:=& \{z_{1}+z_{3}=e^{\frac{2\pi i}{3}}z_{2}+z_{5}=e^{\frac{2\pi i}{3}}z_{4}+z_{6}=0 \} \\
P_{14}&:=& \{e^{-\frac{2\pi i}{3}}z_{1}+z_{3}=e^{-\frac{2\pi i}{3}}z_{2}+z_{4}=e^{-\frac{2\pi i}{3}}z_{5}+z_{6}=0 \} \\
P_{15}&:=& \{e^{-\frac{2\pi i}{3}}z_{1}+z_{3}=z_{2}+z_{4}=e^{\frac{2\pi i}{3}}z_{5}+z_{6}=0 \} \\
P_{16}&:=& \{e^{\frac{2\pi i}{3}}z_{1}+z_{3}=e^{-\frac{2\pi i}{3}}z_{2}+z_{4}=z_{5}+z_{6}=0 \} \\
P_{17}&:=& \{e^{-\frac{2\pi i}{3}}z_{1}+z_{2}=z_{3}+z_{6}=e^{-\frac{2\pi i}{3}}z_{4}+z_{5}=0 \} \\
P_{18}&:=& \{z_{1}+z_{2}=z_{3}+z_{6}=e^{-\frac{2\pi i}{3}}z_{4}+z_{5}=0 \} \\
P_{19}&:=& \{z_{1}+z_{2}=z_{3}+z_{6}=e^{\frac{2\pi i}{3}}z_{4}+z_{5}=0 \} \\
P_{20}&:=& \{z_{1}+z_{2}=e^{-\frac{2\pi i}{3}}z_{3}+z_{5}=e^{-\frac{2\pi i}{3}}z_{4}+z_{6}=0 \} \\
P_{21}&:=& \{z_{1}+z_{2}=e^{-\frac{2\pi i}{3}}z_{3}+z_{4}=e^{-\frac{2\pi i}{3}}z_{5}+z_{6}=0 \}.
\end{array}
\end{equation}

We denote by $\xi_{i}:=[P_{i}]\in A(X_{F})$ the cohomology class of the plane $P_{i}$.
According to Lemma \ref{intersect-planes},
the intersection matrix $((\xi_i.\xi_j))_{1\le i,j\le 21}$ is as follows:
\begin{equation}\label{intersect-matrix-21-planes}
\tiny{
\left( 
\begin{array}{ccccccccccccccccccccc} 
3 & -1 & 0 & 0 & 0 & 0 & 0 & 0 & 0 & 0 & 0 & 1 & 0 & -1 & 0 & 1 & 0 & 0 & 0 & 0 & 0 \\
-1 & 3 & 0 & 1 & 0 & 1 & 0 & 1 & 1 & 1 & 0 & 0 & 1 & 1 & 1 & 0 & 0 & 1 & 0 & 1 & 0 \\
0 & 0 & 3 & -1 & 1 & 0 & 0 & 1 & 0 & 1 & -1 & 0 & 0 & 1 & 1 & 1 & 0 & 0 & 1 & 0 & 0 \\
0 & 1 & -1 & 3 & 1 & 1 & 1 & 0 & 0 & 0 & 1 & 1 & 0 & 0 & 0 & 0 & 1 & 0 & 0 & 0 & 1 \\
0 & 0 & 1 & 1 & 3 & 0 & 0 & 1 & 0 & 0 & 0 & 0 & 0 & 0 & 0 & 0 & 0 & 0 & 1 & -1 & 1 \\
0 & 1 & 0 & 1 & 0 & 3 & 1 & 1 & -1 & 0 & 1 & 0 & 1 & 0 & 0 & 0 & 0 & 1 & 0 & 1 & 0 \\
0 & 0 & 0 & 1 & 0 & 1 & 3 & 1 & -1 & 0 & 1 & 1 & 0 & 0 & 0 & 1 & 1 & 1& -1 & 1 & 0 \\
0 & 1 & 1 & 0 & 1 & 1 & 1 & 3 & 0 & 1 & 0 & 0 & 1 & 0 & 1 & 0 & 0 &1 & 0 & 0 & 0 \\
0 & 1 & 0 & 0 & 0 & -1 & -1 & 0 & 3 & 1 & 0 & 0  & 1 & 1 & 1 & 0 & 0 & 0 & 1 & 0 & 1 \\
0 & 1 & 1 & 0 & 0 & 0 & 0 & 1 & 1 & 3 & 0 & 0 & 1 & 1 & 1 & 0 & 1 & 0 & 0 & 0 & 1\\
0 & 0 & -1 & 1 & 0 & 1 & 1 & 0 & 0 & 0 & 3 & 1 & 1 & 0 & 0 & -1 & 1 & 0 & 0 & 1 & 1 \\
1 & 0 & 0 & 1 & 0 & 0 & 1 & 0 & 0 & 0 & 1 & 3 & -1 & -1 & 1 & 0 & 1 & 0 & 0 & 1 & 0 \\
0 & 1 & 0 & 0 & 0 & 1 & 0 & 1 & 1 & 1 & 1 & -1 & 3 & 1 & 0 & 0 & 0 & 1 & 0 & 0 & 1 \\
-1 & 1 & 1 & 0 & 0 & 0 & 0 & 0 & 1 & 1 & 0 & -1 & 1 & 3 & 1 & 1 & 1 & 0 & 1 & 0 & 1 \\
0 & 1 & 1 & 0 & 0 & 0 & 0 & 1 & 1 & 1 & 0 & 1 & 0 & 1 & 3 & 0 & 1 & 0 & 1 & 1 & 0 \\
1 & 0 & 1 & 0 & 0 & 0 & 1 & 0 & 0 & 0 & -1 & 0 & 0 & 1 & 0 & 3 & 0 & 1 & 0 & 0 & 0 \\
0 & 0 & 0 & 1 & 0 & 0 & 1 & 0 & 0 & 1 & 1 & 1 & 0 & 1 & 1 & 0 & 3 & -1 & 1 & 0 & 1 \\
0 & 1 & 0 & 0 & 0 & 1 & 1 & 1 & 0 & 0 & 0 & 0 & 1 & 0 & 0 & 1 & -1 & 3 & -1 & 1 & -1 \\
0 & 0 & 1 & 0 & 1 & 0 & -1 & 0 & 1 & 0 & 0 & 0 & 0 & 1 & 1 & 0 & 1 & -1 & 3 & -1 & 1 \\
0 & 1 & 0 & 0 & -1 & 1 & 1 & 0 & 0 & 0 & 1 & 1 & 0 & 0 & 1 & 0 & 0 & 1 & -1 & 3 & -1 \\
0 & 0 & 0 & 1 & 1 & 0 & 0 & 0 & 1 & 1 & 1 & 0 & 1 & 1 & 0 &  0 & 1 & -1 & 1 & -1 & 3 
\end{array}
\right).
}
\end{equation}

\subsection*{Planes in the proof of Theorem \ref{mainthm-two}}
The $11$ planes $P^\prime_{ij}$ ($1\le i\le 5, 1\le j\le 2$), 
$P^\prime$ in the proof of Theorem \ref{mainthm-two} are:
\begin{equation}\label{5+1-planes}
\begin{array}{rclcrcl}
P'_{11}&:=& \{e^{\frac{2\pi i}{3}}z_{1}+z_{6}=e^{-\frac{2\pi i}{3}}z_{2}+z_{4}=e^{-\frac{2\pi i}{3}}z_{3}+z_{5}=0 \} \\
P'_{12}&:=&\{e^{\frac{2\pi i}{3}}z_{1}+z_{3}=e^{-\frac{2\pi i}{3}}z_{2}+z_{4}=z_{5}+z_{6}=0\}=P_{16} \\
P'_{21}&:=& \{e^{-\frac{2\pi i}{3}}z_{1}+z_{2}=e^{-\frac{2\pi i}{3}}z_{3}+z_{6}=e^{\frac{2\pi i}{3}}z_{4}+z_{5}=0 \} \\
P'_{22}&:=&\{e^{\frac{2\pi i}{3}}z_{1}+z_{2}=z_{3}+z_{5}=z_{4}+z_{6}=0 \} \\
P'_{31}&:=& \{e^{-\frac{2\pi i}{3}}z_{1}+z_{6}=z_{2}+z_{4}=z_{3}+z_{5}=0 \} \\
P'_{32}&:=& \{e^{-\frac{2\pi i}{3}}z_{1}+z_{5}=e^{\frac{2\pi i}{3}}z_{2}+z_{3}=e^{\frac{2\pi i}{3}}z_{4}+z_{6}=0 \} \\
P'_{41}&:=& \{z_{1}+z_{5}=e^{-\frac{2\pi i}{3}}z_{2}+z_{6}=e^{\frac{2\pi i}{3}}z_{3}+z_{4}=0 \}  \\
P'_{42}&:=& \{e^{\frac{2\pi i}{3}}z_{1}+z_{4}=e^{\frac{2\pi i}{3}}z_{2}+z_{5}=e^{\frac{2\pi i}{3}}z_{3}+z_{6}=0 \} \\
P'_{51}&:=& \{e^{-\frac{2\pi i}{3}}z_{1}+z_{5}=e^{-\frac{2\pi i}{3}}z_{2}+z_{4}=e^{\frac{2\pi i}{3}}z_{3}+z_{6}=0 \} \\
P'_{52}&:=& \{e^{-\frac{2\pi i}{3}}z_{1}+z_{2}=z_{3}+z_{5}=z_{4}+z_{6}=0 \} \\
P'&:=& \{z_{1}+z_{3}=z_{2}+z_{5}=e^{-\frac{2\pi i}{3}}z_{4}+z_{6}=0\}.
\end{array}
\end{equation}

Under the basis $(\xi_{1}, \xi_{2}, \cdots, \xi_{21})$ of $A(X_{F})$, 
the coordinates of the cohomology classes of $P^\prime_{ij}$, $P^\prime$ are expressed as follows:
\begin{equation}\label{11planes}
\begin{array}{rclcrcl}
{[P'_{11}]}&:=& (0,-1,-1,1,-1,-1,0,1,-1,1,-1,1,1,1,-1,0, -1,0,2,1,0) \\
{[P'_{12}]}&:=& (0,0,0,0,0,0,0,0,0,0,0,0,0,0,0,1,0,0,0,0,0) \\
{[P'_{21}]}&:=& (-2,-3,-4,2,0,-3,-1,3,-4,3,-3,3,3,2,-2,1,-3,-1,5,5,0) \\
{[P'_{22}]}&:=& (-1,-2,-3,2,0,-3,0,2,-3,2,-3,3,3,2,-2,0,-3,-1,4,4,0) \\
{[P'_{31}]}&:=& (-3,-6,-7,4,0,-4,-1,4,-6,5,-6,6,6,4,-3,1,-6,-2,8,8,-1) \\
{[P'_{32}]}&:=& (-1,-2,-2,2,0,-2,0,1,-2,2,-2,1,2,1,-1,0,-2,0,3,3,0) \\
{[P'_{41}]}&:=& (0,0,1,0,0,1,0,0,1,0,1,0,-1,0,0,0,0,0,-1,-1,0) \\
{[P'_{42}]}&:=& (1,3,4,-2,0,3,0,-2,4,-3,3,-3,-3,-3,2,0,4,1,-5,-4,1) \\
{[P'_{51}]}&:=& (1,2,3,-1,0,1,0,-1,2,-2,2,-2,-2,-1,1,-1,2,1,-3,-2,1) \\
{[P'_{52}]}&:=& (3,5,6,-3,0,4,1,-4,5,-4,5,-5,-5,-3,3,-1,5,2,-7,-7,1) \\
{[P']}:&=&(-3,-6,-8,3,1,-4,1,3,-6,6,-6,6,7,4,-3,1,-7,-3,9,8,-2)
\end{array}
\end{equation}


\end{document}